\documentclass{article}
\usepackage{graphicx}
\usepackage{enumitem}
\usepackage{float}
\usepackage{graphicx,subfigure}
\usepackage[super]{nth}
\usepackage{amsmath}
\usepackage{bm}
\usepackage[breaklinks,colorlinks,
  citecolor=blue,
  urlcolor=blue]{hyperref}
\usepackage{bigints}
\usepackage[dvipsnames]{xcolor}  %color
\usepackage{lineno} %line number
\usepackage{array}
\usepackage{amssymb}
\usepackage{mathtools}
\usepackage{multirow}
\usepackage{tabularray}
\usepackage{authblk}
\usepackage{xcolor}

\usepackage[utf8]{inputenc}
\usepackage{amsmath}
\usepackage{amsthm}
\usepackage{amssymb}
\usepackage{graphicx}
\usepackage{amsfonts}
\usepackage{latexsym}
\usepackage[title]{appendix}
\usepackage{color}
\newtheorem{theorem}{Theorem}[section]
\newtheorem{lemma}[theorem]{Lemma}
\newtheorem{proposition}[theorem]{Proposition}
\newtheorem{corollary}[theorem]{Corollary}
\newtheorem{definition}[theorem]{Definition}
\newtheorem{assumption}[theorem]{Assumption}
\newtheorem{example}[theorem]{Example}
\newtheorem{remark}[theorem]{Remark}
\usepackage{tikz}
\usepackage{titling}
\numberwithin{figure}{section}
\numberwithin{equation}{section}

\newcommand{\vertiii}[1]{{\left\vert\kern-0.25ex\left\vert\kern-0.25ex\left\vert #1 
    \right\vert\kern-0.25ex\right\vert\kern-0.25ex\right\vert}}

\newcommand{\bfi}{\bfseries\itshape}

\usepackage{titlesec}

\usepackage{lmodern}
\usepackage[absolute,overlay]{textpos}

\title{{\bf A global structure-preserving kernel method for the learning of Poisson systems}}

\author{Jianyu Hu$^{1}$, Juan-Pablo~Ortega$^{1}$, and Daiying Yin$^{1}$}

\begin{document}
% \addbox
\maketitle
%%%%%%%%%%%%%%%%%%%%%%%%%%%%%

\begin{abstract}
A structure-preserving kernel ridge regression method is presented that allows the recovery of globally defined, potentially high-dimensional, and nonlinear Hamiltonian functions on Poisson manifolds out of datasets made of noisy observations of Hamiltonian vector fields. The proposed method is based on finding the solution of a non-standard kernel ridge regression where the observed data is generated as the noisy image by a vector bundle map of the differential of the function that one is trying to estimate. Additionally, it is shown how a suitable regularization solves the intrinsic non-identifiability of the learning problem due to the degeneracy of the Poisson tensor and the presence of Casimir functions. A full error analysis is conducted that provides convergence rates using fixed and adaptive regularization parameters. The good performance of the proposed estimator is illustrated with several numerical experiments.
\end{abstract}

\makeatletter
\addtocounter{footnote}{1}{%
\footnotetext{Jianyu Hu, Juan-Pablo Ortega, and Daiying Yin are with the Division of Mathematical Sciences, School of Physical and Mathematical Sciences, Nanyang Technological University, Singapore. Their email addresses are {\texttt{Jianyu.Hu@ntu.edu.sg}}, {\texttt{Juan-Pablo.Ortega@ntu.edu.sg}}, and {\texttt{YIND0004@e.ntu.edu.sg}}, respectively.}
}
\makeatother

\tableofcontents

\section{Introduction}
The study of Poisson systems plays a crucial role in understanding a wide range of physical phenomena, from classical mechanics to fluid dynamics \cite{Marsden1994, Ortega2004, holm2009geometric, holm2011geometricparti, holm2011geometricpartii}. These systems, characterized by their inherent geometric structures and conservation laws, present unique challenges in modeling and simulation. Traditional methods for learning and approximating Poisson systems often require extensive domain knowledge and computational resources. In recent years, several machine learning techniques have emerged as powerful tools for addressing complex scientific problems that could offer new avenues for efficiently learning the structures underlying  Poisson systems.

Recent literature has extensively explored the learning and prediction of some Poisson systems. These studies incorporate varying degrees of domain-specific knowledge into learning algorithms, enhancing their ability to accurately model and predict such systems' behavior. An initial approach involved directly learning the Hamiltonian \cite{greydanus2019hamiltonian,david2023symplectic, hu2024structure} or Lagrangian \cite{cranmer2020lagrangian,Chen2022DatadrivenMT} functions using neural networks or other machine learning models, such as Gaussian processes \cite{offen2024machinelearningcontinuousdiscrete, pförtner2024physicsinformedgaussianprocessregression, 9992733}. It is worth pointing out that proper variational integrators \cite{marsden2001discrete, leok2012general} should be adopted to match the predicted vector field with observed time series data \cite{zhu2020deephamiltoniannetworksbased}. Researchers have also proposed learning the generating function that produces the flow map \cite{pmlr-v139-chen21r, 10.1063/5.0048129}. Alternative approaches aim to parameterize a (possibly universal) family of symplectic or Poisson maps to fit the flow map directly \cite{chen2019symplectic, JIN2020166,bajars2023locally,jin2022learning, ELDRED2024106162}, with SympNet being a prominent example \cite{JIN2020166}. Using SympNet, symplectic autoencoders have been developed to achieve a structure-preserving model reduction of Hamiltonian systems \cite{brantner2023symplecticautoencodersmodelreduction}, effectively mapping high-dimensional systems to low-dimensional representations. Additionally, physics-informed neural networks have become popular for solving known physical partial differential equations with high precision \cite{RAISSI2019686} and convergence analysis in certain scenarios \cite{doumèche2023convergenceerroranalysispinns, CiCP-28-2042}, though they enforce the PDEs as loss functions and thus do not preserve the exact structure. In general, to characterize and replicate the flow of (port-)Hamiltonian systems with exact structure preservation remains challenging, with a notable exception in the literature for linear port-Hamiltonian systems \cite{JMLR:v25:23-0450}.

This paper proposes a {\it non-standard kernel ridge regression method that allows the recovery of globally defined and potentially high-dimensional and nonlinear Hamiltonian functions on Poisson manifolds out of datasets made of noisy observations of Hamiltonian vector fields}. We now spell out three key features of the approach proposed in this paper. 

First, much of the existing research in this field operates under the assumption that the underlying dynamical system is Hamiltonian or that the flow map is symplectic \cite{chen2019symplectic, JIN2020166}, allowing for the use of generating function models or universal parameterizations of symplectic maps to directly model the flow. However, this assumption does not extend to general Poisson systems on manifolds for which the Hamiltonian function cannot be globally expressed in canonical coordinates, and the flow map is generally not symplectic. Moreover, due to the potential degeneracy of the Poisson tensor, Poisson manifolds typically possess a family of functions known as Casimirs, which do not influence the Hamiltonian vector fields. Preserving these Casimirs is an intrinsic requirement for any machine learning algorithm applied to such systems. Inspired by the Darboux-Weinstein theorem \cite{weinstein1983local, libermann2012symplectic, vaisman2012lectures}, recent approaches have proposed learning coordinate transformations to convert coordinates into canonical ones, followed by inverse transformations to retrieve the dynamics \cite{jin2022learning, bajars2023locally}. While the local existence of these transformations is guaranteed, their global behavior remains unclear, particularly in cases where coordinate patches overlap, and numerical errors on the preservation of Casimir values accumulate. The Lie-Poisson Network \cite{ELDRED2024106162, eldred2024clpnets} attempts to approximate flows of the special case of Lie-Poisson systems using compositions of flow maps of linear systems, preserving Casimirs exactly through projection but not addressing the expressiveness of the constructed flows. Rather than learning the flow map, we propose to discover an underlying Hamiltonian function of the Poisson system from Hamiltonian vector fields in a structure-preserving manner. The presence of Casimir functions renders the identification of the Hamiltonian function from the observed vector field infeasible, which is a challenge that lies beyond the scope of the work already presented by the authors in \cite{hu2024structure} and necessitates the development of new algorithms to address this issue.

Second, while significant research has focused on the forward problem of learning the flow map of dynamical systems using approaches like Physics-Informed neural networks (PINNs), symbolic regression \cite{Brunton2016}, Gaussian processes, or kernel methods, the inverse problem of uncovering unknown Hamiltonian functions from data remains relatively underexplored. Unlike neural networks or likelihood-based methods that rely on iterative or stochastic gradient descent to optimize objective functions, kernel methods offer a more straightforward, easy-to-train, and data-efficient alternative. By transforming the optimization problem into a convex one, kernel methods enable the derivation of a closed-form solution for the Hamiltonian function estimation problem in a single step, even on manifolds. Another advantage of kernel methods is that by selecting a kernel with inherent symmetry, the learning scheme can be designed to automatically preserve it, which is particularly important when learning mechanical systems. We recall that due to the celebrated Noether's Theorem, symmetries usually carry in their wake additional conserved quantities when dealing with mechanical systems.  The Reproducing Kernel Hilbert Space (RKHS) framework also facilitates rigorous estimation and approximation error bounds, offering both theoretical insights and interpretability. Additionally, it is well-suited for the design of online learning paradigms.

Finally, while the literature on learning physical dynamical systems is extensive and the idea of structure-preserving kernel regression has already been proposed to recover either interacting potentials of particle systems \cite{feng2021learning} or Hamiltonian functions on Euclidean spaces \cite{hu2024structure}, the majority of studies have been confined to the context of Euclidean spaces, with a lack of attention to scenarios where the underlying configuration space is a more general manifold, which is typically the case in many applications like rigid body mechanics, robotics, fluids, or elasticity. This generalization introduces a two-fold challenge in model construction or for the learning of either the flows of dynamical systems or their corresponding Hamiltonian functions, namely, preserving simlutaneously the manifold  and the variational structures. When a model is designed to predict Hamiltonian vector fields on a manifold, ensuring that the predicted flow remains on the manifold and preserves symplectic or Poisson structures necessitates structure-preserving integrators. However, even in Euclidean spaces, the design of variational integrators is complex and often requires implicit methods, indicating a heightened difficulty in the manifold setting. Additionally, standard methods such as Galerkin projections are typically approximately but not exactly manifold-preserving \cite{VBorovitskiy2020} and typically fail to preserve the variational structure. In \cite{maggioni2021learninginteractionkernelsagent}, an algorithm is proposed to identify interacting potentials on a Riemannian manifold, though the estimated potential remains dependent only on the Euclidean scalar distance. This approach requires a prior selection of a finite set of basis functions and coordinate charts or embeddings. In contrast, {\it our work introduces a kernel method that globally estimates a Hamiltonian function for Poisson systems on the phase-space manifold} and yields a {\it globally well-defined and chart-independent solution}. This sets it apart from existing approaches in the literature \cite{feng2021learning, maggioni2021learninginteractionkernelsagent, FeiLu2019}.

\paragraph{Main Results.}
This paper proposes a method to learn a Hamiltonian function $H: P\rightarrow \mathbb{R}$ of a globally defined Hamiltonian system on a $d$-dimensional Poisson manifold $P$ from noisy realizations of the corresponding Hamiltonian vector field $X _H \in \mathfrak{X}(P)$, that is:
\begin{align*}
\mathbf{X}_{\sigma^2}^{(n)}:= X_H(\mathbf{Z}^{(n)})+\bm{\varepsilon}^{(n)}, \quad n=1,\ldots,N,
\end{align*}
where $\mathbf{Z}^{(n)}$ are IID random variables with values on the Poisson manifold $P$ and with the same distribution $\mu_{\mathbf{Z}}$, and $\bm\varepsilon^{(n)}$ are independent random variables on the tangent spaces $T_{\mathbf{z}^{(n)}}P$ ($\mathbf{z}^{(n)}) $ is a particular realization of $\mathbf{Z}^{(n)}$ with mean zero and variance $\sigma^2I_d$. In this setup, we consider the inverse problem consisting in recovering the Hamiltonian function $H$, or a dynamically equivalent version of it, by solving the following optimization problem, 
\begin{align}\label{str-min}
 \widehat{h}_{\lambda,N}:=\mathop{\arg\min}\limits_{h\in\mathcal{H}_K} \frac{1}{N}\sum_{n=1}^{N} \left\|X_h(\mathbf{Z}^{(n)})-\mathbf{X}^{(n)}_{\sigma^2}\right\|_g^2 + \lambda\|h\|_{\mathcal{H}_K}^2,
\end{align}
where $X_h$ is the Hamiltonian vector field corresponding to the function $h \in {\mathcal H} _K$, ${\mathcal H} _K $ is the reproducing kernel Hilbert space (RKHS) associated to a kernel function $K:P \times P \rightarrow \mathbb{R} $ that will be introduced later on, $g  $ is a Riemannian metric on $P$, and $\lambda\geq0$ is a Tikhonov regularization parameter. The solution $\widehat{h}_{\lambda,N}$ in the RKHS is sought to minimize the natural regularized empirical risk defined in this setup.
We shall call the solution $\widehat{h}_{\lambda,N}$ of the optimization problem  \eqref{str-min} the {\bfi structure-preserving kernel estimator} of the Hamiltonian function.

We now summarize the outline and the main contributions of the paper.

\begin{enumerate}[leftmargin=*]
\item   In Section \ref{pre-lim}, we review basic concepts in Riemannian geometry, including differentials on Riemannian manifolds and the definitions of random variables and integration in this context. Additionally, we provide the necessary background on Poisson mechanics, covering topics such as Hamiltonian vector fields, Casimir functions, compatible structures, and Lie-Poisson systems, illustrated with two real-world examples: rigid-body dynamics and underwater vehicle dynamics. 
%We introduce the concept of a {\it compatible structure}, defined as a vector bundle map $J: TM\rightarrow TM$ induced by the Poisson tensor $B$ and the Riemannian metric $g$ on a manifold $M$. As a consequence, the Hamiltonian vector field of a function $h$ can be rewritten as $X_h=J\nabla h$.
%In the case of symplectic vector spaces equipped with the canonical symplectic form, this structure reduces to the canonical symplectic matrix 
%$J_{can}=\begin{bmatrix}
%    0&I\\-I&0
%\end{bmatrix}$. The fiber-wise linearity of this compatible structure $J$ plays a crucial role in our kernel ridge regression framework.

\item In Section \ref{Structure-preserving kernel regression on Poisson manifolds}, we extend the differentiable reproducing property, previously established for the compact \cite{zhou2008derivative} and Euclidean cases \cite{hu2024structure} to Riemannian manifolds. Specifically, we show that the differential of a function $f\in\mathcal{H}_K$ in the reproducing kernel Hilbert space (RKHS) $\mathcal{H}_K$ associated to a kernel function  $K:P \times P \rightarrow \mathbb{R}$, when paired with a vector $v\in T_xP$, can be represented as the RKHS inner product of $f\in \mathcal{H}_K$ with the differential of the kernel section of $K$ evaluated along $v\in T_xP$. This property is pivotal throughout our analysis. Following this, we derive the solution of the optimization problem \eqref{str-min} using operator representations and ultimately express it as a closed-form formula made of a linear combination of kernel sections, similar to the standard Representer theorem. Importantly, this formula provides a globally defined expression on the Poisson manifold $P$, independent of the choice of local coordinates, though we also present the formula in local charts. In Subsection \ref{conservation}, we study the case in which the Hamiltonian that needs to be learned is known to be {\it a priori} invariant with respect to the canonical action of a given Lie group and show that if the kernel is chosen to be invariant under that group action, then the flow of the estimated Hamiltonian inherently conserves momentum maps, as a direct consequence of Noether's theorem. 

\item In Section \ref{Estimation and approximation error analysis}, we utilize the operator representation of the estimator to derive bounds for estimation and approximation errors. These bounds demonstrate that, as the number of random samples increases, the estimator converges to a true Hamiltonian function with high probability.

\item   In Section \ref{numerical experiments}, we implement the structure-preserving kernel ridge regression in {\it Python} and conduct extensive numerical experiments. These experiments focus on two primary scenarios: Hamiltonian systems on symplectic manifolds (Subsection \ref{Hamiltonian systems on symplectic manifolds}) and Lie-Poisson systems (Subsection \ref{Applications to Lie-Poisson systems}). In subsection \ref{Hamiltonian systems on symplectic manifolds}, we aim to recover the Hamiltonian function for the classical two-vortex system on products of $2$-spheres, which exhibits singularities, as well as another well-behaved Hamiltonian function represented as a 3-norm defined on the same manifold. In Subsection \ref{Applications to Lie-Poisson systems}, we address symmetric Poisson systems on cotangent bundles of Lie groups. Through Lie-Poisson reduction, these systems can be modeled as Lie-Poisson systems on the dual of a Lie algebra. Numerical simulations validate the effectiveness, ease of training, and data efficiency of our proposed kernel-based approach.

\end{enumerate}

We emphasize that, although this paper focuses on recovering  Hamiltonian functions from observed Hamiltonian vector fields, the proposed framework is more broadly applicable. Specifically, it can be extended to any scenario where the observed quantity (in this paper $J\nabla H$) is the image by a vector bundle map of the differential of a function (in this paper $H$) on the manifold. This makes the framework versatile for various systems beyond Hamiltonian dynamics; for example, in gradient systems with an underlying control-dependent potential function, the same algorithm can be applied to learn an optimal control force based on sensor-measured external forces.

\section{Preliminaries on Poisson mechanics}\label{pre-lim}

\subsection{Hamiltonian systems on Poisson manifolds}
The aim of this paper is to learn Hamiltonian systems on Poisson manifolds. In this subsection, we will provide a brief introduction to Poisson mechanics relevant to the learning problems; for more details, see \cite{Marsden1994,Ortega2004, libermann2012symplectic, vaisman2012lectures}. Let $P$ be a manifold and denote by  $C^\infty(P)$ the space of all smooth functions defined on it.

\begin{definition}
\label{poisson-manifold}
A {\bfi Poisson bracket} (or a {\bfi Poisson structure}) on a manifold $P$ is a bilinear map $\{\cdot , \cdot \}: C^{\infty}(P)\times C^{\infty}(P)\to C^{\infty}(P)$ such that:
\begin{description}
\item [(i)]  $\left( C^\infty(P),\{\cdot , \cdot \}\right)$ is a Lie algebra.
\item [(ii)] $\{\cdot , \cdot \}$ is a derivation in each factor, that is,
$$
\{F G, H\}=\{F, H\} G+F\{G, H\},
$$
for all $F, G$, and $H \in  C^\infty(P)$.
\end{description}
A manifold $P$ endowed with a Poisson bracket $\{\cdot , \cdot \}$ on $ C^\infty(P)$ is called a {\bfi Poisson manifold}.    
\end{definition}

\begin{example}[{\bf Symplectic bracket}]
\label{symplectic-bracket} \normalfont
Any symplectic manifold $(P, \omega)$ is a Poisson manifold with the Poisson bracket defined by the symplectic form $\omega$ as follows:
\begin{align}\label{Symplectic-Bracket}
\{F,G\}(z)= \omega(X_F(z),X_G(z)), \quad \mbox{$F, G \in C^{\infty}(P)$,
}
\end{align}
where $X_F,X_G$ are the Hamiltonian vector fields of $F, G \in C^{\infty}(P)$, respectively,  uniquely determined by the equalities  $\mathbf{i}_{X_F} \omega= \mathbf{d}F $ and $\mathbf{i}_{X_G} \omega= \mathbf{d}G $ ($\mathbf{i}  $ denotes the interior product and $\mathbf{d}  $ the exterior derivative).
%It can be checked that the symplectic bracket \eqref{Symplectic-Bracket} satisfies Condition (i)  in Definition \ref{poisson-manifold}. And 
%Condition (ii) is satisfied as a consequence of the derivation property of vector fields:
%$$
%\{F G, H\}=X_H[F G]=F X_H[G]+G X_H[F]=F\{G, H\}+G\{F, H\} .
%$$
\end{example}

\begin{example}[{\bf Lie-Poisson bracket}]
\label{Lie-Poisson bracket}
\normalfont
If $\mathfrak{g}$ is a Lie algebra, then its dual algebra $\mathfrak{g}^*$ is a Poisson manifold with respect to the Lie-Poisson brackets $\{\cdot , \cdot \}_{+}$ and $\{\cdot , \cdot \}_{-}$ defined by
\begin{align}\label{lie-poisson-bracket}
\{F, G\}_{ \pm}(\mu)= \pm\left\langle\mu,\left[\frac{\delta F}{\delta \mu}, \frac{\delta G}{\delta \mu}\right]\right\rangle,    \quad \mbox{$F,G \in C^{\infty}(\mathfrak{g}^\ast)$,}
\end{align}
where $[\cdot,\cdot]$ is the Lie bracket on $\mathfrak g$, $\langle \cdot , \cdot \rangle$ denotes the pairing between $\mathfrak g^*$ and $\mathfrak{g}$, and $\frac{\delta F}{\delta \mu}\in\mathfrak{g}$ is the {\bfi functional derivative} of $F$ at $\mu\in \mathfrak{g}^*$ defined as the unique element that satisfies
\begin{align}
\label{functional derivative}
\lim_{\varepsilon\to0}\frac{1}{\varepsilon} \left[F(\mu+\varepsilon\delta\mu)-F(\mu)\right]=\left\langle\delta\mu,\frac{\delta F}{\delta\mu}\right\rangle\quad \mbox{for any $\delta\mu \in \mathfrak{g}^{\ast}$}. 
\end{align}
%The bilinearity and skew-symmetry of the Lie-Poisson Bracket \eqref{lie-poisson-bracket} are obvious. The derivation property of the bracket follows from the Leibniz rule for functional derivatives
%$$
%\frac{\delta(F G)}{\delta \mu}=F(\mu) \frac{\delta G}{\delta \mu}+\frac{\delta F}{\delta \mu} G(\mu) .
%$$
%The Jacobi identity for the Lie-Poisson bracket follows from the Jacobi identity for the Lie algebra bracket and the formula
%$$
%\begin{aligned}
%\pm \frac{\delta}{\delta \mu}\{F, G\}_{ \pm}= & {\left[\frac{\delta F}{\delta \mu}, \frac{\delta G}{\delta \mu}\right]-\mathbf{D}^2 F(\mu)\left(\operatorname{ad}_{\delta G / \delta \mu}^* \mu, \cdot\right) } +\mathbf{D}^2 G(\mu)\left(\operatorname{ad}_{\delta F / \delta \mu}^* \mu, \cdot\right),
%\end{aligned}
%$$
%where for each $\xi \in \mathfrak{g}$, $\operatorname{ad}_{\xi}: \mathfrak{g} \rightarrow \mathfrak{g}$ denotes the adjoint operator $\operatorname{ad}_{\xi}(\eta)=[\xi, \eta]$ and $\operatorname{ad}_{\xi}^*: \mathfrak{g}^* \rightarrow \mathfrak{g}^*$ is its dual.
\end{example}

By the derivation property of the Poisson bracket, the value of the bracket $\{F, G\}$ at $z\in P$ (and thus $X_F (z)$ as well) depends on $F$ only through the differential $\mathbf{d}F(z)$. Thus, a  Poisson tensor associated with the Poisson bracket can be defined as follows.
\begin{definition}
Let $(P, \left\{\cdot , \cdot \right\}) $ be a Poisson manifold. The {\bfi Poisson tensor} of this Poisson bracket is the contravariant anti-symmetric two-tensor
$$
B: T^* P \times T^* P \rightarrow \mathbb{R},
$$
defined by
\begin{align}\label{poisson-tensor}
B(z)\left(\alpha_z, \beta_z\right)=\{F, G\}(z),    \quad \mbox{for any $\alpha _z, \beta _z \in T_z ^\ast P$, $z \in P $},
\end{align}
and where $F, G \in C^{\infty} (P) $ are any two functions such that  $\mathbf{d} F(z)=\alpha_z$ and $\mathbf{d} G(z)=\beta_z \in T_z^* P$. 
\end{definition}

\noindent{\bf Hamiltonian vector fields and Hamilton's equations.}\quad Given a Poisson manifold $(P, \left\{\cdot , \cdot \right\})$ and $H\in C^\infty(P)$, there is a unique vector field $X_H \in \mathfrak{X} (P)$ such that 
\begin{align}\label{hamiltonian-vector-field}
X_H[F]=\{F,H\}, \quad \mbox{for all $F\in  C^\infty(P)$}.
\end{align}
The unique vector field $X_H$ is called the {\bfi Hamiltonian vector field} associated with the Hamiltonian function $H$. Let $F _t:P \rightarrow P$ be the flow of the Hamiltonian vector field $X_H$; then, for any $G\in C^\infty(P)$, 
\begin{align}\label{hamilton-equation}
   \frac{d}{dt}(G\circ F_t)=\{G,H\}\circ F _t=\{G\circ F _t,H\},\quad \text{or} \quad \dot{G}=\{G,H\}.
\end{align}
We call \eqref{hamilton-equation} the {\bfi Hamilton equations} associated with the Hamiltonian function $H$.
Let $B^{\sharp}:T^*P\to TP$ be the vector bundle map associated to the Poisson tensor $B$ defined in \eqref{poisson-tensor}, that is,
\begin{align}\label{bundle-map}
B(z)\left(\alpha_z, \beta_z\right)= \alpha_z\cdot B^{\sharp}(z)(\beta_z).
\end{align}
Then the Hamiltonian vector field defined in \eqref{hamiltonian-vector-field} can be expressed as $X_H=B^{\sharp}(\bm{\mathrm{d}}H)$. Furthermore, the Poisson bracket is given by
\begin{align*}
\{F,H\}(z)=B(z)(\mathbf{d} F(z),\mathbf{d} H(z))=\mathbf{d} F(z)\cdot B^{\sharp}(z)(\mathbf{d} H(z)).    
\end{align*}

Poisson geometry is closely related to symplectic geometry: for instance, every Poisson bracket determines a (generalized) foliation of a Poisson manifold into symplectic (immersed) submanifolds such that the leaf that passes through a given point has the dimension of the rank of the Poisson tensor at that point. We recall that the non-degeneracy hypothesis on symplectic forms implies that the Poisson tensor, in that case, has constant maximal rank. The following paragraph describes a distinctive feature of Poisson manifolds with respect to the symplectic case.

\noindent{\bf Casimir functions.} \quad As we just explained, an important difference between the symplectic and Poisson cases is that the Poisson tensor could be degenerate, which leads to the emergence of Casimir functions. 
A function $C\in C^\infty(P)$ is called a {\bfi Casimir function} if $\{C, F\}=0$ for all $F \in  C^\infty(P)$, that is, $C$ is constant along the flow of all Hamiltonian vector fields, or, equivalently, $X_C=0$, that is, $C$ generates trivial dynamics. Casimir functions are conserved quantities along the equations of motion associated with any Hamiltonian function. The Casimirs of a Poisson manifold are the {\bfi center} $\mathcal{C}(P)$ of the Poisson algebra $(C^{\infty}(P), \left\{\cdot , \cdot \right\})$ which is defined by
\begin{align*}
\mathcal{C}(P) = \{C\in C^\infty(P)\mid \{C,F\}=0, \text{ for all } F\in C^\infty(P)\}.  
\end{align*}
If the Poisson tensor is non-degenerate (symplectic case), then $\mathcal{C}(P)$ consists only of all constant functions on $P$. 
The following Lie-Poisson systems are two examples that admit non-constant Casimirs. 

\begin{example}[{\bf Rigid body Casimirs} {\cite{Marsden1994}}] \normalfont \label{rigid_body_example}
The motion of a rigid body follows a Lie-Poisson equation on the dual of the Lie algebra $\mathfrak{so}(3)^*$ of the special orthogonal group $SO(3) $ determined by the Hamiltonian function   
\begin{align*}
H(\Pi)=\frac{1}{2}\Pi^\top\mathbb{I}^{-1}\Pi,
\end{align*}
where $\mathbb{I}=\operatorname{diag}\{I_1,I_2,I_3\}$ is a diagonal matrix. The Poisson bracket on $\mathfrak{so}(3)^*\simeq \mathbb{R}^3$ is given by
\begin{align}
\label{liepoissonso3}
\{F,K\}(\Pi)=-\Pi\cdot (\nabla F\times\nabla K).    
\end{align}
Hence, it can be easily checked that $C(\Pi)=\|\Pi\|^2$ is a Casimir function of Poisson structure \eqref{liepoissonso3}. Furthermore, for any analytic function $\Phi:\mathbb{R}\to\mathbb{R}$, the function $\Phi\circ C$ is also a Casimir.
\end{example}

\begin{example}[{\bf Underwater vehicle Casimirs} \cite{leonard1997stability}] \label{Underwater vehicle}
\normalfont
The underwater vehicle dynamics can be described using the Lie-Poisson structure on $\mathfrak{so}(3)^*\times {\mathbb{R}^{3}}^*\times {\mathbb{R}^{3}}^*$ and the Hamiltonian function 
\begin{align*}
H(\Pi, \mathrm{Q}, \Gamma)=\frac{1}{2}\left(\Pi^{\top} A \Pi+2 \Pi^{\top} B^{\top} \mathrm{Q}+\mathrm{Q}^{\top} C \mathrm{Q}-2 m g\left(\Gamma \cdot r_{\mathrm{G}}\right)\right),   
\end{align*}
where $A,B,C$ are matrices coming from first principles and physical laws, $m$ is the mass of the body, $g$ is the gravity constant, and $r_G$ is the vector that links the center of buoyancy to the center of gravity.
For the underwater vehicle, $\Pi$ and $\mathrm{Q}$ correspond, respectively, to angular and linear components of the impulse of the system (roughly, the momentum of the system less a term at infinity). The vector $\Gamma$ describes the direction of gravity in body-fixed coordinates. The Poisson bracket on $\mathfrak{so}(3)^*\times {\mathbb{R}^{3}}^*\times {\mathbb{R}^{3}}^*$ is
$$
\{F, K\}(\Pi, \mathrm{Q}, \Gamma)=\nabla F^{\top} \Lambda(\Pi, \mathrm{Q}, \Gamma) \nabla K,
$$
where $F$ and $K$ are smooth functions on $\mathfrak{so}(3)^*\times {\mathbb{R}^{3}}^*\times {\mathbb{R}^{3}}^*$  and the Poisson tensor $\Lambda$ is given by
$$
\Lambda(\Pi, \mathrm{Q}, \Gamma)=\left(\begin{array}{ccc}
\hat{\Pi} & \hat{\mathrm{Q}} & \hat{\Gamma} \\
\hat{\mathrm{Q}} & 0 & 0 \\
\hat{\Gamma} & 0 & 0
\end{array}\right) .
$$
In this equation, the operation ${ }^{\wedge}: \mathbb{R}^3 \rightarrow \mathfrak{so}(3)$ is the standard isomorphism of $\mathbb{R}^3$ with the Lie algebra of the rotation group $SO(3)$ and is defined by $\hat{v} w=v \times w$ for $v, w \in \mathbb{R}^3$. There are six Casimir functions for the underwater vehicle dynamics
\begin{align*}
&C_1(\Pi, \mathrm{Q}, \Gamma)=\mathrm{Q} \cdot \Gamma, \quad C_2(\Pi, \mathrm{Q}, \Gamma)=\|\mathrm{Q}\|^2, \quad C_3(\Pi, \mathrm{Q}, \Gamma)=\|\Gamma\|^2 ,\\
&C_4(\Pi, \mathrm{Q}, \Gamma)=\Pi\cdot\mathrm{Q},\quad C_5(\Pi, \mathrm{Q}, \Gamma)=\Pi \cdot \Gamma, \quad C_6(\Pi, \mathrm{Q}, \Gamma)=\|\Pi\|^2.
\end{align*}
It can be easily checked that for any analytic function $\Phi:\mathbb{R}^6\to\mathbb{R}$, the function $\Phi\circ (C_1,\ldots,C_6)$ is also a Casimir.
\end{example}

\noindent{\bf The momentum map and conservation laws.} \quad Symmetries are particularly important in mechanics since they are, on many occasions, associated with the appearance of conservation laws. This fact is the celebrated {\bfi Noether's Theorem} \cite{NoetherTheorem}. A modern and elegant way to encode this fact uses the {\bfi momentum map}, an object introduced in \cite{lie1890, souriau1965geometrie, souriau, kostant66, souriau1966}. The definition of the momentum map requires a canonical Lie group action, that is, an action that respects the Poisson bracket. Its existence is guaranteed when the infinitesimal
generators of this action are Hamiltonian vector fields. In other words, if the Lie
algebra 
$\mathfrak{g}$ of the group $G$ that acts canonically  on the Poisson manifold $(P, \{ \cdot , \cdot \}) $
then, for each 
$\xi\in\mathfrak{g}$ with associated infinitesimal generator vector field $\xi_P \in \mathfrak{X}(P)$, we require the existence of a globally
defined function
$\mathbf{J}^\xi \in C^{\infty}(P)$, such that
\[\xi_M=X_{\mathbf{J}^\xi}.\]
\begin{definition}
Let $\mathfrak{g}$ be the Lie algebra of a Lie group $G$ acting canonically on the Poisson 
manifold $(P,\,\{\cdot,\cdot\})$. Suppose that for any $\xi\in\mathfrak{g}$,
the vector field $\xi_P$ is globally Hamiltonian, with
Hamiltonian function $\mathbf{J}^\xi \in C^{\infty}(P)$. The map
$\mathbf{J}:P\rightarrow\mathfrak{g}^\ast $ defined by the relation
\[\langle\mathbf{J}(z),\,\xi\rangle=\mathbf{J}^\xi (z),\]
for all $\xi\in\mathfrak{g}$ and $z\in M$, is called a 
{\bfi  momentum map\/} 
of the $G$--action. 
\end{definition}

Notice that  momentum maps are not uniquely determined; indeed, $\mathbf{J}_1$ and
$\mathbf{J}_2$ are momentum maps for the same canonical action if and only if for any $\xi\in\mathfrak{g}$, $\mathbf{J}_1^\xi -\mathbf{J}_2^\xi \in\mathcal{C}(P)$.
Obviously, if $P$ is symplectic and connected, then $\mathbf{J}$ is 
determined up to a constant in $\mathfrak{g}^\ast $. 

Noether's theorem is formulated in terms of the momentum map by saying that the fibers of the momentum map are preserved by the flow of the Hamiltonian vector field associated with any $G$-invariant Hamiltonian function.

\begin{example}[{\bf The linear momentum}]
Take the phase space of the $N$--particle
system, that is, $T^\ast  \mathbb{R}^{3 N}$. The
additive group $\mathbb{R}^3$ acts on it by spatial translation on
each factor: $\boldsymbol{v}\cdot
(\boldsymbol{q}_i,\,\boldsymbol{p}^i)=(\boldsymbol{q}_i+
\boldsymbol{v},\,\boldsymbol{p}^i)$,
with $i=1,\ldots,\,N$. This action is canonical and has an
associated momentum map that coincides with the
classical {\bfi linear momentum\/}
\index{momentum!linear}%
\index{linear!momentum}%
\[\begin{array}{cccc}
\mathbf{J}:&T^\ast  \mathbb{R}^{3 N}&\longrightarrow&
{\rm Lie}(\mathbb{R}^3)\simeq \mathbb{R}^3\\
	&(\boldsymbol{q}_i,\,\boldsymbol{p}^i)
&\longmapsto&\sum_{i=1}^{N}\boldsymbol{p}_i. 
\end{array}\]
\end{example}

\begin{example}[{\bf The angular momentum}]
Let ${\rm SO}(3)$ act on $\mathbb{R}^3$ and then, by lift, on
$T^\ast \mathbb{R}^3$, that is,
$A\cdot(\boldsymbol{q},\,\boldsymbol{p})=(A\boldsymbol{q},\,
A\boldsymbol{p})$.
This action is canonical and has an associated momentum map
\[\begin{array}{cccc}
\mathbf{J}:&T^\ast  \mathbb{R}^{3}&\longrightarrow&\mathfrak{so(3)}^\ast
\simeq \mathbb{R}^3\\
	&(\boldsymbol{q},\,\boldsymbol{p})&\longmapsto
&\boldsymbol{q}\times\boldsymbol{p}.
\end{array}\]
which is the classical {\bfi angular momentum\/}. 
\end{example}

\begin{example}[{\bf Lifted actions on cotangent bundles}]
\normalfont
The previous two examples are particular cases of the following situation. Let $G$ be a Lie group acting on the manifold $Q$ and then by lift on its cotangent bundle $T^\ast  Q$. Any such lifted action is canonical with respect to the canonical symplectic form on $T^\ast  Q$ and has an associated momentum map $\mathbf{J}:T^\ast  Q\rightarrow\mathfrak{g}^\ast $ given by
\begin{equation}
\label{lifted actions on cotangent bundles momentum map}
\langle\mathbf{J}(\alpha_q),\,\xi\rangle=\langle\alpha_q,\,\xi_Q(q)\rangle,
\end{equation}
for any $\alpha_q\in T^\ast  Q$ and any $\xi\in \mathfrak{g}$. 
\end{example}

\subsection{High-order differentials and random variables on Riemannian manifolds} 

We start by reviewing higher-order differentials and the spaces of bounded differentiable functions on general manifolds, which will allow us to examine the regularity of functions in a reproducing kernel Hilbert space (RKHS) later on in Section \ref{RKHS on Riemannian manifolds}. Finally, we introduce some basic facts about statistics on Riemannian manifolds that will be needed later on. For further details, see \cite{pennec2019riemannian, hsu2002stochastic, emery2012stochastic} and references therein.

Let $M$ be a smooth manifold with an atlas $\{U_\beta,\varphi_\beta,\beta\in\mathcal{I}\}$, where $\mathcal{I}$ is an index set. A function $f:M\to \mathbb{R}$ is said to be in the $C^k$ class if for any chart $\varphi_\beta, \beta\in\mathcal{I}$, the function $f\circ\varphi_{\beta}:\varphi_{\beta}(U_\beta)\to \mathbb{R}$ is a $C^k$ function and the $k$-order derivative of $f\circ\varphi_{\beta}$ is continuous.

\begin{definition}[{\bf Differentials}]\label{differetial}
If $f:M\to\mathbb{R}$ is in the $C^1$ class, we define its {\bfi differential} $Df:M\to T^*M$ of $f$ as
\begin{align*}
Df(x)\cdot v:=\left.\frac{d}{dt}\right|_{t=0}f(c(t)),\quad \mbox{for all $x\in M, v\in T_xM$},  
\end{align*}
where $c:[-1,1]\to M$ is a differentiable curve with $c(0)=x$ and $c'(0)=v$.  
\end{definition}
In the sequel, we shall also denote $\mathbf{d}f$ as the differential of $f$. The differential $Df$ can be identified with a function on $TM$. Then, if $f$ is regular enough, we can continue taking the differential of $Df$ on $TM$. In this way, we can define high-order differentials as follows. 

\begin{definition}[{\bf High-order differentials}]
If $f:M\to\mathbb{R}$ is in $C^k$ class, we define the {\bfi $k$-order differential} of $f$ denoted as $D^kf:T^kM\to \mathbb{R}$ inductively to be the differential of $D^{k-1}f:T^{k-1}M\to \mathbb{R}$.  
\end{definition}
\begin{remark}
\normalfont
For a differentiable function $K:M\times M\to \mathbb{R}$, we denote $D^{(m,l)}K:T^mM\times T^lM\to \mathbb{R}$ as the $m$-order differential with respect to the first variable and $l$-order differential with respect to the second variable.  
\end{remark}

\begin{remark}[{\bf Functions with bounded differentials}]
\label{Functions with bounded differentials}
\normalfont
Having defined high-order differentials, we can now define the spaces $C_b^s(M)$ of {\bfi functions with bounded differentials}  for all positive integers $s$. We assume first that $M$ is a second-countable manifold, and we hence can construct a complete Riemannian metric $g$ on it \cite{nomizu1961existence}. A general result \cite[Proposition 12.3]{tu2008manifolds} guarantees that the higher order tangent $T ^kM $ bundles are also second countable and can hence also be endowed with complete metrics $g _k$
Let now $C^s_b(M)$ be the subset of $C^s(M)$ given by
\begin{align*}
C_b^s(M):=\left\{f\in C^s(M)~\mid ~\|f\|_{C_b^s}:=\|f\|_{\infty}+\sum_{k=1}^s\|D^kf\|_{\infty}<\infty\right\},    
\end{align*}
where $\|f\|_{\infty}:=\sup_{x\in M}|f(x)|$ and
\begin{align*}
\|D^kf\|_{\infty}:=\sup_{y\in T^{k-1}M}\  \sup_{v\in T_y(T^{k-1}M), v \neq 0}\frac{\left|D^kf(y)\cdot v\right|}{\|v\|_{k-1}}.  
\end{align*}
Here, $\|v\|_{k-1}$ stands for the norm of $v$ in the tangent space $T_{y}T^{k-1}M$ that is defined using the complete Riemannian metric $g_{k-1}$ on $T^{k-1}M$. 
\end{remark}

The metric $g=g _0$ is also used to define the {\bfi gradient} $\nabla f\in \mathfrak{X}(M)$ of a function $f:M\to\mathbb{R}$ which is the vector field defined by its differential as
\begin{align*}
g(x)(\nabla f(x), v)=Df(x)\cdot v, \quad \text{for\ all\ } v\in T_xM.   
\end{align*}
For a differentiable function $f:M\to\mathbb{R}$, a Fundamental Theorem of Calculus can be formulated as
\begin{equation}\label{fun-the}
\begin{aligned}
f(x)-f(y)&=\int_a^b\frac{d}{dt} f(\gamma(t))\mathrm{d}t= \int_a^b g(\gamma(t))\left(\nabla f(\gamma(t)),\gamma'(t)\right)\mathrm{d}t=\int_a^bDf(\gamma(t))\cdot\gamma'(t)\mathrm{d}t,   
\end{aligned}
\end{equation}
where $\gamma:[a,b]\to M$ is a differentiable curve with $\gamma(a)=x$ and $\gamma(b)=y$.
Taking absolute values and using the Cauchy-Schwarz inequality, we obtain
\begin{align*}
|f(x)-f(y)| \leq \|Df\|_{\infty}\int_a^b\|\gamma'(t)\|_g\mathrm{d}t=\|Df\|_{\infty}\ \text{length}(\gamma),    
\end{align*}
for each $f\in C^1_b(M)$. Note that, if we choose $\gamma$ to be the geodesic connecting $x$ and $y$, then $\text{length}(\gamma)$ is the geodesic distance of $x$ and $y$.

\noindent{\bf Random variables on Riemannian manifolds.}\quad
Let $(M,g)$ be a smooth, connected, and complete $d$-dimensional Riemannian manifold with Riemannian metric $g$. Existence and uniqueness theorems from the theory of ordinary differential equations guarantee that there exists a unique geodesic integral curve $\gamma_{x,v}$ going through the point $x\in M$ at time $t=0$ with tangent vector $v\in T_xM$.
The Riemannian {\bfi exponential map} at $x\in M$ denoted by $\operatorname{exp}_{x}$ maps each vector $v\in T_xM$ to the point of the manifold reached in a unit time, that is, 
\begin{align*}
\operatorname{exp}_{x}(v):=\gamma_{x,v}(1),   
\end{align*}
where $\gamma_{x,v}$ is the geodesic starting at $x$ along the vector $v$. Let $t_0(x,v)>0$ be the first time when the geodesic $t\mapsto \operatorname{exp}_x(tv)$ stops to be length-minimizing. The point $\operatorname{exp}_x(t_0(x,v)v)$ is called a {\bfi cut point} and the corresponding tangent vector $t_0(x,v)v$ is called a {\bfi tangent cut point}. The set of all cut points of all geodesics starting from $x\in M$ is the {\bfi cut locus} $C(x)\in  M$ and the set of corresponding vectors the {\bfi tangential cut locus} $\mathcal{C}(x)\in T_xM$. Thus, we have $C(x) = \operatorname{exp}_x(\mathcal{C}(x))$. Define the maximal definition domain for the exponential chart as
\begin{align*}
\mathcal{D}(x):=\left\{tv\in T_xM\mid v\in S^{d-1}, 0\leq t<t_0(x,v)\right\},   
\end{align*}
where $S^{d-1}$ is the unit sphere in the tangent space with respect to the metric $g(x)$. Then the exponential map $\operatorname{exp}_x$ is a diffeomorphism from $\mathcal{D}(x)$ on its image $\operatorname{exp}_x(\mathcal{D}(x))$.
Denote by $\operatorname{log}_x$ the inverse of the exponential map. The Riemannian distance $d(x,y)=\|\operatorname{log}_x(y)\|_{g(x)}$.
Recall that for connected complete Riemannian manifolds we have that
\begin{align*}
M= \operatorname{exp}_x(\mathcal{D}(x)) \cup \mathcal{C}(x).
\end{align*}
It is easy to check that $\mathcal{D}(x)$ is an open star-shaped subset of $T_{x} M$, so $\left(\mathcal{D}(x), \exp _{x}\right)$ is a smooth chart and one has that for any Borel measurable function $f:M \rightarrow \mathbb{R}$ on $M$
$$
\int_{M} f d V_{M}=\int_{\mathcal{D}(x)}\left(f \circ \exp _{x}\right) \sqrt{\operatorname{det}\left(G_{x}\right)} d \lambda^d,
$$
where $d V_{M}$ is the Riemannian measure on the manifold $(M, g), G_{x}:=\exp _{x}^* g$ is the pullback of $g$ by $\exp _{x}$, and $\lambda^d$ denotes the $d$-dimensional Lebesgue measure on $\mathbb{R}^d$.

\begin{definition}[{\bf Random variables on Riemannian manifold}] Let $(\Omega,\mathcal{F},\mathbb{P})$ be a probability space and $(M,g)$ be a Riemannian manifold. A {\bfi random variable} $\xi$ on the Riemannian manifold $M$ is a Borel measurable function from $\Omega$ to $M$. 
\end{definition}
Let $\mathcal{B}(M)$ be the Borel $\sigma$-algebra of $M$. The random variable $\xi$ on the Riemannian manifold $(M,g)$ has a probability density function $p:M \rightarrow \mathbb{R}$ (real, positive and integrable function) if for all $A \in \mathcal{B}(M)$,
$$
\mathbb{P}(X \in A)=\int_{A} p(x) dV_M(x) \quad \text { and } \quad \mathbb{P}(M)=\int_{M} p(x) dV_M(x)=1.
$$
The definitions of the mean and variance of a random variable on a Riemannian manifold can be found in the literature \cite{pennec2006intrinsic}.

\begin{remark}[{\bf Random variables on tangent spaces}]\label{Random variables on tangent spaces}
\normalfont Since the tangent space of a manifold at any given point is a vector space, the definition of a random variable on them is the same as for vector spaces. Let $\eta$ be a random variable on the tangent space $T_xM$ with density $p:T _xM \rightarrow \mathbb{R}$, that is, $\eta$ is a Borel
measurable function from $\Omega$ to $T_xM$. 

\noindent {\bf (i) (Mean and variance)}. The mean and variance of a random variable $V$ on $T _xM $ are defined as
\begin{align*}
\mathbb{E}[\eta]:&= \int_{T_xM} vp(v)d\lambda^d(v) \in T_xM,\\
\operatorname{Var}[\eta]:&=\mathbb{E}\left[\|\eta-\mathbb{E}[\eta]\|_g^2\right]=\int_{T_xM}\left\|v-\mathbb{E}[\eta]\right\|_g^2p(v)d\lambda^d(v),
\end{align*}
where $\lambda^d$ is the $d$-dimensional Lebesgue measure.

\noindent {\bf (ii) (Independence).} Let $\eta_x$ and $\eta_y$ are two random variables on the tangent spaces $T_xM$ and $T_yM$, respectively. Let  $\mathcal{B}(T_xM)$ and $\mathcal{B}(T_yM)$ be the Borel $\sigma$-algebras of $T_xM$ and $T_yM$, respectively. Then we say $\eta_x$ and $\eta_y$ are independent if for all $A_x\in \mathcal{B}(T_xM)$ and  $A_y\in \mathcal{B}(T_yM)$, we have that 
\begin{align*}
\mathbb{P}(\eta_x\in A_x, \eta_y\in A_y)= \mathbb{P}(\eta_x\in A_x)\mathbb{P} (\eta_y\in A_y). 
\end{align*}
\end{remark}

\subsection{Compatible structures and local coordinate representations}

This subsection introduces the concept of compatible structure, a vector bundle map from the tangent space of the Poisson manifold to itself that enables us to express Hamiltonian vector fields as images of gradients by a vector bundle map. Using this tool, we can formulate adapted differential reproducing properties in RKHS setups and establish an operator framework later in Section \ref{Structure-preserving kernel regression on Poisson manifolds} for them. Finally, we shall apply the Poisson-Darboux-Lie-Weinstein theorem to provide local coordinate representations of the compatible structure.

We first endow the Poisson manifold $(P, \left\{\cdot , \cdot \right\})$ with a Riemannian metric $g$ and define the vector bundle map $J:TP\to TP$  by
\begin{align}\label{Com-str}
J(z) v: =B^\sharp(z)\left(g^\flat(z)(v)\right), \quad \mbox{for all $z\in P, v\in T_zP$,}  
\end{align}
where $B^\sharp:T^*P\to TP$ is the vector bundle map defined in \eqref{bundle-map}
and $g^\flat:TP\to T^*P$ is the vector bundle isomorphism determined by the Riemannian metric $g$. Notice that the vector bundle map $B^\sharp$ may not be an isomorphism due to the possible degeneracy of the Poisson tensor $B$, but nevertheless, the map $J$ is always well-defined since we do not need to use the inverse of $B^\sharp$ at any moment. Note that the Poisson tensor $B$ and the Riemannian metric $g$ are linked by the following relation:
\begin{align}\label{com-str}
B(z)(\alpha_z,\beta_z)=\alpha_z\cdot B^\sharp(z)(\beta_z)=g(z)\left(g^\sharp(z)(\alpha_z),B^\sharp(z)(\beta_z)\right)
,\quad \mbox{for all  $ \alpha_z,\beta_z\in T^*_zP$.}    
\end{align}
We call $J:TP\to TP$ the {\bfi compatible structure} associated with the Poisson tensor $B$ and the Riemannian metric $g$. Using the compatible structure $J$, the Hamiltonian vector field associated with a Hamiltonian function $h$ can be expressed as 
\begin{align}\label{rep-ham-vector-field}
X_h=B^{\sharp}(\mathbf{d}h)=B^{\sharp}(g^{\flat}(\nabla h))=J\nabla h.    
\end{align}
\begin{remark}
\normalfont
If the Poisson manifold $P$ is symplectic with a symplectic form $\omega$, then $B^{\sharp}=\omega^\sharp$, and the relationship \eqref{com-str} reduces to 
\begin{align*}
g(z)(u,v)=\omega(z)(u,-Jv),\quad u,v\in T_zP,   
\end{align*}
where $J(z)v:=\omega^\sharp(z)(g^\flat(z)(v))$ is, in this case, an isomorphism from $TP$ to $TP$.
When the compatible structure $J$ defined in \eqref{com-str} satisfies $J^2=-I$, it is called a complex structure, and the corresponding manifold is said to be a K\"{a}hler manifold. Even-dimensional Euclidean vector spaces are special cases of K\"{a}hler manifolds by choosing $J$ as the canonical (Darboux) symplectic matrix $J_{\mathrm{can}}$, whose Hamiltonian vector field associated with a Hamiltonian function is expressed as $X_h=J_{\mathrm{can}}\nabla h$.
\end{remark}

Using the compatible structure $J$, we obtain a useful property of Casimir functions which will be needed later on in Section \ref{Structure-preserving kernel regression on Poisson manifolds} to prove the uniqueness of the solution to the learning problem.

\begin{proposition}\label{Casimir property}
Let $(P, \left\{\cdot , \cdot \right\})$ be a Poisson manifold equipped with a Riemannian manifold $g$. Then for any vector field $Y\in\mathfrak{X}(P)$ and any  $h \in C^{\infty}(P)$:
\begin{align}\label{change J}
g(z)(J(z)\nabla h(z), Y(z))=g(z)(\nabla h(z), -J(z)Y(z)),\quad \mbox{for all  $z\in P$}.
\end{align}
Furthermore, Casimir functions are constant along the flow of the vector field $-JY$, that is, $(-JY)[h] = 0$ for any vector field $Y\in\mathfrak{X}(P)$ and any $h\in\mathcal{C}(P)$.
\end{proposition}

\begin{proof}
By the definition of the compatible structure $J$ and the vector bundle map $B^\sharp$, we have that
\begin{align*}
g(z)(J(z)\nabla h(z),Y(z))&=g(z)\left(B^\sharp(z)\mathbf{d}h(z),Y(z)\right)=B(z)(g^\flat(z)(Y)(z),\mathbf{d}h(z))\\
&=-B(z)(\mathbf{d}h(z),g^\flat(z)(Y)(z)) =-\mathbf{d}h(z)\cdot B^\sharp(z)(g^\flat(z)(Y))(z)\\
&=-g(z)(\nabla h(z), B^\sharp(z)(g^\flat(z)(Y))(z))=g(z)(\nabla h(z), -J(z)Y(z)),
\end{align*}
where the second and fifth equalities are due to the equation \eqref{com-str} and the symmetry of the Riemannian metric $g$, and the third equality holds thanks to the anti-symmetry of the Poisson tensor $B$. Furthermore, let $h\in\mathcal{C}(P)$. Then equation \eqref{change J} implies that for all $ z\in P$,  
\begin{align*}
(-JY)[h](z)=\mathbf{d}h(z)\cdot (-JY)(z) = g(z)(\nabla h(z),-J(z)Y(z)) = g(z)(X_h(z), Y(z)) =0.
\end{align*}
\end{proof}

\noindent{\bf Local coordinate representations.} \quad  We start by recalling the local Lie-Darboux-Weinstein coordinates of Poisson manifolds (see the original paper \cite{weinstein1983local} or \cite{libermann2012symplectic, vaisman2012lectures} for details). 

\begin{theorem}[{\bf Lie-Darboux-Weinstein coordinates}] 
Let $(P, \{ \cdot , \cdot \})$ be a $m$--dimensional Poisson manifold and $z_0 \in M $
a point where the rank of the Poisson structure equals $2n $, $0 \leq 2 n \leq
m $. There exists a chart  $(U, \varphi)$ of $P$ whose domain contains the point $z_0$
and such that the associated local coordinates, denoted by $(q ^1,\ldots, q ^n, p _1,
\ldots, p _n,y _1, \ldots, y _{m-2n})$, satisfy
\[
\{q ^i, q ^j\}=\{ p _i, p _j\}=\{q ^i, y _k\}=\{p _i, y _k\}=0,\quad\text{ and } 
\quad\{q
^i, p _j\}= \delta^i_j,
\] 
for all $i,j,k $, $1 \leq i,j\leq n $, $1\leq k\leq m-2n $.  For all $k,l $, $1\leq k,l\leq m-2n $, the Poisson bracket $\{y _k, y _l\}
$ is a function of the local coordinates $y
^1, \ldots, y ^{m-2n} $ exclusively, and vanishes at $z _0$. Hence, the restriction of
the bracket $\{
\cdot , \cdot \} $ to the coordinates $y
^1, \ldots, y ^{m-2n} $ induces a Poisson structure that is usually referred to as the
{\bfi  transverse Poisson structure} 
of $(P, \{ \cdot , \cdot \})$ at $m$. This structure is unique up to isomorphisms.

If the rank of $(P, \{ \cdot , \cdot \})$ is constant and equal to $2n $ on a
neighborhood $W$ of $z _0$ then, by choosing the domain $U$ of the chart contained in
$W$, the coordinates $y$ satisfy
\[
\{y _k, y _l\}=0,
\]
for all $k,l $, $1\leq k,l\leq m-2n $.
\end{theorem}

This theorem proves that if the rank of the Poisson structure around a given point $z _0 \in P $ is locally constant, then there exist local coordinates $\{z^1,\ldots,z^m\}$ in which the Poisson tensor $B$ can be represented by a constant matrix $\left(B^{ij}\right)_{i,j=1}^m$ as $B=B^{ij}\frac{\partial}{\partial z^i}\otimes\frac{\partial}{\partial z^j}$. Hence, the Poisson bracket can be locally expressed as (using the Einstein convention in the summation of the indices)
\begin{align*}
\{F,G\}=B^{ij}\frac{\partial F}{\partial z^i}\frac{\partial G}{\partial z^j}.
\end{align*}
Furthermore, for any one-form $\alpha=\alpha_jdz^j$, the vector bundle map $B^\sharp$ can be locally expressed as
\begin{align}\label{B-sharp}
B^\sharp(\alpha)= B^{ij}\alpha_i\frac{\partial}{\partial z^j}.  
\end{align}
Note that $g^\flat:TP\to T^*P$ is the bundle isomorphism with respect to the Riemannian metric $g:=g_{ij}dz^i\otimes dz^j$. Thus, for any $X:=X^j\frac{\partial}{\partial z^j}$, the bundle isomorphism $g^\flat$ can be expressed as 
\begin{align}\label{g-flat}
g^\flat (X) = g_{ij}X^idz^j.    
\end{align}
Therefore, combining equations \eqref{B-sharp} and \eqref{g-flat}, the compatible structure $J$ defined in \eqref{Com-str} has the local representation
\begin{align*}
J\left(X^j\frac{\partial}{\partial z^j}\right)= B^\sharp\left(g^\flat\left(X^j\frac{\partial}{\partial z^j}\right)\right)=B^\sharp\left(g_{ij}X^idz^j)\right)= B^{ij}g_{ik}X^k\frac{\partial}{\partial z^j},    
\end{align*}
which proves that the matrix representation of the bundle map is given by $J=-Bg$, that is, 
\begin{align}\label{local-com-str}
    J^{ij} = -B^{ik}g_{kj}.
\end{align}

\section{Structure-preserving kernel regressions on Poisson manifolds}
\label{Structure-preserving kernel regression on Poisson manifolds}

This section proposes a structure-preserving kernel regression method to recover Hamiltonian functions on Poisson manifolds, which is a significant generalization of the same learning problem on even-dimensional symplectic vector spaces \cite{hu2024structure}. First, in Section \ref{RKHS on Riemannian manifolds}, we establish a differential reproducing property, which not only ensures the existence and boundedness of function gradients in the RKHS but also provides an operator framework to represent the minimizers of our structure-preserving kernel regression problems.
In Section \ref{Operator representations of the kernel estimator}, we present an operator framework to represent the estimator for the structure-preserving regression problem. In Section \ref{A global solution to the learning problem}, we develop a differential kernel representation of the estimator, providing a closed-form solution that is computationally feasible in applications. The error analysis of this structure-preserving estimator is conducted later in Section \ref{Estimation and approximation error analysis}, where we establish convergence rates with both fixed and adaptive regularization parameters.

\subsection{RKHS on Riemannian manifolds}\label{RKHS on Riemannian manifolds}

The structure-preserving learning of Hamiltonian systems on symplectic vector spaces has been tackled by proving a differential reproducing property in \cite{hu2024structure}. 
We shall show in this section that this differential reproducing property still holds on a generic Riemannian manifold, and we will give a {\it global} version of it, where the term global means that the result that we provide is not formulated in terms of locally defined partial derivatives but of globally defined differentials as in \eqref{differetial} that do not depend on specific choices of local coordinates.

We first quickly recall Mercer kernels and the associated reproducing kernel Hilbert spaces (RKHS). A {\bfi Mercer kernel} on a non-empty set $\mathcal{X}$ is a positive semidefinite symmetric function $K:\mathcal{X}\times \mathcal{X}\to \mathbb{R}$, where positive semidefinite means that it satisfies 
\begin{align}\label{pos-def}
\sum_{i=1}^n\sum_{j=1}^n c_ic_jK(x_i,x_j)\geq 0,    
\end{align}
for any $x_1,\ldots,x_n\in \mathcal{X}$, $c_1,\ldots,c_n\in\mathbb{R}$, and any $n \in \mathbb{N}$. Property \eqref{pos-def} is equivalent to requiring that the Gram matrices $[K(x_i,x_j)]_{i,j=1}^n$ are positive semidefinite for any  $x_1,\ldots,x_n\in\mathcal{X}$ and any given $n \in \mathbb{N}$. 
We emphasize that the positive semidefiniteness of kernels in non-Euclidean spaces requires care in the sense that the natural non-Euclidean generalization of positive definite kernels in Euclidean spaces (like the Gaussian kernel) may not be positive definite \cite{feragen2015geodesic, da2023gaussian, da2023geometric, li_gaussian_2023}. However, as the next proposition shows, it is easy to construct  Mercer kernels in any non-Euclidean space.

\begin{proposition} {\rm \cite[page 69]{van2012harmonic}} For any real-valued function $f$ on an non-empty set $\mathcal{X}$, the function $K(x,y):=f(x) f(y)$ is a Mercer kernel on $\mathcal{X}$. Moreover, if $K_1,K_2$ are Mercer kernels on $\mathcal{X}$, so is $K_1K_2$.
\end{proposition}

\noindent A Mercer kernel is the key element to define a {\bfi reproducing kernel Hilbert space (RKHS)} as follows.

\begin{definition}[{\bf RKHS}] Let $K:\mathcal{X}\times\mathcal{X}\to \mathbb{R}$ be a Mercer kernel on a nonempty set $\mathcal{X}$. A Hilbert space $\mathcal{H}_K$ of real-valued functions on ${\cal X} $ endowed with the pointwise sum and pointwise scalar multiplication, and with inner product $\langle\cdot , \cdot \rangle_{\mathcal{H}_K}$ is called a reproducing kernel Hilbert space (RKHS) associated to $K$ if the following properties hold:
\begin{description}
\item [(i)]  For all $x\in\mathcal{X}$, we have that the function $K(x,\cdot)\in\mathcal{H}_K$.
\item [(ii)]  For all $x\in\mathcal{X}$ and for all $f\in\mathcal{H}_K$, the following reproducing property holds 
$$
f(x)=\langle f,K(x,\cdot)\rangle_{\mathcal{H}_K}.
$$
\end{description}
\end{definition}
The Moore-Aronszajn Theorem \cite{aronszajn1950theory} establishes that given a Mercer kernel $K$ on a set ${\cal X} $, there is a unique Hilbert space of real-valued functions $\mathcal{H}_K$ on ${\cal X} $ for which $K$ is a reproducing kernel, which allows us to talk about the RKHS $\mathcal{H}_K$ associated to $K$. Conversely, an RKHS naturally induces a kernel function that is a Mercer kernel.
Therefore, there is a bijection between RKHSs and Mercer kernels.

We now state a differential reproducing property on Riemannian manifolds whose proof can be found in Appendix \ref{proof-differential-reproducing-properties}.
The following theorem is a generalization of similar results proved for compact \cite{novak2018reproducing, zhou2008derivative}, bounded \cite{ferreira2012reproducing}, and unbounded \cite{hu2024structure} subsets of Euclidean spaces.  

\begin{theorem}[{\bf Differential reproducing property on Riemannian manifolds}]
\label{Par-Rep}
Let $M$ be a complete Riemannian manifold. Let $K:M\times M\rightarrow \mathbb{R}$ be a Mercer kernel such that $K\in C_b^{2s+1}(M\times M)$, for some $s \in \mathbb{N} $ (see Remark \ref{Functions with bounded differentials} for the definition of this space). 
Then the following statements hold:
\begin{description}
\item [(i)] For all $1\leq k\leq s$, we have that the function $D^{(k,0)}K(y,\cdot)\cdot v\in \mathcal{H}_K$ for all $y \in T^{k-1}M  $  and $v \in T _y( T^{k-1}M)$.
\item [(ii)] A differential reproducing property holds true for all $f\in \mathcal{H}_K$:
\begin{equation}\label{dif-rep}
D^k f(y)\cdot v=\langle D^{(k,0)}K(y,\cdot)\cdot v,f \rangle_{\mathcal{H}_K},\quad \mbox{for all  $y \in T^{k-1}M  $, $v \in T _y( T^{k-1}M)$, and $1\leq k\leq s$.}
\end{equation}
\item [(iii)] Denote $\kappa^2=(s+1)\|K\|_{C_b^{s}}$. The inclusion $J: \mathcal{H}_K \hookrightarrow C_b^s(M)$ is well-defined and bounded:
\begin{align}
\label{inclusion diff rkhs}
\|f\|_{C_b^s} \leqslant \kappa\|f\|_{\mathcal{H}_K}, \quad \forall~ f \in \mathcal{H}_K .    
\end{align}
\end{description}
\end{theorem}

For convenience, we use the following notation in what follows: given $x\in M$, we denote by $K_{x}:M \rightarrow  \mathbb{R}$ the function on $M$ given by $K_x(y):=K(x,y)$ and we call it the {\bfi kernel section} of $K$ at $x\in M$. Therefore, $K_{\cdot}$ can be regarded as a function on $M\times M$ such that $K_x(y):=K(y,x)$.
For any function $f\in\mathcal{H}_K$, we denote 
by $\langle f,K_{\cdot}\rangle_{\mathcal{H}_K}$ the function in $\mathcal{H}_K$ given by 
$\langle f,K_{y}\rangle_{\mathcal{H}_K}=f(y)$ (by the reproducing property). Furthermore, for a Poisson manifold $P$ equipped with a Riemannian metric $g$, we shall also denote by $X_{K_{\cdot}}$ the family of Hamiltonian vector fields parametrized by the elements in $P$, that is, given any $z\in P$, $X_{K_{z}}$ is the Hamiltonian vector field of the function $K_z$. In particular, for any $z \in P $  and any vector $v\in T_zP$, $g(z)\left(X_{K_{\cdot}}(z), v\right)$ is a function on the manifold $P$.

\begin{remark}\label{Par-Rep1}
\normalfont
{\bf (i)}  By Theorem \ref{Par-Rep} {\bf (i)}, $D^{(1,0)}K(x,\cdot)\cdot u \in \mathcal{H}_K$ for any $u\in T_xM$. 
Thus, the differential reproducing property \eqref{dif-rep} with $f=D^{(1,0)}K(x,\cdot)\cdot u$ implies that
\begin{align*}
D^{(1,1)}K(x,y)\cdot(u,v) = \left\langle D^{(1,0)}K(x,\cdot)\cdot u,D^{(1,0)}K(y,\cdot)\cdot v \right\rangle_{\mathcal{H}_K} , \quad \mbox{for all  $y\in M, v\in T_yM$.}
\end{align*}
Furthermore, in the presence of a Poisson structure and combining with Proposition \ref{Casimir property}, we obtain that for any vector field $Y\in\mathfrak X( M)$,
\begin{align*}
g(z)\left( X_{K_{\cdot}}(z), Y(z)\right)   = D^{(1,0)}K(z,\cdot)\cdot (-J(z)Y(z))\in\mathcal{H}_K.
\end{align*}

\noindent {\bf (ii)} Note that unlike the differential reproducing property in \cite{hu2024structure} that was formulated in terms of partial derivatives, \eqref{dif-rep} is a coordinate-free expression that has been written down in terms of global differentials.
\end{remark}

\begin{corollary}\label{Wel-Ope}
Let $M$ be a Riemannian manifold. Denote by $\mathcal{H}_K$ the RKHS associated with the kernel $K: M\times M\to \mathbb{R}$. Then for each vector field $X\in \mathfrak{X}(M)$ and each $f\in\mathcal{H}_K$, we have that
\begin{align*}
X[f](x)=X[\langle f, K_{\cdot}\rangle_{\mathcal{H}_K}](x),\quad \forall\ x\in M. 
\end{align*}
\end{corollary}

\begin{proof}
By the reproducing property, we have that for each  $f\in\mathcal{H}_K$,
$
f(x)=\langle f, K_{x}\rangle_{\mathcal{H}_K}$, for all $ x\in M$.   
Let $F_t$ be the flow of the vector field $X$. Then,
\begin{equation*}
X[\langle f, K_{\cdot}\rangle_{\mathcal{H}_K}](x):=\left.\frac{d}{dt}\right|_{t=0}[\langle f, K_{F_{t}(x)}\rangle_{\mathcal{H}_K}]= \left.\frac{d}{dt}\right|_{t=0} f(F_t(x))= X[f](x).   
\end{equation*}
\end{proof}

\subsection{The structure-preserving learning problem}
As we already explained in the introduction, the data that we are given for the learning problem are noisy realizations of the Hamiltonian vector field $X _H \in \mathfrak{X}(P) $ whose Hamiltonian function $H \in C^{\infty}(P)$ we are trying to estimate, that is,
\begin{align*}
\mathbf{X}_{\sigma^2}^{(n)}:= X_H(\mathbf{Z}^{(n)})+\bm{\varepsilon}^{(n)}, \quad n=1,\ldots,N,
\end{align*}
where $\mathbf{Z}^{(n)}$ are IID random variables with values in the Poisson manifold $P$ with the same distribution $\mu_{\mathbf{Z}}$ and $\bm\varepsilon^{(n)}$ are independent random variables on the tangent spaces $T_{\mathbf{z}^{(n)}}P$ with mean zero and variance $\sigma^2I_d$ (see Remark \ref{Random variables on tangent spaces}). 
We notice that unlike the Euclidean case studied in \cite{hu2024structure} where it was supposed that $\bm \varepsilon^{(n)}$ have the same distribution in $\mathbb{R}^{d}$, in the present situation, the tangent spaces $T_{\mathbf{z}^{(n)}}P$ are different vector spaces even though they are isomorphic and hence, we can not assume that the random variables $\bm \varepsilon^{(n)}$  in different spaces have the same distribution. However, the independence hypothesis and the assumptions on the moments are sufficient for our structure-preserving kernel regression. 

\noindent {\bf Notation.} We shall write the random samples as: 
\begin{align*}
\begin{split}
\mathbf{Z}_N&:= \mathrm{Vec}\left(\mathbf{Z}^{(1)}|\cdots |\mathbf{Z}^{(N)}\right)\in \Pi_{i=1}^NP, \\
\mathbf{X}_{\sigma^2,N}&:=\mathrm{Vec}\left(\mathbf{X}_{\sigma^2}^{(1)}|\cdots|\mathbf{X}_{\sigma^2}^{(N)}\right)\in \Pi_{i=1}^NT_{\mathbf{Z}^{(i)}}P,
\end{split}
\end{align*}
where the symbol `$\mathrm{Vec}$' stands for the vectorization of the corresponding matrices. We shall denote by $\mathbf{z}^{(n)}$ and $\mathbf{x}_{\sigma^2}^{(n)}$ the realizations of the random variables  $\mathbf{Z}^{(n)}$ and $\mathbf{X}_{\sigma^2
}^{(n)}$, respectively. The collection of realizations is denoted by
\begin{align*}
\begin{split}
\mathbf{z}_N&:= \mathrm{Vec}\left(\mathbf{z}^{(1)}|\cdots |\mathbf{z}^{(N)}\right)\in \Pi_{i=1}^NP, \\
\mathbf{x}_{\sigma^2,N}&:=\mathrm{Vec}\left(\mathbf{x}_{\sigma^2}^{(1)}|\cdots |\mathbf{x}_{\sigma^2}^{(N)}\right)\in \Pi_{i=1}^NT_{\mathbf{z}^{(i)}}P.
\end{split}
\end{align*}
In the sequel, if $f:P\to \mathbb{R}^s$ is a function, we then shall denote the value $\mathrm{Vec}\left(f(\mathbf{Z}^{(1)})|\cdots|f(\mathbf{Z}^{(N)})\right)\in \mathbb{R}^{sN}$ by $f(\mathbf{Z}_N)$.

\noindent {\bf Structure-preserving kernel ridge regression.}\quad
The strategy behind structure-preservation in the context of kernel regression is that we search the vector field $X\in\mathfrak{X}(P)$ that minimizes a risk functional among those that have the form  $X:=X_h$, where $h: P \rightarrow  \mathbb{R}$ belongs to the RKHS $\mathcal{H}_K$ associated with a kernel $K$ defined on the Poisson manifold $P$. 
This approach obviously guarantees that the learned vector field is Hamiltonian. 
To make the method explicit, we shall be solving the following optimization problem 
\begin{align}
\widehat{h}_{\lambda,N}&:=\mathop{\arg\min}\limits_{h\in \mathcal{H}_{K}} \ \widehat{R}_{\lambda,N}(h), \label{emp-pro}\\
\widehat{R}_{\lambda,N}(h)&:=\frac{1}{N}\sum_{n=1}^{N} \|X_h(\mathbf{Z}^{(n)})-\mathbf{X}^{(n)}_{\sigma^2}\|_{g}^2 + \lambda\|h\|_{\mathcal{H}_K}^2, \label{emp-fun}
\end{align}
where $X_h$ is the Hamiltonian vector field of $h$ and $\lambda\geq0$ is the Tikhonov regularization parameter.
We shall refer to the minimizer $ \widehat{h}_{\lambda,N}$ as the {\bfi structure-preserving kernel estimator} of the Hamiltonian function $H$. The functional $\widehat{R}_{\lambda,N}$ is called the {\bfi regularized empirical risk}. 

\begin{remark}
\normalfont
{\bf (i)} If $K\in C_b^{3}(P\times P)$ then the differential reproducing property in Theorem \ref{Par-Rep} implies that the functions in ${\mathcal H}_K $ are all differentiable because and hence they always have a Hamiltonian vector field associated. This implies that the regularized empirical risk $\widehat{R}_{\lambda,N}$ defined in \eqref{emp-fun} and the associated optimization problem \eqref{emp-pro} is always well-defined in that situation. 

\noindent {\bf (ii)} The minimization problem \eqref{emp-pro} might be ill-posed since for any Casimir function $C \in \mathcal{C} \cap \mathcal{H}_K$ and $h \in \mathcal{H}_K$, the function $h+C$ shares the same Hamiltonian vector field as $h$. This could result in the lack of unique solutions for the minimization problem. Later on in Remark \ref{uniqueness thanks to regularization}, we show how ridge regularization solves this problem.
\end{remark}

Given the probability measure $\mu_{\mathbf{Z}}$ and the Riemannian metric $g$ on the manifold $P$, a vector field $Y\in\mathfrak{X}(P)$ is said $L^2(\mu_{\mathbf{Z}})$-integrable if the following $L^2(\mu_{\mathbf{Z}})$ norm of $Y$ is finite, that is,
\begin{align}
\label{l2norm}
\|Y\|_{L^2(\mu_{\mathbf{Z}})}:=\int_{P}g(z)(Y(z),Y(z))\mathrm{d}\mu(z)<\infty.  
\end{align}
Denote by $L^2(\mathfrak{X}(P),g,\mu_{\mathbf{Z}})$ the space consisting of all $L^2(\mu_{\mathbf{Z}})$-integrable vector fields. 
The measure-theoretic analogue $R_{\lambda}$ of \eqref{emp-fun} is referred to as {\bfi regularized statistical risk} and is defined by
\begin{align}
R_{\lambda}(h)&:=\|X_h- X_H\|_{L^2(\mu_{\mathbf{Z}})}^2+\lambda \|h\|_{\mathcal{H}_K}^2+\sigma^2, \label{exp-fun}
\end{align}
where the $L^2(\mu_{\mathbf{Z}})$ norm is defined in \eqref{l2norm}  and $X_h,X_H$ are the Hamiltonian vector fields of $h$ and $H$, respectively. We denote by $h^{*}_{\lambda}\in\mathcal{H}_K$ the best-in-class functional with the minimal regularized statistical risk, that is,
\begin{align}\label{exp-pro}
h^{*}_{\lambda}:= \mathop{\arg\min}\limits_{h\in \mathcal{H}_K}  R_{\lambda}(h) .  
\end{align}
We note that by the strong law of large numbers, the regularized empirical and statistical risks are consistent within the RKHS, that is, for every $h\in\mathcal{H}_K$, we have that
\begin{align*}
\lim\limits_{N\to \infty} \mathbb{E}_{\bm\varepsilon}\left[\widehat{R}_{\lambda, N} (h)\right]=R_{\lambda}(h),\quad \mbox{almost surely}, 
\end{align*}
where $\mathbb{E}_{\bm\varepsilon}$ means taking the conditional expectation for all random variables $\bm\varepsilon^{(n)}$ given all $\mathbf{Z}^{(n)}$.

\subsection{Operator representations of the kernel estimator}\label{Operator representations of the kernel estimator}
In order to find and study the solutions of the inverse learning problems \eqref{emp-pro}-\eqref{emp-fun} and \eqref{exp-fun}-\eqref{exp-pro}, we introduce the operators $A:\mathcal{H}_K\to \mathfrak{X}_{\mathrm{Ham}}(P)$ and $A_N:\mathcal{H}_K\to T_{\mathbf{Z}_N}P:=\Pi_{i=1}^NT_{\mathbf{Z}^{(i)}}P$ as 
\begin{align}\label{operatorA}
Ah:=X_{h}=J\nabla h, \text{ and } A_Nh:=X_{h}(\mathbf{Z}_N)=\frac{1}{\sqrt{N}}\mathbb{J}_N\nabla h(\mathbf{Z}_N),\quad \text{ for all } h\in \mathcal{H}_K, 
\end{align} 
where $\mathbb{J}_N=\operatorname{diag}\{J(\mathbf{Z}^{(1)}),\ldots,J(\mathbf{Z}^{(N)})\}$ and $J$ is the compatible structure defined in \eqref{Com-str}. By Theorem \ref{Par-Rep}, if the Mercer kernel $K\in C_b^3(P\times P)$, then the functions in $\mathcal{H}_K$ are differentiable, and hence the operators $A$ and $A_N$ are well-defined.  We aim to show that the image of the operator $A$ belongs to the space $L^2(\mathfrak{X}(P),g,\mu_{\mathbf{Z}})$. 
To achieve this, we must make the following boundedness assumption on the compatible structure $J$.

\begin{assumption}\label{Con-com}
The compatible structure $J$  defined in \eqref{Com-str} satisfies 
\begin{align}\label{Asu-com}
g(z)(J(z)v,J(z)v)\leq \gamma(z)g(z)(v,v), \quad \text{for all } z\in P, \text{ and } v\in T_zP,   
\end{align}
where $\gamma$ is a positive function on the Poisson manifold $P$ bounded above by some constant $C^2$ ($C>0$). 
% where $\gamma$ is a non-negative integarble function on $M$ with respect to $\mu_{\mathbf{Z}}$.    
\end{assumption}

Using this boundedness condition on the compatible structure $J$, we can obtain the boundedness of the operators $A$ and $A_N$ defined in \eqref{operatorA} as maps from $\mathcal{H}_K$ to $L^2(\mathfrak{X}(P),g,\mu_{\mathbf{Z}})$ and from $\mathcal{H}_K$ to $T_{\mathbf{Z}_N}P$, respectively. A detailed proof of the following proposition can be found in the Appendix \ref{proof of boundA}

\begin{proposition}\label{BoundA}
Let $P$ be a Poisson manifold and let $K\in C_b^3(P\times P)$ be a Mercer kernel. Suppose that Assumption \ref{Con-com} holds.
Then, the operator $A$ defined in \eqref{operatorA} is a bounded linear operator that maps $\mathcal{H}_K$ into $L^2(\mathfrak{X}(P),g,\mu_{\mathbf{Z}})$ with an operator norm $\|A\|$ that satisfies $\|A\|\leq \kappa C$, where $\kappa^2=2\|K\|_{C_b^{2}}$. The adjoint operator 
$A^{*}: L^2(\mathfrak{X}(P),g,\mu_{\mathbf{Z}}) \longrightarrow {\mathcal H} _K$ of $A$ is given by
\begin{align*}
A^{*}Y=\int_{P} g(z)(X_{K_{\cdot}}(z),Y(z)) \mathrm{d} \mu_{\mathbf{Z}}(z), \quad \mbox{for all $Y\in L^2(\mathfrak{X}(P),g,\mu_{\mathbf{Z}})$.}
\end{align*} 
As a consequence, {the bounded linear operator $Q: \mathcal{H}_K\longrightarrow \mathcal{H}_K$, defined by}
\begin{align}\label{operatorQ}
Qh:=A^{*}A h=\int_{P} g(z)( X_{K_{\cdot}}(z),X_h(z)) \mathrm{~d}\mu_{\mathbf{Z}}(z),
\end{align}
is a positive trace class operator that satisfies $\operatorname{Tr}(Q)\leq \frac{1}{2}d\kappa^2C^2$.   
\end{proposition}

\begin{remark}\normalfont
{\bf (i)} The proof of 
Proposition \ref{BoundA} shows that the boundedness condition of the compatible structure $J$ in Assumption \ref{Con-com} can be weakened by requiring that the function $\gamma$ is positive $L^1$-integrable with respect to the probability measure $\mu_{\mathbf{Z}}$. This condition obviously holds if $\mu_{\mathbf{Z}}$ is compactly supported or it is a Gaussian probability measure, which shows this boundedness condition holds for a large class of compatible structures.  

\noindent {\bf (ii)} When the Poisson manifold $P$ is symplectic and the compatible structure $J$ satisfies $J^2=-I$ then it is a complex structure and $P$ becomes a K\"ahler manifold. Even-dimensional vector spaces are special cases of K\"ahler manifolds by picking as the complex structure the one naturally associated to the canonical (Darboux) symplectic form. In these cases, the boundedness condition \eqref{Asu-com} naturally holds by choosing $\gamma\equiv1$. This is why this condition never appears in \cite{hu2024structure}.
\end{remark}

We now deal with the empirical version $A_N$ of the operator $A$ defined in \eqref{operatorA}. 
For convenience, we define the inner product $g_N(\cdot,\cdot)$ on the product space $T_{\mathbf{Z}_N}P$ as follows:
\begin{align}
\label{gnoutofg}
g_N(u,v):=\sum_{i=1}^N  g\left(\mathbf{Z}^{(i)}\right)(u_i,v_i), \quad u=(u_1,\ldots,u_N),v=(v_1,\ldots,v_N)\in T_{\mathbf{Z}_N}P,
\end{align}
where $u_i,v_i\in T_{\mathbf{Z}^{(i)}}P$ for $i=1,\ldots,N$. Furthermore, we denote by $\|\cdot\|_{g_N}$ the corresponding norm.

\begin{proposition}
Suppose that Assumption \ref{Con-com} holds. Then, the operator
$A_N: \mathcal{H}_K \rightarrow T_{\mathbf{Z}_N}P$ defined in \eqref{operatorA}
is a bounded linear operator whose operator norm $\|A_N\|\leq \kappa C$.
The adjoint operator $A_N^*:T_{\mathbf{Z}_N}P  \rightarrow \mathcal{H}_K$ of $A_N$ is a finite rank operator given by
\begin{align*}
A^*_{N}W=\frac{1}{\sqrt{N}}g_N(W,X_{K_{\cdot}}(\mathbf{Z}_N)),\quad W\in T_{\mathbf{Z}_N}P.    
\end{align*}
Moreover, the operator $Q_N$ defined by
\begin{align}\label{emp-ope}
Q_Nh:=A^*_{N}A_Nh=\frac{1}{N}g_N(X_h(\mathbf{Z}_N),X_{K_{\cdot}}(\mathbf{Z}_N)),\quad h\in \mathcal{H}_K,    
\end{align}
is a positive-semidefinite compact operator.  
\end{proposition}
\begin{proof}
The explicit expressions of $A_N^{\ast}$ and $Q_N$ follow from a direct calculation. Under Assumption \ref{Con-com} it follows from the equation \eqref{norm-control} that
\begin{equation*}
\|A_Nh\|^2=\frac{1}{N}\|\mathbb{J}_N\nabla h(\mathbf{Z}_N)\|_{g_N}^2\leq C^2\kappa^2\|h\|^2_{\mathcal{H}_K},
   \end{equation*}
which implies that $A_N$ is bounded and that $\|A_N\|\leq \kappa C$. The compactness of the operator $Q_N$ follows from a proof similar to that of Proposition \ref{BoundA} for $Q$.
\end{proof}

Having defined the operators $A$ and $A_N$, we obtain an operator representation of the minimizers of the learning problems \eqref{emp-pro}-\eqref{emp-fun} and \eqref{exp-fun}-\eqref{exp-pro} that is similar to the one that was already stated in in \cite{hu2024structure}.

\begin{proposition}\label{Wel-Emp}
Let $\widehat{h}_{\lambda,N}$ and $h_{\lambda}^*$ be the minimizers of \eqref{emp-fun} and \eqref{exp-fun}, respectively. Then, for all $\lambda>0$, the minimizers $h_{\lambda}^*$ and $ \widehat{h}_{\lambda,N}$ are unique and they are given by 
\begin{align}
h_{\lambda}^*:&=(Q+\lambda I)^{-1}A^{*}X_H. \nonumber \\
 \widehat{h}_{\lambda,N}:&=\frac{1}{\sqrt{N}}(Q_N+\lambda I)^{-1}A_N^{*}\mathbf{X}_{\sigma^2,N}, \mbox{ with } Q_N=A_N^*A_N. \label{firstexpressionestimator}
\end{align}    
\end{proposition}

\subsection{The global solution of the learning problem}
\label{A global solution to the learning problem}
We now derive a kernel representation for the solution of the learning problem \eqref{emp-pro}-\eqref{emp-fun}. More precisely, we derive a closed-form expression for the estimator of a Hamiltonian function of the data-generating process by defining a generalized differential Gram matrix on manifolds. We shall see that the symmetric and positive semidefinite properties of this generalized differential Gram matrix guarantee the uniqueness of the estimator. The main expression that will be obtained is stated below in Theorem \ref{Rep-Ker} and amounts to a differential version of the Representer Theorem for Poisson manifolds and that is how we shall refer to it. In order to derive it, we first define a generalized differential Gram matrix $G_N:T_{\mathbf{Z}_N}P\rightarrow T_{\mathbf{Z}_N}P$ as
\begin{align}\label{general-Gram}
G_N\mathbf{c}:= X_{g_N(\mathbf{c},X_{K_{\cdot}}(\mathbf{Z}_N))}(\mathbf{Z}_N),\quad \mathbf{c}\in T_{\mathbf{Z}_N}P.   
\end{align}
As it was studied in \cite{hu2024structure}, in symplectic vector spaces endowed with the canonical symplectic form $\omega_{\mathrm{can}}$, the generalized differential Gram matrix $G_N$ defined in \eqref{general-Gram} reduces to the usual differential Gram matrix $\mathbb{J}_{\mathrm{can}}\nabla_{1,2}K(\mathbf{Z}_N,\mathbf{Z}_N)\mathbb{J}_{\mathrm{can}}^{\top}$, where $\nabla_{1,2}K$ represents the matrix of partial derivatives of $K$ with respect to the first and second arguments.

The general differential Gram matrix $G_N$ defined in \eqref{general-Gram} corresponds, roughly speaking, to the operator $Q_N$ defined in \eqref{emp-ope} (see also \cite{hu2024structure}). In the following proposition that is proved in the appendix, we see that the generalized differential Gram matrix $G_N$ is symmetric and positive semidefinite.

\begin{proposition}
\label{Pos-Gra} 
Given a Mercer kernel $K\in C_b^3(P\times P)$ on the $d$-dimensional Poisson manifold $(P, \left\{\cdot , \cdot \right\})$ and endowed with a Riemannian metric $g$, the generalized differential Gram matrix $G_N$ defined in \eqref{general-Gram} is symmetric and positive semidefinite with respect to the metric $g_N$ in \eqref{gnoutofg}, that is,
\begin{equation*}
g_N(\mathbf{c},G_N\mathbf{c}')=g_N(\mathbf{c}',G_N\mathbf{c})~\text{ and }~g_N(\mathbf{c},G_N\mathbf{c})\geq 0 ~\text{ for all } \mathbf{c},\mathbf{c}'\in T_{\mathbf{Z}_N}P.
    \end{equation*}
\end{proposition}

\begin{theorem}[{\bf Differential Representer Theorem on Poisson manifolds}]
\label{Rep-Ker} 
% Suppose that the matrix $\nabla_{1,2}K(\mathbf{Z}_N,\mathbf{Z}_N)+\lambda N I$ is invertible.
Let $K\in C_b^3(P\times P)$ be a Mercer kernel on the Poisson manifold $(P, \left\{\cdot , \cdot \right\})$ and endowed with a Riemannian metric $g$. Suppose that Assumption \ref{Con-com} holds. Then for every $\lambda>0$, the estimator $\widehat{h}_{\lambda, N}$ introduced in \eqref{emp-pro} as the minimizer of the regularized empirical risk functional $ \widehat{R}_{\lambda,N}$ in \eqref{emp-fun}
%\eqref{regularizedrisk1}
can be written as
\begin{equation}\label{rep-ker}
\widehat{h}_{\lambda,N}= g_N( \widehat {\bf c},X_{K_{\cdot}}(\mathbf{ Z}_N)),
\end{equation}
where $\widehat{\mathbf{c}}\in T_{\mathbf{Z}_N}P$ is given by
\begin{align*}
\widehat{\mathbf{c}}=(G_N+\lambda NI)^{-1}\mathbf{X}_{\sigma^2,N}.
\end{align*}
\end{theorem}

\begin{proof}
Denote the space $\mathcal{H}_K^{N}$ by
\begin{equation*}
\mathcal{H}_K^{N}:=\left\{g_N(X_{K_{\cdot}}(\mathbf{Z}_N),\mathbf{c}), \  \mathbf{c}\in T_{\mathbf{Z}_N}P \right\},
\end{equation*}
By part {\bf (i)} of Theorem \ref{Par-Rep}, $\mathcal{H}_K^N$ is a subspace of $\mathcal{H}_K$. Then by the representation of the operator $Q_N$ in the equation \eqref{emp-ope}, we know that $Q_N(\mathcal{H}_{K}^N) \subseteq \mathcal{H}_{K}^N$, that is, $\mathcal{H}_{K}^N$ is an invariant space for the operator $Q_N$. This implies that, for any $\lambda>0 $, $(Q_N+ \lambda I)(\mathcal{H}_{K}^N) \subseteq \mathcal{H}_{K}^N$. By Proposition \ref{Pos-Gra}, the operator $Q_N $ is positive semi-definite, we can conclude that the restriction $(Q_N+ \lambda I)|_{\mathcal{H}_{K}^N} $ is invertible and since the space $\mathcal{H}_{K}^N  $ is finite-dimensional then it is also an invariant subspace of $(Q_N+ \lambda I)|_{\mathcal{H}_{K}^N}^{-1} $, that is 
\begin{equation*}
(Q_N+ \lambda I)|_{\mathcal{H}_{K}^N}^{-1} \left(\mathcal{H}_{K}^N\right) \subset \mathcal{H}_{K}^N.
\end{equation*}
Then, there exists a vector 
$  \widehat {\mathbf{c}}\in T_{\mathbf{Z}_N}P $ such that 
\begin{equation}\label{rep-ker1}
\widehat{h}_{\lambda,N}= g_N(\widehat {\mathbf{c}}, X_{K_{\cdot}}(\mathbf{Z}_N)).
\end{equation}
Then, applying $(Q_N + \lambda I)$ on the left of both sides on the estimator $\widehat{h}_{\lambda,N}$ in \eqref{firstexpressionestimator} and plugging \eqref{rep-ker1} into the identity, we obtain that
\begin{align}\label{euler equ in this case} 
\frac{1}{N}g_N\left(X_{K_{\cdot}}(\mathbf{Z}_N),\mathbf{X}_{\sigma^2,N}\right)=\frac{1}{N}g_N\left(X_{\widehat{h}_{\lambda,N}}(\mathbf{Z}_N),X_{K_{\cdot}}(\mathbf{Z}_N)\right)+\lambda g_N\left(\widehat {\mathbf{c}}, X_{K_{\cdot}}(\mathbf{Z}_N)\right).
\end{align}
Therefore, we obtain that
\begin{align*}
G_N\widehat{\mathbf{c}}+\lambda N \widehat{\mathbf{c}}=\mathbf{X}_{\sigma^2,N}.
\end{align*}
Since the matrix $G_N+\lambda N I$ is invertible due to the positive semi-definiteness of $G_N$ that we proved in Proposition \ref{Pos-Gra}, we can write that 
\begin{align}
\label{expression for chat}
\widehat{\mathbf{c}}= (G_N+\lambda NI)^{-1} \mathbf{X}_{\sigma^2,N}.
\end{align}
This shows that the function $\widehat{h}_{\lambda,N} $ in \eqref{rep-ker1} with $\widehat{\mathbf{c}} $ determined by \eqref{expression for chat} coincides with minimizer \eqref{firstexpressionestimator} of the regularized empirical risk functional $ \widehat{R}_{\lambda,N}$ in \eqref{emp-fun}. Since by Proposition \ref{Wel-Emp}, this minimizer is unique, the result follows.
\end{proof}

\begin{remark}[{\bf On the uniqueness of the estimated Hamiltonian}]
\label{uniqueness thanks to regularization}
\normalfont
Define the kernel of the operator $A$ as follows:
\begin{align*}
\mathcal{H}_{\mathrm{null}}:=\{h\in\mathcal{H}_K\mid Ah=J\nabla h=0\}.
\end{align*}
It can be checked that $\mathcal{H}_{\mathrm{null}}$ is a closed vector subspace of $\mathcal{H}_K$. Indeed, let $ \left\{f _n\right\}_{n \in \mathbb{N}} $ be a convergent sequence in $\mathcal{H}_{\mathrm{null}}$ such that $\left\|f _n-f\right\|_{{\mathcal H} _K}\rightarrow 0 $ as $n \rightarrow \infty $, for some $f \in {\mathcal H}_K$. 
Then, by \eqref{norm-control} in the proof of Proposition \ref{BoundA} and Assumption \ref{Con-com}, it holds that
\begin{equation*}
\sup_{z\in P} \|X_{f}(z)\|_g^2=\sup_{z\in P} \|X_{\left(f _n-f\right)}(z)\|_g^2 \leq
\kappa^2 C^2\|f_n-f\|_{\mathcal{H}_K}^2  \rightarrow 0, \quad \mbox{as $n \rightarrow \infty$,}
\end{equation*}
which implies that $Af=0$ and hence $\mathcal{H}_{\mathrm{null}}$ is a closed subspace of ${\mathcal H}_K $. Therefore, $\mathcal{H}_K$ can be decomposed as 
\begin{align*}
\mathcal{H}_K=\mathcal{H}_{\mathrm{null}}\oplus \mathcal{H}_{\mathrm{null}}^\bot,
\end{align*}
where $\mathcal{H}_{\mathrm{null}}^\bot $ denotes the orthogonal complement of $\mathcal{H}_{\mathrm{null}}$ with respect to the RKHS inner product $\left\langle \cdot , \cdot \right\rangle_{\mathcal{H}_K}$. 
Using this decomposition and the expression of the kernel estimator \eqref{rep-ker}, we conclude that 
\begin{align}\label{estimator-null-bot}
\widehat h_{\lambda,N}\in\mathcal{H}_{\mathrm{null}}^\bot,   
\end{align}
which is due to the fact that for any $h\in\mathcal{H}_{\mathrm{null}}$,
\begin{align*}
\langle \widehat{h}_{\lambda,N},h\rangle_{\mathcal{H}_K} &= \langle g_N(X_{K_{\cdot}}(\mathbf{Z}_N),\widehat{\mathbf{c}}),h\rangle_{\mathcal{H}_K} = \langle D^{(1,0)}K(\mathbf{Z}_N,\cdot)\cdot(-\mathbb{J}_N\widehat{\mathbf{c}}),h\rangle_{\mathcal{H}_K}\\
& = Dh(\mathbf{Z}_N)\cdot(-\mathbb{J}_N\widehat{\mathbf{c}}) = g_N(X_h(\mathbf{Z}_N), \widehat{\mathbf{c}})=0,
\end{align*}
where the second and fourth equalities follow by Proposition \ref{Casimir property}, and the third equality is due to the differential reproducing property \eqref{dif-rep} in Theorem \ref{Par-Rep}.

In general, the space $\mathcal{H}_{\mathrm{null}}$ could contain Casimir functions in $\mathcal{C}(P)$ 
%(even in the case of the Gaussian kernel, where the corresponding RKHS does not contain any polynomial or non-zero constant functions \cite{minh2010some}, the RKHS could contain Casimir functions; see Section \ref{Applications to Lie-Poisson systems}). 
and this could, in principle, cause ill-posedness of the learning problem since adding Casimir functions in $\mathcal{H}_{\mathrm{null}}$ to $\widehat{h}_{\lambda,N} $ would not change the corresponding Hamiltonian vector field. Therefore, one may wonder why, according to Theorem \ref{Rep-Ker}, the estimator $\widehat{h}_{\lambda,N} $ is unique but not up to Casimir functions. The explanation is in the regularization term. Indeed, let $\widehat{h}_{\lambda,N}$ be the minimizer in \eqref{rep-ker}  and let $h\in\mathcal{H}_{\mathrm{null}}$. Even though $\widehat{h}_{\lambda,N}$ and $\widehat{h}_{\lambda,N}+h$ have the same Hamiltonian vector field associated, it is easy to show that $\widehat{h}_{\lambda,N}+h$ is a minimizer of \eqref{emp-pro} if and only if $h\equiv0$. This is because of \eqref{estimator-null-bot} and
\begin{align*}
\widehat{R}_{\lambda,N}(\widehat{h}_{\lambda,N}+h)&:=\frac{1}{N}\sum_{n=1}^{N} \|X_{\widehat{h}_{\lambda,N}+h}(\mathbf{Z}^{(n)})-\mathbf{X}^{(n)}_{\sigma^2}\|_g^2 + \lambda\|\widehat{h}_{\lambda,N}+h\|_{\mathcal{H}_K}^2  \\
&=\frac{1}{N}\sum_{n=1}^{N} \|X_{\widehat{h}_{\lambda,N}}(\mathbf{Z}^{(n)})-\mathbf{X}^{(n)}_{\sigma^2}\|_g^2 + \lambda\left(\|\widehat{h}_{\lambda,N}\|_{\mathcal{H}_K}^2+\|h\|^2_{\mathcal{H}_K} +2\langle \widehat{h}_{\lambda,N},h\rangle_{\mathcal{H}_K}\right) \\
&=\frac{1}{N}\sum_{n=1}^{N} \|X_{\widehat{h}_{\lambda,N}}(\mathbf{Z}^{(n)})-\mathbf{X}^{(n)}_{\sigma^2}\|^2 + \lambda\left(\|\widehat{h}_{\lambda,N}\|_{\mathcal{H}_K}^2+\|h\|^2_{\mathcal{H}_K} \right).
\end{align*}
\end{remark}

\noindent{\bf Local coordinate representation of the kernel estimator.} \quad  In applications, it is important to have local coordinate expressions for the kernel estimator $\widehat h_{\lambda,N}$. Recall that by Theorem \ref{Rep-Ker}, the kernel estimator is given by
\begin{equation*}
\widehat{h}_{\lambda,N}= g_N( \widehat {\bf c},X_{K_{\cdot}}(\mathbf{ Z}_N))=g_N( \widehat {\bf c},\mathbb{J}_N\nabla K(\mathbf{ Z}_N,\cdot)),
\end{equation*}
where $\widehat{\mathbf{c}}\in T_{\mathbf{Z}_N}P$ is given by
\begin{align*}
\widehat{\mathbf{c}}=(G_N+\lambda NI)^{-1}\mathbf{X}_{\sigma^2,N}.
\end{align*}
Since $J$ can be locally represented as \eqref{local-com-str}, it suffices to write down the Gram matrix $G_N$ locally. Let $\mathbf{z}^{(1)},\ldots,\mathbf{z}^{(N)}$ be the available data points and let $\{z_n^1,\ldots,z_n^d\}$, $n =1, \ldots, N $, be local coordinates on open sets containing the points $\mathbf{z}^{(n)}$, for all $n=1,\ldots, N$. Let $g^{(n)}$ be the matrix associated to the Riemannian metric $g$ in the local coordinates $\{z_n^1,\ldots,z_n^d\}$.
We first compute the following function for each $\mathbf{c}\in T_{\mathbf{z}_N}P$, 
\begin{align*}
g_N(\mathbf{c},X_{K_{\cdot}}(\mathbf{z}_N)) = g_N(\mathbf{c},\mathbb{J}_N\nabla K(\mathbf{z}_N,\cdot))=\sum_{i=1}^N (\partial_1K)^\top(\mathbf{z}^{(i)},\cdot){g^{(i)}}^{-1}J^\top(\mathbf{z}^{(i)})g^{(i)}\mathbf{c}^i.
\end{align*}
By the definition of the Gram matrix in \eqref{general-Gram} and combining it with the local representation of $J$ in \eqref{local-com-str}, we obtain that 
\begin{align*}
(G_N\mathbf{c})^i&=X_{g_N(\mathbf{c},X_{K_{\cdot}}(\mathbf{z}_N))}(\mathbf{z}^{(i)})= \sum_{j=1}^N  J(\mathbf{z}^{(i)}){g^{(i)}}^{-1}\partial_{1,2}K(\mathbf{z}^{(i)},\mathbf{z}^{(i)}){g^{(j)}}^{-1}J^\top(\mathbf{z}^{(j)})g^{(j)}\mathbf{c}^j\\
&=\sum_{j=1}^N  (-B(\mathbf{z}^{(i)})g^{(i)}){g^{(i)}}^{-1}\partial_{1,2}K(\mathbf{z}^{(j)},\mathbf{z}^{(i)}){g^{(j)}}^{-1}(-{g^{(j)}}^\top B^\top(\mathbf{z}^{(j)}))g^{(j)}\mathbf{c}^j\\
&=\sum_{j=1}^NB(\mathbf{z}^{(i)})\partial_{1,2}K(\mathbf{z}^{(j)},\mathbf{z}^{(i)})B^\top(\mathbf{z}^{(j)})g^{(j)}\mathbf{c}^j,
\end{align*}
which shows that the $(i,j)$-component matrix $G_N^{i,j}$ of the general Gram matrix is given by
\begin{align}\label{local-Gram}
    G_N^{i,j} = B(\mathbf{z}^{(i)})\partial_{1,2}K(\mathbf{z}^{(i)},\mathbf{z}^{(j)})B^\top(\mathbf{z}^{(j)})g^{(j)}.
\end{align}

\subsection{Symmetries and momentum conservation by kernel estimated Hamiltonians}\label{conservation}

Studies have been carried out to bridge invariant feature learning with kernel methods \cite{mroueh2015learning}. In this subsection, we will see that if the unknown Hamiltonian function is invariant under a canonical group action, then the structure-preserving kernel estimator can also be made invariant under the same action by imposing certain symmetry constraints on the kernel $K$. As a result, the fibers of any momentum map associated with that group action that are preserved by the flow of the ground-truth Hamiltonian $H$ will also be preserved by the flow of the estimator $\widehat{h}_{\lambda,N}$. 

\begin{definition}
Left $G$ be a group acting on the manifold $P$. A function $h:P \rightarrow \mathbb{R}$ on the Poisson manifold $P$ is called $G$-{\bfi invariant} if 
\begin{align*}
    h(g\cdot x) = h(x),\quad \text{ for all   $g,g'\in G $ and all $x \in P $.} 
\end{align*}
A kernel $K:P \times P \rightarrow \mathbb{R}$ is called {\bfi argumentwise invariant} by the group action if 
\begin{align*}
K(g\cdot x, g'\cdot x')=K(x,x'),\quad \text{ for all   $g,g'\in G $ and all $x,x' \in P $.}   
\end{align*}
\end{definition}

Argumentwise invariant kernels can be constructed by double averaging (using the Haar measure if the group is compact) or, alternatively, if $K$ is an arbitrary Mercer kernel and $f$ is a $G$-invariant function. Then the function $K'$ defined as $K'(x,y)=K(f(x),f(y))$ is an argumentwise invariant kernel. 

\begin{proposition}[{\bf Noether's Theorem for the structure-preserving kernel estimator}]
Let $G$ be a Lie group acting canonically on the Poisson manifold $(P, \left\{\cdot , \cdot \right\})$.  If the Mercer kernel $K\in C_b^3(P\times P)$ is argumentwise invariant under the group $G$ then the kernel estimator $\widehat h_{\lambda,N}$  in Theorem \ref{Rep-Ker} is $G$-invariant. If additionally, the group action has a momentum map $\mathbf{J}:P \longrightarrow \mathfrak{g}^\ast $ associated, then its fibers are preserved by the flow $F _t $ of the corresponding Hamiltonian vector field $X_{\widehat h_{\lambda,N}} $, that is,
$$
\mathbf{J} \circ F _t=\mathbf{J}.
$$
\end{proposition}
\begin{proof}
By \cite[Property 3.13]{ginsbourger2012argumentwise}, if the kernel is argumentwise invariant under $G$-action, then the corresponding RKHS $\mathcal{H}_K$ consists exclusively of $G$-invariant functions. This implies that the kernel estimator is $G$-invariant since $\widehat h_{\lambda,N}\in \mathcal{H}_{null}^\bot\subset\mathcal{H}_K$. The statement on the momentum conservation is a straightforward consequence of Noether's Theorem (see, for instance, \cite[Theorem 11.4.1]{Marsden1994}).
\end{proof}

\section{Estimation and approximation error analysis}\label{Estimation and approximation error analysis}

We now analyze the extent to which the structure-preserving kernel estimator $ \widehat{h}_{\lambda,N}$ can recover the unknown Hamiltonian $H$ on a Poisson manifold $P$. 
A standard approach is to decompose the {\bfi reconstruction error} $\widehat{h}_{\lambda,N}-H $ into the sum of what we shall call {\bfi estimation} and {\bfi approximation errors}, and further decompose the estimation error $\widehat{h}_{\lambda,N}-h_{\lambda}^* $ into what we shall call {\bfi sampling error} and {\bfi noisy sampling error}. More specifically,
\begin{equation}\label{error decomposition}
\begin{aligned}
\widehat{h}_{\lambda,N}-H&=\underbrace{\widehat{h}_{\lambda,N}-h_{\lambda}^*}_{\text{Estimation error}}\quad+\underbrace{h_{\lambda}^*-H}_{\text{Approximation error}},\\
&=\underbrace{\widetilde{h}_{\lambda,N}-h_{\lambda}^*}_{\text{Sampling error}}\quad+ \quad\underbrace{\widehat{h}_{\lambda,N}-\widetilde{h}_{\lambda,N}}_{\text{Noisy sampling error}}\quad+\quad\underbrace{h_{\lambda}^*-H}_{\text{Approximation error}},
\end{aligned}
\end{equation}
where $h^{*}_{\lambda}\in\mathcal{H}_K$ is the best-in-class function introduced in \eqref{exp-pro} that minimizes the regularized statistical risk \eqref{exp-fun} and $\widetilde{h}_{\lambda,N}:=(Q_N+\lambda)^{-1}Q_NH$ is the noise-free part of the estimator $\widehat{h}_{\lambda,N}$. Indeed, the noisy sampling error can be computed as 
\begin{equation*} 
\begin{aligned}
\widehat{h}_{\lambda,N}-\widetilde{h}_{\lambda,N}
= \frac{1}{\sqrt{N}} (Q_N+\lambda)^{-1}A_N^{*}\mathbf{X}_{\sigma^2,N}-  (Q_N+\lambda)^{-1}Q_NH=\frac{1}{\sqrt{N}}(Q_N+\lambda)^{-1}A_N^{*}\mathbf{E}_N
\end{aligned}
\end{equation*}
where the noise vector $\mathbf{E} _N  $ is
\begin{equation*}
\mathbf{E}_N=\mathrm{Vec}\left(\bm{\varepsilon}^{(1)}|\cdots |\bm{\varepsilon}^{(N)}\right)\in T_{\mathbf{Z}_N}P,
\end{equation*}
which, by hypothesis, follows a multivariate distribution with zero mean and variance $\sigma^2I_{dN}$. 

Following the approach introduced in \cite{feng2021learning,hu2024structure}, we shall separately analyze these three errors. The estimation of the total error relies on both the operator representations in Section \ref{Operator representations of the kernel estimator} and the differential kernel representations in Section \ref{A global solution to the learning problem} of the estimator. Since the operator representations of the estimator are similar to those in the Euclidean case, the approximation error and sampling error can be analyzed following an approach very close to the one in \cite{hu2024structure}. Therefore, to obtain the total error in \eqref{error decomposition}, it is sufficient to analyze the noisy sampling error. As in the previous sections, we assume that $K \in C_b^3(P \times P)$ is a Mercer kernel. The proofs of the following results can be found in the appendix.

\begin{lemma}\label{operator error}
Let  $K\in C_b^3(P\times P)$ be a Mercer kernel. For any function $h \in \mathcal{H}_K$ and $0< \delta <1$, with probability at least $1-\delta$, there holds 
\begin{align*}
\|Q_N h-Q h\|_{\mathcal{H}_K} \leq \left(\sqrt{ \frac{8\log(2/\delta)}{N}}+ 
1\right)\sqrt{ \frac{\log(2/\delta)}{N}}C^2\kappa^2\|h\|_{\mathcal{H}_K}.
\end{align*}
\end{lemma}

\begin{lemma}[Error bounds of noisy sampling error]\label{noisy sampling error} 
For any $\delta>0$, with a probability of at least $1-\delta/2$, it holds that
\begin{equation*}
\begin{aligned}
\left\|\widetilde h_{\lambda,N}- \widehat{h}_{\lambda,N}\right\|_{\mathcal{H}_K} \leq \frac{\sigma \kappa}{\lambda }\sqrt{\frac{1}{N}}\left(1+\sqrt{\frac{1}{c}\log(4/\delta)}\right),
\end{aligned} 
\end{equation*} 
where the constant $c$ is related to the Hanson-Wright inequality \cite{rudelson2013hanson} (see the proof in Appendix \ref{proof of noisy sampling error}).
\end{lemma}

\begin{assumption}[{\bf Source condition}]\label{source condition}
Assume that the unknown Hamiltonian function in the Poisson system \eqref{hamilton-equation} lies in so-called {\bfi source space}, that is, 
\begin{align}\label{sou-con}
H\in \Omega_S^\gamma:=\{h\in\mathcal{H}_K\mid h=Q^\gamma \psi, \psi\in\mathcal{H}_K, \|\psi\|_{\mathcal{H}_K} < S\},  
\end{align}  
for some $\gamma\in(0,1)$ and $S>0$.
\end{assumption}

\begin{proposition}\label{sou-con-imply-perp}
    If a Hamiltonian function $H$ satisfies the source condition \eqref{sou-con}, then $H\in \mathcal{H}_{\mathrm{null}}^{\bot}$. In other words, $\Omega_S^\gamma\subset \mathcal{H}_{\mathrm{null}}^{\bot}$.
\end{proposition}

\begin{proof} 
Recall first that by Proposition \ref{BoundA}, the operator $Q=A^*A$ is a positive compact operator. Let $Q=\sum_{n=1}^{L}\lambda_n\langle \cdot, e_n\rangle e_n$ (possibly $L=\infty$) be the spectral decomposition of $Q$ with $0<\lambda_{n+1}<\lambda_{n}$ and $\{e_n\}_{n=1}^{L}$ be an orthonormal basis of $\mathcal{H}_K$. Hence for any $\gamma\in (0,1)$, we have $Q^\gamma=\sum_{n=1}^L\lambda_n^\gamma \langle \cdot,e_n\rangle_{\mathcal{H}_K}e_n$, which implies that $Q^\gamma$ is adjoint.
Notice that by the representation of the operator $Q$ given in \eqref{operatorQ}, for an arbitrary Casimir function $h\in\mathcal{H}_{\mathrm{null}}$ (if it is not an empty set), we have that $Qh=0$, which implies that $\lambda_n\langle h,e_n\rangle_{\mathcal{H}_K}=0$ for all $n=1,\ldots,L$. Hence for any $\gamma\in (0,1)$, we have $\lambda_n^\gamma\langle h,e_n\rangle_{\mathcal{H}_K}=0$ for all $n=1,\ldots,L$. Then we obtain that $Q^\gamma h=0$. Finally for arbitrary $\psi\in\mathcal{H}_K$, we compute the inner product 
$\langle Q^\gamma\psi,h\rangle_{\mathcal{H}_K}=\langle \psi,Q^\gamma h\rangle_{\mathcal{H}_K}=0$, which yields that $Q^\gamma\psi\in \mathcal{H}_{\mathrm{null}}^{\bot}$ for any $\psi\in \mathcal{H}_K$. Therefore, we conclude that $\Omega_S^\gamma\subset \mathcal{H}_{\mathrm{null}}^{\bot}$.
\end{proof}

\begin{proposition}\label{charac_source}
Let $Q=\sum_{n=1}^{L}\lambda_n\langle \cdot, e_n\rangle e_n$ (possibly $L=\infty$) be a spectral decomposition of $Q$, with an orthonormal basis $\{e_n\}_{n=1}^L$ with $0<\lambda_{n+1}\leq\lambda_n$ for all $n=1, \ldots, L-1$. Then, $\Omega_S^\gamma$ has the following characterization 
\begin{align*}
    \Omega_S^\gamma =\mathcal{H}_{S}^{\gamma} := \left\{h\in\mathcal{H}_{\mathrm{null}}^\bot\mid \sum_{n=1}^L\frac{1}{\lambda_n^{2\gamma}}|\langle h,e_n\rangle_{\mathcal{H}_K}|^2<S^2\right\}.
\end{align*}
\end{proposition}
\begin{proof}
First, we show that $\mathcal{H}_{S}^{\gamma}\subset\Omega_{S}^\gamma$. For an arbitrary $h\in \mathcal{H}_{S}^{\gamma}$, we can define a function $\psi\in \operatorname{span}\{e_n, n\in \mathbb{N}_+\}$ such that $\langle\psi,e_n\rangle_{\mathcal{H}_K}=\frac{1}{\lambda_n^{\gamma}}\langle h,e_n\rangle_{\mathcal{H}_K}$ for all $n\in\mathbb{N}^+$. One computes
\begin{align*}
\|\psi\|_{\mathcal{H}_K}^2=\sum_{n\in\mathbb{N}^+}\frac{1}{\lambda_n^{2\gamma}}|\langle h,e_n\rangle_{\mathcal{H}_K}|^2<S^2.
\end{align*}
Furthermore, since $h\in \mathcal{H}_{\mathrm{null}}^{\bot}=\operatorname{span}\{e_n, n\in \mathbb{N}_+\}$, we have that
\begin{align*}
 B^\gamma\psi = \sum_{n\in\mathbb{N}^+} \lambda_n^\gamma \langle \psi,e_n\rangle_{\mathcal{H}_K}e_n = \sum_{n\in\mathbb{N}^+} \lambda_n^\gamma \frac{1}{\lambda_n^{\gamma}}\langle h,e_n\rangle_{\mathcal{H}_K}e_n =h,
\end{align*}
which yields that $h\in\Omega_S^\gamma$. Now, we show that $\Omega_{S}^\gamma\subset\mathcal{H}_{S}^{\gamma}$. For an arbitrary $h\in \Omega_{S}^{\gamma}$, there exists a function $\psi\in\mathcal{H}_K$ with norm $\|\psi\|_{\mathcal{H}_K}<S$, such that $B^\gamma \psi=h$. Let $\psi^{\prime}$ be the projection of $\psi$ onto $\mathcal{H}_{\mathrm{null}}^{\bot}$, then $h =B^{\gamma}\psi^{\prime}$. Hence, we obtain that
\begin{align*}
  \sum_{n\in\mathbb{N}^+} \lambda_n^\gamma \langle \psi^{\prime},e_n\rangle_{\mathcal{H}_K}e_n=\sum_{n\in\mathbb{N}^+}  \langle h,e_n\rangle_{\mathcal{H}_K}e_n,   
\end{align*}
which implies that $\langle\psi^{\prime},e_n\rangle_{\mathcal{H}_K}=\frac{1}{\lambda_n^{\gamma}}\langle h,e_n\rangle_{\mathcal{H}_K}$ for all $n\in\mathbb{N}^+$ since $\lambda_n>0$. Therefore, 
\begin{align*}
    \|\psi^{\prime}\|_{\mathcal{H}_K}^2=\sum_{n\in\mathbb{N}^+}\frac{1}{\lambda_n^{2\gamma}}|\langle h,e_n\rangle_{\mathcal{H}_K}|^2<S^2.
\end{align*}
Combing Proposition \ref{sou-con-imply-perp}, we have that $h\in\mathcal{H}_{S}^{\gamma}$. Therefore, we can conclude that $\Omega_{S}^\gamma=\mathcal{H}_{S}^{\gamma}$.
\end{proof}

In Section \ref{A global solution to the learning problem}, we showed that the kernel estimator $\widehat h_{\lambda,N}\in\mathcal{H}_{\mathrm{null}}^\bot$. In this section, we will prove that for an arbitrary Hamiltonian function $H$ in the source space $\Omega_S^\gamma$, the kernel estimator $\widehat h_{\lambda,N}$ converges to the Hamiltonian $H$ as $N$ tends to infinity.

Notice that the results in Lemma \ref{noisy sampling error} are very similar to the error bounds of the noisy sampling error in \cite{hu2024structure}. Hence, combining the approximation error and the noisy sampling error bounds obtained in \cite{hu2024structure}, we immediately formulate probably approximately correct (PAC) bounds for the reconstruction error in which the regularization constant $\lambda$ remains fixed and the size of the estimation sample is allowed to vary independently from it.

\begin{theorem}[{\bf PAC bounds of the total reconstruction error}]\label{PAC bounds}
Let $ \widehat{h}_{\lambda,N}$ be the unique minimizer of the optimization problem \eqref{emp-pro}.
Suppose that Assumption \ref{source condition} holds, that is, $H\in\Omega_S^{\gamma}$ as defined in \eqref{sou-con}. 
Then, for any $\varepsilon, \delta>0$, there exist $\lambda>0$ and $n\in\mathbb{N}_{+}$ such that for all $N>n$, it holds that 
\begin{align*}
\mathbb{P}\left(\left\| \widehat{h}_{\lambda,N}-H\right\|_{\mathcal{H}_K}>\varepsilon\right) < \delta.  
\end{align*}    
\end{theorem}

Notice that by Theorem \ref{PAC bounds}, we can make the total reconstruction error arbitrarily small with the probability close to one with a fixed regularization parameter. However, this theorem provides no convergence rates. Similar to the approach followed in \cite{hu2024structure}, we now consider an adaptive regularization parameter to obtain convergence rates. To this end, we shall assume that
\begin{align}\label{dyn-sca}
\lambda\propto N^{-\alpha},\quad \alpha >0,    
\end{align}
where the symbol $\propto$ means that $\lambda$ has the order $N^{-\alpha}$ as $N\to\infty$. Then, combining Lemma \ref{noisy sampling error} with \cite[Theorem 4.7]{hu2024structure}, we obtain convergence rates for the total reconstruction error when $\alpha\in(0,\frac{1}{3})$. If we further suppose that the {\bfi coercivity condition} in the following definition holds, we can then improve the convergence integral to $\alpha\in(0,\frac{1}{2})$.

\begin{definition}[{\bf Coercivity condition}]\label{Coe-Con}
In the same setup as in Theorem \ref{Rep-Ker} we say that the {\bfi coercivity condition} is satisfied,  if  for all $h \in \mathcal{H}_K$ there exists $ c_{\mathcal{H}_K}>0$ such that
\begin{align}
\label{coercivity}
\|A h\|^2_{L^2(\mu_{\mathbf{Z}})}=\|X_h\|^2_{L^2(\mu_{\mathbf{Z}})}\geq c_{\mathcal{H}_K}\|h\|^2_{\mathcal{H}_K}. \end{align}
The largest constant $c_{\mathcal{H}_K}$ that satisfies \eqref{coercivity} is called the {\bfi coercivity constant}. 
\end{definition}

\begin{theorem}[{\bf Convergence upper rates for the total reconstruction error}]\label{convergence_theorem}
Let $ \widehat{h}_{\lambda,N}$ be the unique minimizer of the optimization problem \eqref{emp-pro}. Suppose that $\lambda $ satisfies \eqref{dyn-sca} and that $H$ satisfies the source condition \eqref{sou-con}, that is, $H\in\Omega_S^{\gamma}$. Then for all $\alpha\in(0,\frac{1}{3})$, and for any $0<\delta<1$, with probability as least $1-\delta$, it holds that
\begin{align*}
\left\| \widehat{h}_{\lambda,N}-H\right\|_{\mathcal{H}_K} \leq C(\gamma,\delta,\kappa) ~ N^{-\min\{\alpha\gamma, \frac{1}{2}(1-3\alpha)\}},   
\end{align*}  
where 
$$C(\gamma,\delta,\kappa)=\max\left\{\|B^{-\gamma}H\|_{\mathcal{H}_K}, 8\sqrt{ 4\log(8/\delta)}d^{\frac{3}{2}}\kappa^3\|H\|_{\mathcal{H}_K}\right\}.$$ 

Moreover, if the coercivity condition \eqref{coercivity} holds, then for all $\alpha\in(0,\frac{1}{2})$, and for any $0<\delta<1$, and with probability as least $1-\delta$, it holds that
\begin{align*}
\| \widehat{h}_{\lambda,N}-H\|_{\mathcal{H}_K} \leq C(\gamma,\delta,\sigma,\kappa,c,c_{\mathcal{H}_K}) ~ N^{-\min\{\alpha\gamma, \frac{1}{2}(1-2\alpha)\}},   
\end{align*}
where 
\begin{multline*}
C(\gamma,\delta,\sigma,\kappa,c,c_{\mathcal{H}_K})\\
=\max\left\{\|B^{-\gamma}H\|_{\mathcal{H}_K}, \sigma \kappa\sqrt{d} \left(1+\sqrt{\frac{1}{c}\log(4/\delta)}\right), 4\sqrt{ 2\log(8/\delta)}d\kappa^2\left(2+\kappa\sqrt{\frac{d}{c_{\mathcal{H}_K}}}\right)\|H\|_{\mathcal{H}_K}\right\}.
\end{multline*}
\end{theorem}

Theorem \ref{convergence_theorem} guarantees that the structure-preserving kernel estimator provides a function that is close to the data-generating Hamiltonian function with respect to the RKHS norm. As a consequence, we now prove that the flow of the learned Hamiltonian system will uniformly approximate that of the data-generating one. 

In the following lemma and proposition, we take the Sasaki metric on the tangent buddle $TP$ induced by the Riemannian metric $g$ on the Poisson manifold $P$. To define Sasaki metric on $TP$, let $(x,v)\in TP$ and $V,W$ be tangent vectors to $TP$ at $(x,v)$. Choose curves in $TP$
\begin{align*}
\alpha:t\mapsto(x(t),v(t)),\quad \beta:s\mapsto(y(t),w(t)),
\end{align*}
with $x(0)=y(0)=x$, $v(0)=w(0)=v$ and $V=\alpha^{\prime}(0)$, $W=\beta^{\prime}(0)$. Then the Sasaki metric is defined as 
\begin{align}\label{Sasaki-metric}
g_1(x,v)(V,W)= g(x)(d\pi(V),d\pi(W)) + g(x)\left(\frac{D}{\mathrm dt}\bigg|_{t=0}v(t), \frac{D}{\mathrm dt}\bigg|_{t=0}w(t)\right),
\end{align}
where $\pi:TP\to P$ is the canonical projection, $\frac{D}{\mathrm dt}$ stands for the covariant derivative.
As stated in Remark \ref{Functions with bounded differentials}, we shall assume that this Sasaki metric \eqref{Sasaki-metric} is complete. Then we obtain the following lemma.

\begin{lemma}\label{Lip-inequality}
Let $h\in C_b^2(P)$. Then for any $f\in C^1(P)$ and $z_1,z_2\in P$, we have
\begin{align*}
X_{h}[f](z_1)-X_{h}[f](z_2)\leq \|h\|_{C_b^2} d(z_1,z_2),   
\end{align*}
where $d$ is the geodesic distance on the Riemannian manifold $(P,g)$.
\end{lemma}

\begin{proof}
Since $h\in C_b^2(P)$, we have that $Dh\in C_b^1(TP)$. Then by the Fundamental Theorem of Calculus \eqref{fun-the}, we obtain that
\begin{align*}
Dh(y_1)-Dh(y_2) =  \int_0^1D^2h(\gamma(t))\cdot \gamma'(t)\mathrm dt\leq \|D^2h\|_{\infty}\int_0^1\|\gamma'(t)\|_{g_1}  \mathrm dt,
\end{align*}
where $\gamma:[0,1]\to TP$ is a differentiable curve connecting $\gamma(0)=y_1$ and $\gamma(1)=y_2$ in the  Riemannian manifold $(TP,g_1)$ with Sasaki metric $g_1$ induced by the Riemannian metric $g$ on $P$. Denote $\gamma(t)=(z(t),v(t))$ with $z(t)\in P$ and $v(t)\in T_{z(t)}P$ for all $t\in[0,1]$. For each $t\in[0,1]$, let $\alpha_t:[0,1]\to TP$ be a smooth curve such that $\alpha_t(0)=(z(t),v(t))$ and $\alpha'_t(0)=(z'(t),v'(t))$. Denote $\alpha_t=(a_t,b_t)$ with $a_t(s)\in P$ and $b_t(s)\in T_{a_t(s)}P$ for all $s\in[0,1]$.
Notice that 
\begin{align*}
\frac{D}{\mathrm{d}s}\bigg|_{s=0}b_{t}(s) = \nabla_{\dot{a}_t(0)}b_t(0)= \nabla_{z'(t)}v(t).   
\end{align*}
Then, by the definition of the Sasaki metric $g_1$ in \eqref{Sasaki-metric}, we obtain
\begin{align*}
\|\gamma'(t)\|_{g_1}=\|\dot{a}_t(0)\|_g + \left\|\frac{D}{\mathrm{d}s}\bigg|_{s=0}b_{t}(s)\right\|_g = \|z'(t)\|_g + \left\|\nabla_{z'(t)}v(t)\right\|_g. 
\end{align*}
{Let $y_1=(z_1,X_f(z_1))$ and $y_2=(z_2,X_f(z_2))$ for $z_1,z_2\in P$. Choose the $\{z(t),t\in[0,1]\}$ be the geodesic curve in $P$. Let $\{v(t),t\in[0,1]\}$ be a curve which is parallel along the curve $\{z'(t):t\in[0,1]\}$, i.e., $\nabla_{z'(t)}v(t)=0$ for all $t\in[0,1]$. Then we obtain that $\|\gamma'(t)\|_{g_1}=\|z'(t)\|_g$. Therefore,
\begin{align*}
X_{h}[f](z_1)-X_{h}[f](z_2) = Dh(y_1)-Dh(y_2) \leq \|D^2h\|_{\infty}\int_0^1\|z'(t)\|_{g}  \mathrm dt\leq \|h\|_{C_b^2} d(z_1,z_2),
\end{align*}
where $d$ is the geodesic distance on the Riemannian manifold $(P,g)$. The result follows.
}
\end{proof}

\begin{proposition}
[{\bf From discrete data to continuous-time flows}]
Let $\widehat{H}= \widehat{h}_{\lambda,N}$ be the structure-preserving kernel estimator of $H$ using a kernel $K\in C_b^5(P \times P)$. Let $F:[0,T] \times  P \longrightarrow P $ and $ \widehat{F}:[0,T] \times  P \longrightarrow P $ be the flows over the time interval $[0,T]$ of the Hamilton equations associated to the Hamiltonian functions $H$ and $\widehat{H}$, respectively. 
Suppose that the geodesic distance $d(\cdot,y)$ of the Poisson manifold $P$ is in $C_b^1(P)$ for each $y\in P$ and that the $C_b^1$ norm of the family $\{d(\cdot,y), y\in p\}$ is uniformly bounded, that is, $C_1=\sup_{y\in P}\|d(\cdot,y)\|_{C_b^1}<\infty$.
Then, for any initial condition $z\in P $, we have that 
\begin{align*}
d^\infty\left(F({z}),\widehat{F}({z})\right):=\max_{t\in[0,T]}d\left(F_t({z}),\widehat{F}_t({z})\right)\leq \widetilde{C} \|H-\widehat{H}\|_{\mathcal{H}_K},  
\end{align*}
with the constant $\widetilde{C}=\kappa CC_1T\exp\{\|\widehat{H}\|_{C_b^2}T\}$, where $C$ is the constant introduced in Assumption 
\ref{Con-com}.
\end{proposition}
\begin{proof}
Let $f\in C_b^1(P)$. For each $t\in[0,T]$ and ${z} \in P $, we have 
\begin{align*}
\frac{\mathrm d}{\mathrm d t} (f\circ F_t)(z) = \mathbf{d}f(F_t(z))\cdot X_H(F_t(z)) = X_H[f](F_t(z)).  
\end{align*}
By the equation \eqref{norm-control}, we have that for any $z\in P$, 
\begin{align*}
X_H[f](z)-X_{\widehat{H}}[f](z) &=Df(z)\cdot (X_H(z)-X_{\widehat{H}}(z)) \leq \|f\|_{C_b^1}\|X_{H-\widehat{H}}(z)\|_g \leq \kappa C\|f\|_{C_b^1}\|H-\widehat{H}\|_{\mathcal{H}_K}.
\end{align*}
Since $K\in C_b^5(P \times P)$, Theorem \ref{Par-Rep} implies that $\widehat{H}\in C_b^2(P)$. Then by Lemma \ref{Lip-inequality}, for any $z_1,z_2\in P$,
\begin{align*}
X_{\widehat{H}}[f](z_1)-X_{\widehat{H}}[f](z_2)\leq \|\widehat{H}\|_{C_b^2} d(z_1,z_2).   
\end{align*}
Furthermore, for each $t\in[0,T]$, we can choose $y_t\in P$ such that the point $\widehat{F}_t({z})$ is in the geodesic curve connecting $F_t({z})$ and $y_t$. Now, by hypothesis, $f:=d(\cdot, y_t)\in C_b^1(P)$. 
Therefore, we obtain that
\begin{align*}
d\left(F_t({z}),\widehat{F}_t({z})\right)&= f(F_t (z))-f( \widehat{F}_t (z)) =  \int_0^tX_ H[f](F_s({z}))-X_{\widehat{H}}[f](\widehat{F}_s({z}))\mathrm{d}s\\
&\leq \int_0^t \left|X_H[f](F_s({z}))-X_{\widehat{H}}[f](F_s({z}))\right|+\left|X_{\widehat{H}}[f](F_s({z}))-X_{\widehat{H}}[f](\widehat{F}_s({z}))\right|\mathrm{d}s\\
&\leq \kappa C\|f\|_{C_b^1}\|H-\widehat{H}\|_{\mathcal{H}_K}t+\|\widehat{H}\|_{C_b^2}\int_0^t d\left(F_s (z),\widehat{F}_s(z)\right) \mathrm{d}s.
\end{align*}
Then by the integral form of Gr\"{o}nwall's inequality, for $d\left(F_t({z}),\widehat{F}_t({z})\right)$, we obtain
\begin{align*}
d\left(F_t({z}),\widehat{F}_t({z})\right)\leq \kappa C\|d(\cdot,y_t)\|_{C_b^1}\|H-\widehat{H}\|_{\mathcal{H}_K}t \exp\left\{\|\widehat{H}\|_{C_b^2}t\right\}. 
\end{align*}
Therefore, we obtain 
\begin{align*}
d^\infty\left(F({z}),\widehat{F}({z})\right):=\max_{t\in[0,T]}d\left(F_t({z}),\widehat{F}_t({z})\right)\leq \widetilde{C} \|H-\widehat{H}\|_{\mathcal{H}_K},  
\end{align*}
where $\widetilde{C}=\kappa CC_1T\exp\{\|\widehat{H}\|_{C_b^2}T\}$. The result follows.
\end{proof}

\section{Test examples and numerical experiments}\label{numerical experiments}
This section spells out the global structure-preserving kernel estimator and illustrates its performance on two important instances of Poisson systems, namely, Hamiltonian systems on non-Euclidean symplectic manifolds and Lie-Poisson systems on the dual of a Lie algebra. More precisely, we shall derive explicit expressions for the global estimator $\widehat{h}_{\lambda,N}$ in \eqref{rep-ker} in these two scenarios.

\subsection{Learning of Lie-Poisson systems}\label{Applications to Lie-Poisson systems}
The Lie-Poisson systems introduced in Example \ref{Lie-Poisson bracket} are very important examples of Poisson systems. 
In this subsection, we shall make explicit the estimator $\widehat{h}_{\lambda,N}$ in this case.  Let $\mathfrak g$ be a Lie algebra and $\mathfrak g^*$ be its dual equipped with the Lie-Poisson bracket $\{\cdot , \cdot \}_{+}$ defined in \eqref{lie-poisson-bracket}. A straightforward computation shows that  the Lie-Poisson dynamical system associated with a Hamiltonian function $H:\mathfrak g^*\to\mathbb{R}$ is 
\begin{align}\label{Lie-Poisson systems}
\dot{\mu}=X_H(\mu)= \operatorname{ad}^*_{\frac{\delta H}{\delta \mu}}\mu,
\end{align}
where $\delta H/ \delta \mu \in \mathfrak{g} $ is the functional derivative introduced in \eqref{functional derivative}. Since $\mathfrak g$ and $\mathfrak{g}^\ast $ are vector spaces, natural Riemannian metrics can be obtained out of a non-degenerate inner product $\left\langle \cdot , \cdot \right\rangle _{\mathfrak{g}} : \mathfrak{g} \times \mathfrak{g} \rightarrow  \mathbb{R}$ on $\mathfrak g$. The non-degeneracy of $\left\langle \cdot , \cdot \right\rangle _{\mathfrak{g}} $ allows us to define musical isomorphisms $\flat : \mathfrak{g}\rightarrow \mathfrak{g}^\ast $ and $\sharp: \mathfrak{g}^\ast \rightarrow \mathfrak{g}$ by $\xi^{\flat}:= \left\langle \xi , \cdot \right\rangle _{\mathfrak{g}} $, $\xi \in \mathfrak{g} $, and $\sharp:= \flat ^{-1} $, as well as a natural non-degenerate inner product $\left\langle \cdot , \cdot \right\rangle _{\mathfrak{g}^\ast } : \mathfrak{g}^\ast  \times \mathfrak{g}^\ast  \rightarrow  \mathbb{R}$ on $\mathfrak g^\ast $ given by $\left\langle \xi^{\flat} , \eta^{\flat} \right\rangle _{\mathfrak{g}^\ast } =\left\langle \xi , \eta \right\rangle _{\mathfrak{g} }$, $\xi, \eta \in \mathfrak{g} $. Notice that
\begin{equation}
\label{relation pairings}
\left\langle \xi^{\flat}, \eta\right\rangle=\left\langle \left\langle\xi, \cdot  \right\rangle_{\mathfrak{g}}, \eta \right\rangle=\left\langle\xi, \eta  \right\rangle_{\mathfrak{g}}= \left\langle\xi^\flat, \eta^\flat\right\rangle_{\mathfrak{g}^\ast }\quad \mbox{for all $\xi, \eta \in \mathfrak{g} $.}
\end{equation}
Using these relations and \eqref{functional derivative}, we note that
\begin{equation}
\label{characgradient}
\left\langle \nabla H(\mu), \nu\right\rangle _{\mathfrak{g}^\ast }=DH(\mu)\cdot  \nu= \left\langle \nu, \frac{\delta H}{\delta \mu}\right\rangle= \left\langle \nu, \frac{\delta H}{\delta \mu}^{\flat}\right\rangle_{\mathfrak{g}^\ast}, \mbox{ for all $\nu, \mu \in \mathfrak{g} ^\ast $.}
\end{equation} 
Consequently,
\begin{equation}
\label{gradient H}
\nabla H(\mu)=\frac{\delta H}{\delta \mu}^{\flat}= \left\langle\frac{\delta H}{\delta \mu}, \cdot \right\rangle_{\mathfrak{g}} \quad \mbox{and hence $\frac{\delta H}{\delta \mu} =\nabla H(\mu)^{\sharp}$.}
\end{equation}
We now compute the compatible structure $J$ defined in \eqref{com-str} associated with the Lie-Poisson systems \eqref{Lie-Poisson systems}. 
Note first that, for any  $\mu\in\mathfrak{g}^*$ and $\eta \in \mathfrak{g} $, by \eqref{Lie-Poisson systems} and \eqref{gradient H}, we have that
\begin{align*}
J(\mu)\nabla H(\mu)=X_H(\mu)=  \operatorname{ad}^*_{\frac{\delta H}{\delta \mu}}\mu = \operatorname{ad}^*_{\nabla H(\mu)^{\sharp}}\mu.
\end{align*}
As this can be done for any Hamiltonian function and their gradients span $\mathfrak{g}^{\ast} $, this yields
\begin{align}\label{Lie-Poisson comp-str}
J(\mu) \nu = \operatorname{ad}^*_{\nu ^{\sharp}}\mu, \quad \text{ for all } \nu\in{\mathfrak g}^\ast .    
\end{align}
To compute the estimator to $\widehat{h}_{\lambda,N}$ one first evaluates $J$ explicitly using \eqref{Lie-Poisson comp-str}, we then set 
$$\mathbb{J}_N=\operatorname{diag}\{J(\mathbf{Z}^{(1)}),\ldots,J(\mathbf{Z}^{(N)})\},$$
and substitute it into \eqref{rep-ker}. More explicitly, note that the generalized differential Gram matrix $G_N$ defined in \eqref{general-Gram} reduces in this case to
\begin{align*}
G_N \mathbf{c}=X_{g_N(\mathbf{c},X_{K_{\cdot}}(\mathbf{Z}_N))}(\mathbf{Z}_N)=X_{\mathbf{c}^\top \mathbb{J}_N\nabla_1K(\mathbf{Z}_N,\cdot)}(\mathbf{Z}_N)=\mathbb{J}_N\nabla_{1,2}K(\mathbf{Z}_N,\mathbf{Z}_N)\mathbb{J}_N^\top\mathbf{c}.    
\end{align*}
for all $\mathbf{c}\in T_{\mathbf{Z}_N}\mathfrak g^*$. 
Therefore, the estimator in Theorem \ref{Rep-Ker} turns to
\begin{align}
\widehat{h}_{\lambda,N}&= g_N( (G_N+\lambda NI)^{-1}\mathbf{X}_{\sigma^2,N},X_{K_{\cdot}}(\mathbf{ Z}_N,\cdot))\notag\\
&=\mathbf{X}_{\sigma^2,N}^\top \left(\mathbb{J}_N\nabla_{1,2}K(\mathbf{Z}_N,\mathbf{Z}_N)\mathbb{J}_N^\top+\lambda NI\right)^{-1}\mathbb{J}_N\nabla_1K(\mathbf{Z}_N,\cdot).\label{lie_poisson_estimator}
\end{align}

In what follows, we demonstrate the effectiveness of this learning scheme using two examples. The first one is the three-dimensional classical rigid body, which models the motion of solid bodies without deformation. The second example is the dynamics of the underwater vehicle, which has a nine-dimensional underlying phase space. We emphasize that the state spaces of Lie-Poisson systems are duals of Lie algebras and, hence, are isomorphic to Euclidean spaces; in such a framework, the Gaussian kernel is a natural choice of universal kernels. As such, even when $H$ does not belong to $\mathcal{H}_K$, we can hope for learning a proxy function in $\mathcal{H}_K$ that is close to a ground-truth Hamiltonian $H$.\par

We will see in these two examples that $\widehat{h}_{\lambda, N}$ can only recover the ground-truth Hamiltonian function up to a Casimir function. We emphasize that the estimator $\widehat{h}_{\lambda, N}$ is always unique, and the possible difference by a Casimir function is present because $\widehat{h}_{\lambda,N}\in \mathcal{H}^{\bot}_{\mathrm{null}}$, while the ground-truth $H$ may not belong to $\mathcal{H}^{\bot}_{\mathrm{null}}$, and hence the difference between $H$ and $\widehat{h}_{\lambda,N}$ could converge not to zero, but to a Casimir function. An analytic reason for this is that the Hamiltonian functions that we shall be using for data generation are quadratic and violate the source condition \eqref{sou-con} (see Example 2.9 in \cite{hu2024structure} for a characterization of the RKHS of the Gaussian kernel). To better illustrate this point, we provide a third example in which the Hamiltonian $H\in\mathcal{H}^{\bot}_{\mathrm{null}}$ satisfies the source condition \eqref{sou-con} for sufficiently large $S>0$ ($S$ is the constant introduced in \eqref{sou-con}; see also Proposition \ref{charac_source}). In that case, according to Theorem \ref{convergence_theorem}, $\widehat{h}_{\lambda,N}$ converges to $H$ in the $\mathcal{H}_K$ norm, and the difference converges to $0$ as $N\rightarrow \infty$. We will see that, numerically, $\widehat{h}_{\lambda,N}$ accurately recovers the ground-truth Hamiltonian.

\subsubsection{The rigid body}

Rigid body motion has been introduced in Example \ref{rigid_body_example}. Using \eqref{Lie-Poisson comp-str}, the compatible structure $J$ associated with the Euclidean metric in ${\Bbb R}^3$ can be written in this case as
\begin{align*}
    J(\Pi) = \hat{\Pi}.
\end{align*}

\noindent{\bf Numerical results} \quad We pick $\mathbb{I}=\operatorname{diag}\{1,10,0.1\}$ and consider the standard Gaussian kernel with parameter $\eta$ on $\mathfrak{so}(3)^*\cong \mathbb{R}^3$. In order to generate the data, we set $N=500$ and sample $N$ states from the uniform distribution on the cube $[-1,1]^3$. We then obtain the Hamiltonian vector fields at these states that, this time, we do not perturb with noise ($\sigma=0$). In view of \eqref{dyn-sca}, we set the regularization parameter $\lambda=c\cdot N^{-\alpha}$ with $\alpha=0.4$. We then perform a grid search for the parameters $\eta$ and $c$ using $5$-fold cross-validation, which results in the optimal parameters $\eta=2.5$ and $c=2.5\cdot 10^{-5}$. Given that $\mathfrak{so}(3)^*\cong \mathbb{R}^3$ is $3$-dimensional, we fix the third dimension, that is, $\Pi_3=0$, and visualize both the ground-truth Hamiltonian functions $H$ and the reconstructed Hamiltonian functions $\widehat{h}_{\lambda,N}$ in the dimensions $(\Pi_1,\Pi_2)$, see Figure \ref{rigid_body} (a) and (b). The ground-truth Hamiltonian $H$ is quadratic and violates the source condition \eqref{sou-con}, so $\widehat{h}_{\lambda,N}$ and $H$ could differ by a Casimir function. 

\noindent From Example \ref{rigid_body_example}, all possible Casimir functions are of the form $\Phi(\|\Pi\|^2)$ for some function $\Phi:\mathbb{R}\rightarrow \mathbb{R}$. We try to compensate for the difference via a simple Casimir of the form $a\cdot \|\Pi\|^2$, for some constant $a\in \mathbb{R}$. After trial and error, we select $a=1.85$ and visualize the ``Casimir-corrected'' estimator, that is $\widehat{h}_{\lambda, N}+1.85\cdot \|\Pi\|^2$, see Figure \ref{rigid_body} (c). We also visualize the mean-square error of the predicted Hamiltonian vector fields in a heatmap, see Figure \ref{rigid_body} (d).

\begin{figure}[htp]
    \centering
    \subfigure[]{\includegraphics[width=0.4\textwidth]{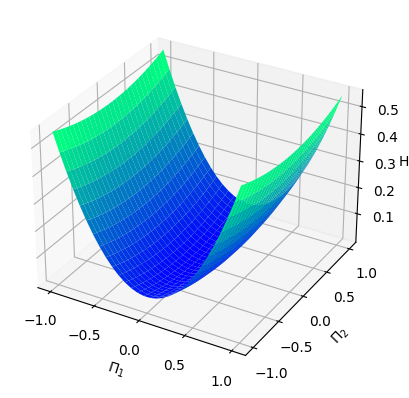}}
    \subfigure[]{\includegraphics[width=0.4\textwidth]{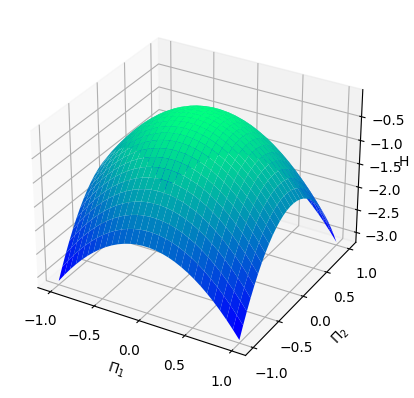}} 
    \medskip
    \subfigure[]{\includegraphics[width=0.4\textwidth]{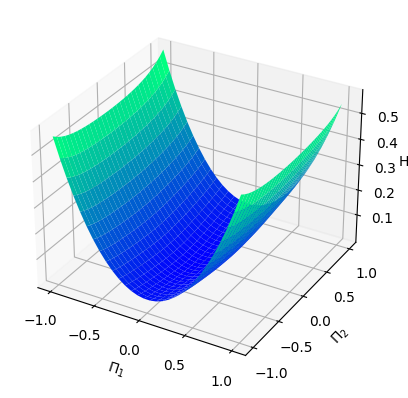}} 
    \subfigure[]{\includegraphics[width=0.4\textwidth]{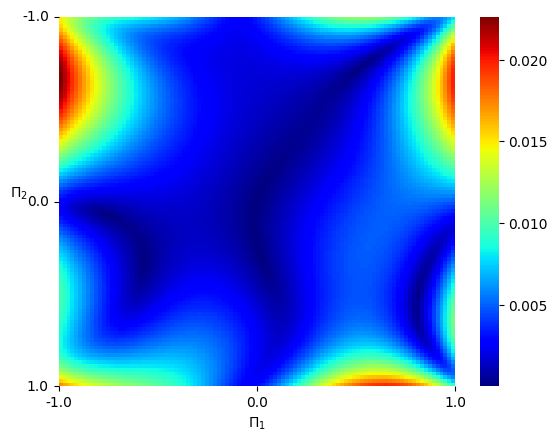}}
    \caption{Rigid body dynamics: (a) Ground-truth Hamiltonian (b) Learned Hamiltonian with $N=500$ (c) Learned Hamiltonian adjusted by a Casimir function (d) Squared error of the predicted Hamiltonian vector field } 
    \label{rigid_body}
\end{figure}

\noindent{\bf Analysis} \quad In this setting, the Hamiltonian function $H$ is a quadratic function, which, according to Example 2.9 in \cite{hu2024structure}, does not belong to $\mathcal{H}_K$. In particular, the source condition \eqref{sou-con} cannot be satisfied. Therefore, even though the Gaussian kernel is universal, the proxy function in $\mathcal{H}_K$, which approximates the ground-truth Hamiltonian $H$, can only be learned at best up to a Casimir function. That is why we choose to manually correct the estimator by adding a Casimir function. From the numerical results, we see that without the ``Casimir correction'', $H$ and $\widehat{h}_{\lambda, N}$ can be very different, despite the fact that the Hamiltonian vector fields are very well replicated.

\subsubsection{Underwater vehicle dynamics}
The Lie-Poisson underwater vehicle dynamics has been introduced in Example \ref{Underwater vehicle}. Using also a Euclidean metric, \eqref{Lie-Poisson comp-str}, the compatible structure $J$ is 
\begin{align*}
J(\Pi, \mathrm{Q}, \Gamma)=  \left(\begin{array}{ccc}
\hat{\Pi} & \hat{\mathrm{Q}} & \hat{\Gamma} \\
\hat{\mathrm{Q}} & 0 & 0 \\
\hat{\Gamma} & 0 & 0
\end{array}\right).
\end{align*}

\noindent{\bf Numerical results} \quad We pick $m=1$, $g=9.8$, $I = \operatorname{diag}\{1,2,3\}$, $M=\operatorname{diag}\{3,2,1\}$ and $r_G=(1,1,1)^{\top}$. We consider the standard Gaussian kernel with parameter $\eta$ on $\mathfrak{so}(3)^*\times {\mathbb{R}^{3}}^*\times {\mathbb{R}^{3}}^*\cong \mathbb{R}^9$. We set $N=400$ and sample $N$ states from the uniform distribution on the cube $[-1,1]^9$, and then obtain the Hamiltonian vector fields at these states. We set the regularization parameter $\lambda=c\cdot N^{-\alpha}$ with $\alpha=0.4$. We then perform a grid search of the parameters $\eta$ and $c$ using $5$-fold cross-validation, which results in the optimal parameter $\eta=5$ and $c=7.5\cdot 10^{-6}$. Given that $\mathfrak{so}(3)^*\times {\mathbb{R}^{3}}^*\times {\mathbb{R}^{3}}^*$ is $9$-dimensional, we fix the 3rd to 9th dimensions, that is, $\Pi_3=0$, $Q=\Gamma={\bf 0}$, and visualize both the ground-truth Hamiltonian functions $H$ and the reconstructed Hamiltonian functions $\widehat{h}_{\lambda,N}$ in the dimensions $(\Pi_1,\Pi_2)$, see Figure \ref{under_water_vehicle} (a) and (b). The ground-truth Hamiltonian $H$ does not satisfy the source condition \eqref{sou-con}, so $\widehat{h}_{\lambda,N}$ and $H$ generically differ by a Casimir function. We also visualize the mean-square error of the predicted Hamiltonian vector fields with a heatmap, see Figure \ref{under_water_vehicle} (c).

\begin{figure}[htp]
    \centering
    \subfigure[]{\includegraphics[width=0.32\textwidth]{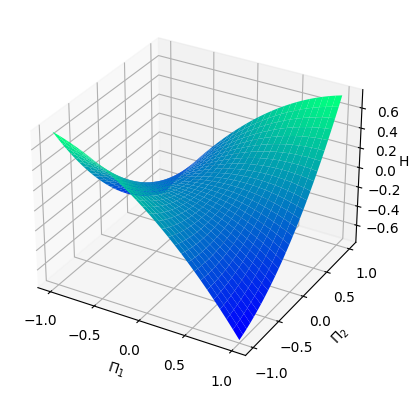}}
    \subfigure[]{\includegraphics[width=0.32\textwidth]{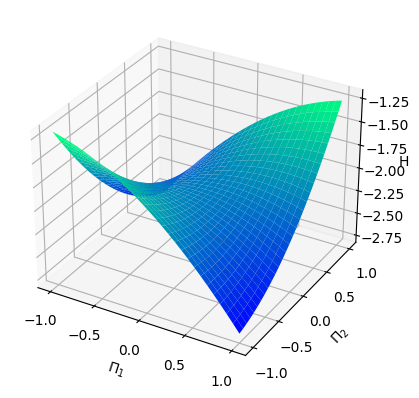}} 
    \medskip
    \subfigure[]{\includegraphics[width=0.32\textwidth]{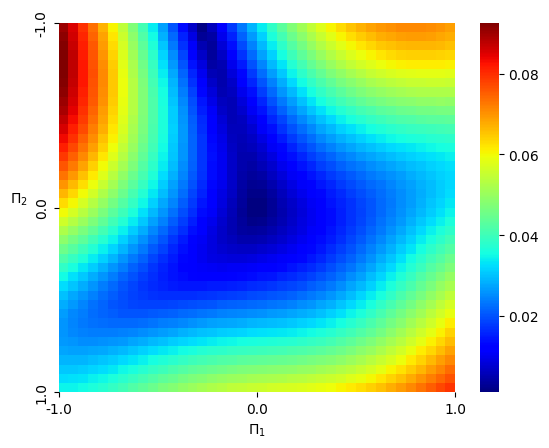}}
    \caption{Underwater Vehicle: (a) Ground-truth Hamiltonian (b) Learned Hamiltonian with $N=400$ (c) Squared error of the predicted Hamiltonian vector field } 
    \label{under_water_vehicle}
\end{figure}

\noindent{\bf Analysis} \quad In this setting, the Hamiltonian function $H$ is a quadratic function, which, according to Example 2.9 in \cite{hu2024structure}, does not belong to $\mathcal{H}_K$. In particular, the source condition \eqref{sou-con} cannot be satisfied. Therefore, even though the Gaussian kernel is universal, the proxy function in $\mathcal{H}_K$, which approximates the ground-truth Hamiltonian $H$, can only be learned at best up to a Casimir function. From Example \ref{Underwater vehicle}, all possible Casimir functions are of the form $\Phi\circ (C_1,\ldots, C_6)$ for some function $\Phi:\mathbb{R}^6\rightarrow \mathbb{R}$. To determine such a function $\Phi$, one might consider running a polynomial regression to approximate $\Phi$. From the $H$-axis, we see that, although the shapes of the Hamiltonian look alike, the exact Hamiltonian value still differs significantly by a constant, which is a trivial Casimir function. On the other hand, we notice that the Hamiltonian vector fields are very well replicated.

\subsubsection{Exact recovery of Hamiltonians in the RKHS}

In the rigid body and the underwater vehicle examples, we saw that even though the Hamiltonian vector field can be learned very well, the data-generating Hamiltonian function can only be estimated up to a Casimir function due to the violation of the source condition. We now aim to manually pick a ground-truth Hamiltonian that satisfies the source condition and reconstruct the Hamiltonian function accurately from the data. To find such a Hamiltonian, we use \cite{minh2010some}, which provides an orthonormal basis of the RKHS of the Gaussian kernel. It turns out that exactly one of the basis functions is a Casimir and any other basis function belongs to $\mathcal{H}^{\bot}_{\mathrm{null}}$ according to Example \ref{rigid_body_example}. Moreover, by Proposition \ref{charac_source}, the basis functions, except the Casimir basis element, satisfies the source condition for sufficiently large $S>0$. In view of the above, we consider the standard Gaussian kernel on $\mathbb{R}^3$ with fixed parameter $\eta=2$ and choose the Hamiltonian function as
\begin{equation}\label{Gaussian_RKHS_Ham}
    H({\bf x}) = x_1^2x_2^3 e^{-\frac{\|{\bf x}\|^2}{\eta^2}}.
\end{equation}

\noindent The Lie-Poisson structure is the same as the one we used for the rigid body, with the only difference in the experiment being the Hamiltonian function.

\noindent{\bf Numerical results} \quad To guarantee the matching of the RKHS, we consider the standard Gaussian kernel with the fixed parameter $\eta=2$ (the same as the ground-truth) on $\mathbb{R}^3$. We set $N=500$ and sample $N$ states from the uniform distribution on the cube $[-1,1]^3$, and then obtain the Hamiltonian vector fields at these states. In view of \eqref{dyn-sca}, we set the regularization parameter $\lambda=c\cdot N^{-\alpha}$ with $\alpha=0.4$. We then perform a grid search of the parameter $c$ using $5$-fold cross-validation, which results in the optimal parameter $c=7.5\cdot 10^{-6}$. Given that $\mathfrak{so}(3)^*\cong \mathbb{R}^3$ is $3$-dimensional, we fix the third dimension, that is, $x_3=0$, and visualize both the ground-truth Hamiltonian functions $H$ and the reconstructed Hamiltonian functions $\widehat{h}_{\lambda,N}$ in the dimensions $(x_1,x_2)$, see Figure \ref{kernel_sections} (a) and (b).

\noindent{\bf Analysis} \quad In this setting, the Hamiltonian function $H$ is chosen to satisfy the source condition \eqref{sou-con}. In this case, both the ground-truth Hamiltonian $H$ and the estimator $\widehat{h}_{\lambda, N}$ belong to $\mathcal{H}^{\bot}_{\mathrm{null}}$. The possibility of the presence of Casimir functions is eliminated, and hence by Theorem \ref{convergence_theorem}, $\widehat{h}_{\lambda, N}$ converges to $H$ in the RKHS norm. As shown in Figure \ref{kernel_sections} (c), the Hamiltonian function can replicated even without ``Casimir correction'', in contrast with the rigid body example.

\begin{figure}[htp]
    \centering
    \subfigure[]{\includegraphics[width=0.45\textwidth]{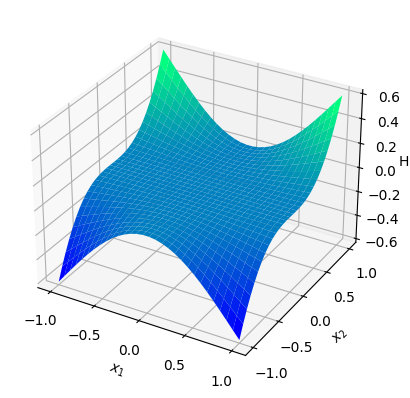}}
    \subfigure[]{\includegraphics[width=0.45\textwidth]{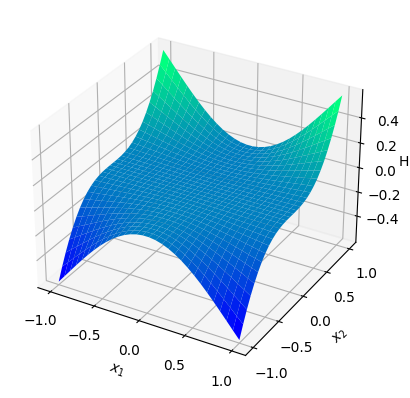}}
    \medskip
    \subfigure[]{\includegraphics[width=0.45\textwidth]{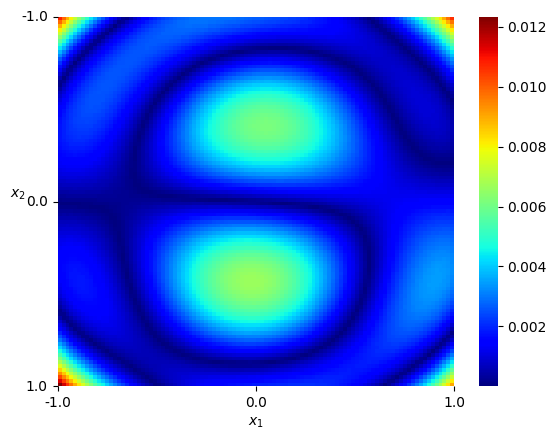}}
    \subfigure[]{\includegraphics[width=0.45\textwidth]{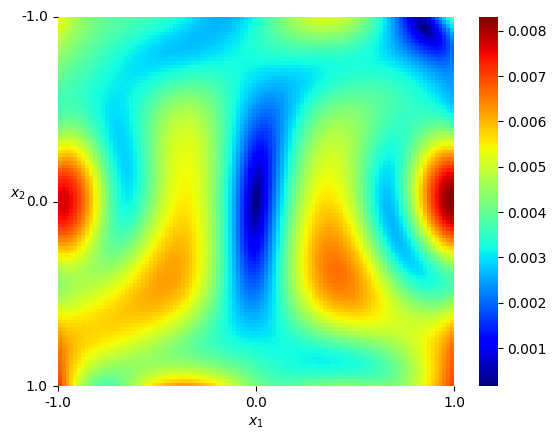}}
    \caption{Gaussian kernel sections: (a) Ground-truth Hamiltonian (b) Learned Hamiltonian with $N=500$ (c) Absolute error of the predicted Hamiltonian function (d) Squared error of the predicted Hamiltonian vector field } 
    \label{kernel_sections}
\end{figure}

\subsection{Learning on non-Euclidean symplectic manifolds}
\label{Hamiltonian systems on symplectic manifolds}
As mentioned in Example \ref{symplectic-bracket}, any symplectic manifold is a Poisson manifold. Let $M$ be a symplectic manifold with a symplectic form $\omega$. Then the Hamiltonian vector field associated with Hamiltonian function $H$ defined in \eqref{hamiltonian-vector-field} is
\begin{align*}
X_H=\omega^{\sharp}(\bm{\mathrm{d}}H)=J\nabla H,
\end{align*}
where $\omega^{\sharp}:T^*M\to TM$ is the bundle isomorphism induced by the symplectic form $\omega$ and $J:TM\to TM$ given by
\begin{align}\label{sym-com-str}
J(x) v: =\omega_x^\sharp\left(g_x^\flat(v)\right), \quad \forall x\in M, v\in T_xM,  
\end{align}
with the bundle isomorphism $g^\flat:TM\to T^*M$ induced by the Riemannian metric $g$. Unlike in the more general Poisson setting, $J$ is a bundle isomorphism in the symplectic setup. 
We aim to learn the Hamiltonian function $H$ out of the equation
\begin{align}\label{Ham-sym-sys}
\dot{z} =J\nabla H(z).
\end{align}
Following Theorem \ref{Rep-Ker}, the kernel estimator in \eqref{rep-ker} is
\begin{equation*}
\widehat{h}_{\lambda,N}= g_N( \widehat {\bf c},X_{K_{\cdot}}(\mathbf{ Z}_N,\cdot))=g_N( \widehat {\bf c},\mathbb{J}_N\nabla K(\mathbf{ Z}_N,\cdot)),
\end{equation*}
where $\widehat{\mathbf{c}}\in T_{\mathbf{Z}_N}M$ is given by
\begin{align*}
\widehat{\mathbf{c}}=(G_N+\lambda NI)^{-1}\mathbf{X}_{\sigma^2,N}.
\end{align*}

In applications, it is important to compute the kernel estimator $\widehat h_{\lambda,N}$ in pre-determined local patches that cover the manifold $M$. By the local representations \eqref{local-com-str} and \eqref{local-Gram}, the compatible structure $J$
and the generalized Gram matrix $G_N$ of the Hamiltonian system \eqref{Ham-sym-sys} are 
\begin{align*}
    J^{ij} = -\omega^{ik}g_{kj}, \quad G_N^{i,j} = \omega(\mathbf{Z}^{(i)})\partial_{1,2}K(\mathbf{Z}^{(i)},\mathbf{Z}^{(j)})\omega^\top(\mathbf{Z}^{(j)})g^{(j)}.
\end{align*}

In what follows, we demonstrate the effectiveness of this learning scheme on two examples that share the same symplectic manifold $M= S ^2 \times S ^2$ with the same symplectic form $\omega$ and the same Riemannian metric $g$, while governed by two different Hamiltonian functions $H$. The first Hamiltonian function is the restricted Euclidean 3-norm to the manifold $M$, and the second one describes the two-vortex dynamics. We will see that the first Hamiltonian is ``well-behaved'' and the second Hamiltonian exhibits singularities.

We emphasize that, technically, there are three challenges in relation to learning these systems. First, it is difficult to verify that a given Hamiltonian function belongs to the RKHS $\mathcal{H}_K$ of a pre-determined kernel $K$ on the manifold due to the lack of results in the literature on the characterization of functions in this type of RKHS on manifolds. Second, in case the Hamiltonian function does not lie in the RKHS of the selected kernel $K$, it is important to investigate the universality properties of kernels on manifolds to hope for at least a good approximation. Third, the Hamiltonian function of the two-vortex system exhibits singularities, which pose even more difficulties to recover $H$ from $X_H$, since the vector field diverges at singularities. We will see that even in this example, $\widehat{h}_{\lambda,N}$ still recovers the qualitative behaviors of the Hamiltonian function. 

\subsubsection{Metric, symplectic, and compatible structures in $S ^2 \times  S ^2  $}

Let $\overline{\mathcal{P}}=S^2 \times S^2$, the product of $2$ copies of unit $2$-spheres in $\mathbb{R}^3$. The phase space for the $2$-vortex problem is $\mathcal{P}=\overline{\mathcal{P}} \backslash \Delta$, where $\Delta$ is the diagonal, that is,
\begin{align*}
\Delta=\left\{\mathbf{x}=(x_1, x_2) \mid x_1=x_2\right\}.
\end{align*}
\noindent
The symplectic structure on $\mathcal{P}$ is given by
\begin{align}\label{symp-vort}
\omega_{\mathcal{P}}=\lambda_1 \pi_1^* \omega_{S^2}+\lambda_2 \pi_2^* \omega_{S^2},    
\end{align}
where $\pi_j$ is the Cartesian projection on to the $j$-th factor, $\omega_{S^2}$ the natural symplectic form on $S^2$, and $\lambda_j$ the vorticity of the $j$ th vortex. The Poisson structure is given by
$$
\{f, g\}= \lambda_1^{-1}\left[\mathrm{~d}_1 f, \mathrm{~d}_1 g, x_1\right] + \lambda_2^{-1}\left[\mathrm{~d}_2 f, \mathrm{~d}_2 g, x_2\right],
$$
where $\left[\mathrm{d}_j f, \mathrm{~d}_j g, x_j\right]$ is the triple product of $[a, b, c]=a \cdot(b \times c)$ for $a, b, c \in \mathbb{R}^3$. 

\noindent{\bf Local coordinates.}\quad We first consider the local structure of the unit sphere $S^2$ as a Riemannian manifold.
We parametrize the unit sphere $S^2$ in terms of two angles, that is, $\theta$ (the colatitude) and $\varphi$ (the longitude), 
and use two patches $(V_1,\phi_1)$ and $(V_2,\phi_2)$ to cover the whole sphere $S^2$. The first patch is given by 
\begin{equation*}
\begin{cases}
V_1 =\{(\theta,\varphi)\mid  \theta\in(0,\pi),\varphi\in(0,2\pi)\},\\  
\phi_1:(\theta,\varphi) \mapsto (\sin\theta\cos\varphi,\sin\theta\sin\varphi,\cos\theta),
\end{cases} 
\end{equation*}
which implies that the semicircle from the north pole to the south pole
lying in the $xz$-plane is not covered by this chart. In order to cover the whole sphere, we need another chart obtained by deleting the semicircle in the
$xy$-plane from $(0,1,0)$ to $(0,1,0)$ and with $x\leq 0$. In view that the patch $(V_1,\phi_1)$ is already topologically dense in $S^2$, we omit the explicit parameterization of $(V_2,\phi_2)$ for convenience and focus on the first patch.

\iffalse
\begin{equation*}
\begin{cases}
V_1 =\{(\theta,\varphi)\mid  \theta\in(0,\pi),\varphi\in(0,2\pi)\},\\  
\phi_1:(\theta,\varphi) \mapsto (\sin\theta\cos\varphi,\sin\theta\sin\varphi,\cos\theta).
\end{cases} 
\begin{cases}
V_2 =\{(\theta,\varphi)\mid  \theta\in(\frac{1}{2}\pi,\frac{5}{2}\pi),\varphi\in(\frac{1}{2}\pi,\frac{3}{2}\pi),\\  
\phi_2:(\theta,\varphi) \mapsto \begin{cases}
(\sin\theta\cos\varphi,\sin\theta\sin\varphi,\cos\theta),\theta\in(\frac{1}{2}\pi,\pi)\cup(2\pi,\frac{5}{2}\pi)\\
(-\sin\theta\cos\varphi,-\sin\theta\sin\varphi,\cos\theta),\theta\in[\pi,2\pi]\\
\end{cases}
\end{cases}
\end{equation*}
\fi

The unit sphere $S^2$ admits a natural symplectic form given by its area form that is locally expressed by $\omega_{S^2}=\sin\theta \mathrm{d}\theta\wedge \mathrm{d}\varphi$ and a Riemannian metric $g_{S^2}=\mathrm{d}\theta^2+\sin^2\theta \mathrm{d}\varphi^2$ inherited from the ambient Euclidean metric in ${\Bbb R}^3 $. The corresponding matrix forms in the coordinate patch $(V_1,\phi_1)$ can be expressed by 
\begin{equation*}
\omega_{S^2}(\theta,\varphi)= \begin{bmatrix}
    0 & \sin\theta\\
    -\sin\theta &0
\end{bmatrix},   \quad\text{ and } \quad g_{S^2}(\theta,\varphi)= \begin{bmatrix}
    1 & 0\\
    0 &\sin^2\theta
\end{bmatrix}.  
\end{equation*}
\noindent
These structures induce local symplectic and Riemannian metrics on the patch $(V_1 \times V_1,\phi_1\times \phi_1)$ (that is, two copies of $(V_1,\phi_1)$) of the product manifold $S ^2 \times S ^2 $ given by
\begin{align*}
\omega_{\mathcal{P}}(\theta_1,\varphi_1,\theta_{2},\varphi_2) = \operatorname{diag}\{\lambda_1\omega_{S^2}(\theta_1,\varphi_1),\lambda_N\omega_{S^2}(\theta_2,\varphi_2)\},
\end{align*}
where $(\theta_1,\varphi_1,\theta_{2},\varphi_2) \in V_1 \times V_1$. The matrix form of the product metric $g_{\mathcal{P}}$ is 
\begin{align*}
g_{\mathcal{P}}(\theta_1,\varphi_1,\theta_{2},\varphi_2) = \operatorname{diag}\{g_{S^2}(\theta_1,\varphi_1), g_{S^2}(\theta_2,\varphi_2)\},
\end{align*}
hence, the compatible structure $J_{\mathcal{P}}$ of $2$-vortex system is 
\begin{align*}
J_{\mathcal{P}}(\theta_1,\varphi_1,\theta_{2},\varphi_2) = \operatorname{diag}\{\lambda_1\omega_{S^2}(\theta_1,\varphi_1)g_{S^2}(\theta_1,\varphi_1),\lambda_2\omega_{S^2}(\theta_2,\varphi_2)g_{S^2}(\theta_2,\varphi_2)\}.  
\end{align*}

\subsubsection{A universal kernel and data sampling on $S ^2 \times  S ^2  $}

\noindent {\bf Universal kernel.} \quad Consider the RKHS ${\mathcal H} _K  $ defined on a Hausdorff topological space ${\cal X}  $. Let  ${\cal Z}\subset {\cal X}$ be an arbitrary but fixed compact subset of ${\cal X} $ and denote by ${\mathcal H} _K  ({\cal Z})\subset {\mathcal H} _K$ the completion in the RKHS norm of the span of kernel sections determined by the elements of ${\cal Z} $. We write this as:
\begin{equation}
\label{universal kernel}
{\mathcal H} _K  ({\cal Z})= \overline{{\rm span} \left\{K _z\mid z \in {\cal Z}\right\}}.
\end{equation} 
Denote now by $\overline{\mathcal{H}_{K}({\cal Z})}$ the uniform closure of ${\mathcal H} _K  ({\cal Z}) $.
A kernel $K$ is called universal if  for any compact subset $\mathcal{Z}\subset {\cal X}$, we have that $\overline{\mathcal{H}_{K}({\cal Z})}=C({\cal Z})$, with $C({\cal Z})$ the set of real-valued continuous functions on ${\cal Z} $. Equivalently, this implies that for any $\varepsilon>0$, and any function $f\in C(\mathcal{Z})$, there exists a function $g\in\mathcal{H}_{K}({\cal Z})$, such that $\|f-g\|_{\infty}<\varepsilon$.  Many kernels that are used in practice are indeed universal  \cite{micchelli2006universal,steinwart2001influence}, e.g., the Gaussian kernel on Euclidean space. For non-Euclidean spaces there is still a lack of literature in this field. The universality of kernels is a crucial feature in recovering unknown functions using the corresponding RKHS. 
In our learning approach, we shall consider the kernel on $S^2\times S^2$ defined by
\begin{align}\label{numerical kernel}
K_{\eta}(x,y):= \exp\left\{-\frac{1}{\eta^2}|\iota(x)-\iota(y)|^2\right\}, \quad \text{ for all}\quad x,y\in S^2\times S^2,
\end{align}
where $\eta>0$ is a constant and  $\iota:S^2\times S^2\to \mathbb{R}^3\times \mathbb{R}^3$ is the injection map. We emphasize that since $\iota$ is continuous and injective, by \cite[Theorem 2.2]{christmann2010universal},  the kernel $K_\eta$ in \eqref{numerical kernel} is universal for all $\eta>0$ on $S^2\times S^2$.

\noindent{\bf Sampling method.} \quad Note that the total surface of the unit sphere is
\begin{align*}
    \int_0^{2\pi}\int_0^{\pi}\sin\theta \mathrm d \theta \mathrm d\varphi=4\pi.
\end{align*}
The cumulative distribution function can be computed as $\Theta,\Phi$ such that
\begin{align*}
\mathbb{P}(\Theta\leq \theta,\Phi\leq \varphi) =\frac{1}{4\pi}\int_0^{\varphi}\int_0^{\theta}\sin t ~dt ds= \frac{1}{4\pi} \varphi(1-\cos \theta), 
\end{align*}
which factorizes into the product of the margins
\begin{align*}
\mathbb{P}(\Theta\leq \theta)=\frac{1}{2}(1-\cos\theta),\quad \mathbb{P}(\Phi\leq \varphi)   =\frac{1}{2\pi}\varphi.
\end{align*}
One can hence obtain a uniform distribution on $S^2$ by sampling $U_1,U_2$ uniformly on the interval $(0,1)$, and compute 
\begin{align*}
\Theta = \arccos(2U_1-1),\quad \Phi = 2\pi U_2.
\end{align*}
This extends to a uniform distribution on the product space $\mathcal{P}$, which we adopt to generate training data. 

\subsubsection{Exact recovery of Hamiltonians using a universal kernel}

In our first example, we consider as Hamiltonian function on $\overline{\mathcal{P}}$ the Euclidean 3-norm on $\mathbb{R}^3\times \mathbb{R}^3$.

\noindent{\bf Formulation} \quad We choose the globally defined Hamiltonian function 
\begin{equation*}
    H^{\prime}\left(x_1, x_2\right)= (\|x_1\|^3_3+\|x_2\|^3_3)^{\frac{1}{3}},
\end{equation*}
where $\|\cdot \|_3$ is the usual 3-norm on $\mathbb{R}^3$ defined as $\|(a,b,c)\|_3 = (a^3+b^3+c^3)^{\frac{1}{3}}$.
In local coordinates, we focus on the patch 
$(V_1 \times V_1,\phi_1\times \phi_1)$ and so that the Hamiltonian become
\begin{align*}
H(\theta_1,\varphi_1,\theta_{2},\varphi_2) = H^{\prime}(\phi_1(\theta_1,\varphi_1), \phi_1(\theta_2, \varphi_2)).
\end{align*}
We shall call any Hamiltonian function constructed in this way a spherical 3-norm because the domain is on the products of spheres. The restricted Gaussian kernel $K_{\eta}$ as in \eqref{numerical kernel}  will be adopted for learning, also focused on the patch $(V_1 \times V_1,\phi_1\times \phi_1)$ on $\overline{\mathcal{P}}$.

\noindent{\bf Numerical results} \quad 
We set $N=1200$ and sample $N$ states from the uniform distribution on $\overline{\mathcal{P}}$, and then obtain the Hamiltonian vector fields at these states. In view of \eqref{dyn-sca}, we set the regularization parameter $\lambda=c\cdot N^{-\alpha}$ with $\alpha=0.4$. We perform hyperparameter tunning and selected parameter $\eta = 0.9$ and $c=1e^{-4}$. Given that $\overline{\mathcal{P}}$ is $4$-dimensional, we fix two of the dimensions, that is, $(\theta_2,\varphi_2)=(\frac{\pi}{2},0)$, $(\theta_1,\theta_2) = (\frac{\pi}{3},\frac{\pi}{3})$, and visualize both the ground-truth Hamiltonian functions $H$ and the reconstructed Hamiltonian functions $\widehat{h}_{\lambda,N}$ in the remaining $2$ dimensions, see Figure \ref{sphere_norm} (a)(b) and (c)(d). We also visualize the absolute error of the predicted Hamiltonian function in a heatmap, after a vertical shift to offset any potential constant, see Figure \ref{sphere_norm} (e)(f). We also picture the Hamiltonian function globally on the first sphere and the second sphere respectively as heatmaps, while fixing $(\theta_2,\varphi_2)=(\frac{\pi}{2},0)$ and $(\theta_1,\varphi_1)=(\frac{\pi}{2},0)$ respectively on the other sphere, see Figure \ref{sphere_norm_heatmap}. Lastly, we picture the absolute error of the learned Hamiltonian in the setting of only one sphere and $N$ increased to 2000, see Figure \ref{single_sphere_norm}, which demonstrates the enhanced quantitative performance of the algorithm as the number of data becomes large.

\begin{figure}[htp]
    \centering
    \subfigure[]{\includegraphics[width=0.4\textwidth]{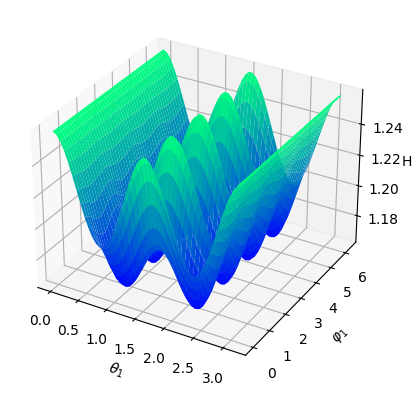}}
    \subfigure[]{\includegraphics[width=0.4\textwidth]{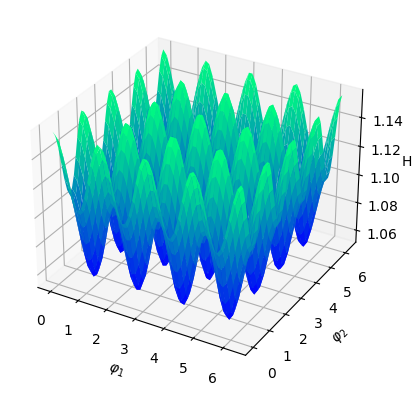}} 
    \medskip
    \subfigure[]
    {\includegraphics[width=0.4\textwidth]{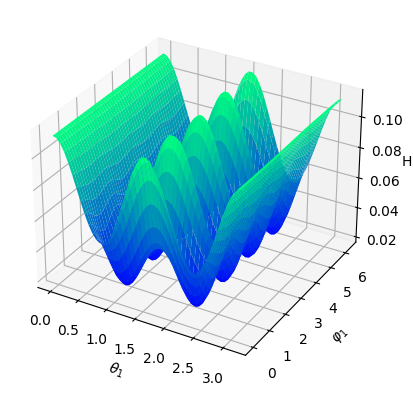}}
    \subfigure[]{\includegraphics[width=0.4\textwidth]{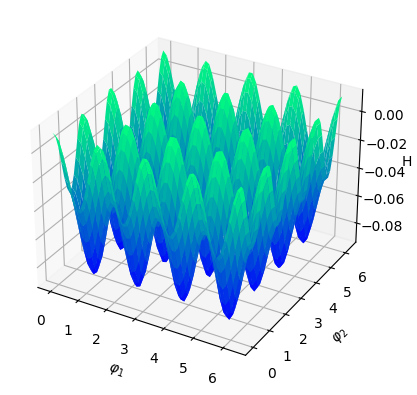}} 
    \medskip
    \subfigure[]
    {\includegraphics[width=0.4\textwidth]{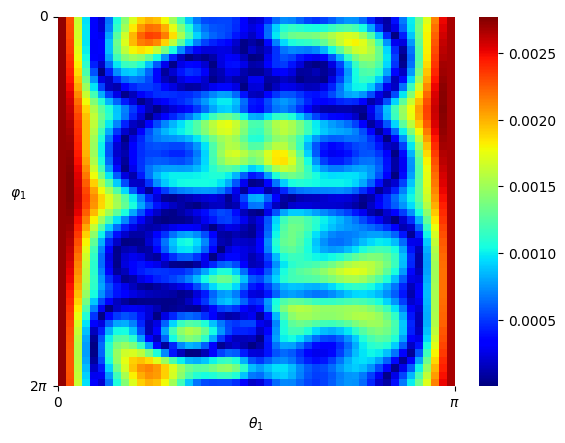}}
    \subfigure[]{\includegraphics[width=0.4\textwidth]{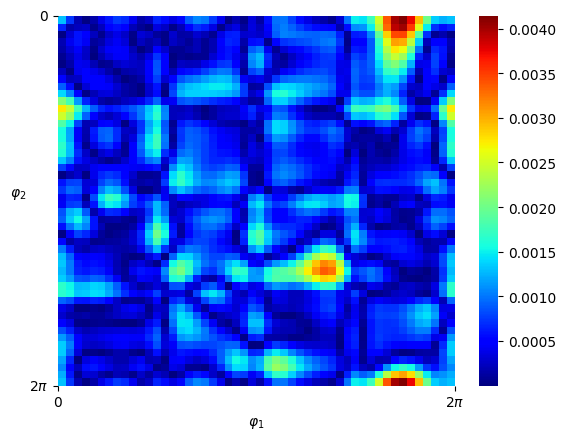}} 
    \caption{Spherical 3-norm on $S^2\times S^2$: (a)(b) Ground-truth Hamiltonian (c)(d) Learned Hamiltonian with $N=1200$ (e)(f) Absolute Error of the predicted Hamiltonian function adjusted by constant.} 
    \label{sphere_norm}
\end{figure}

\begin{figure}[htp]
    \centering
    \subfigure[]{\includegraphics[width=0.8\textwidth]{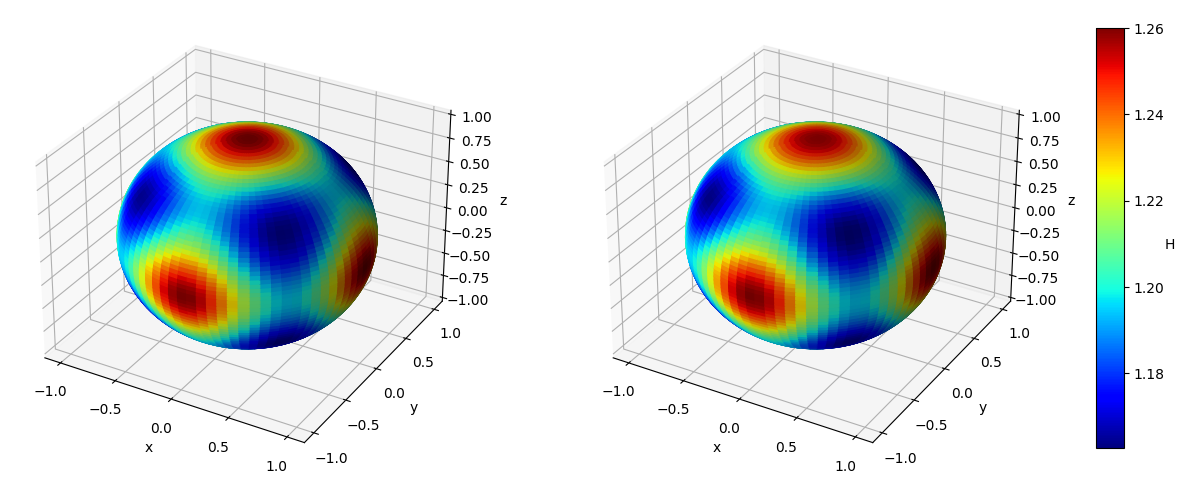}}
    \medskip
    \subfigure[]{\includegraphics[width=0.8\textwidth]{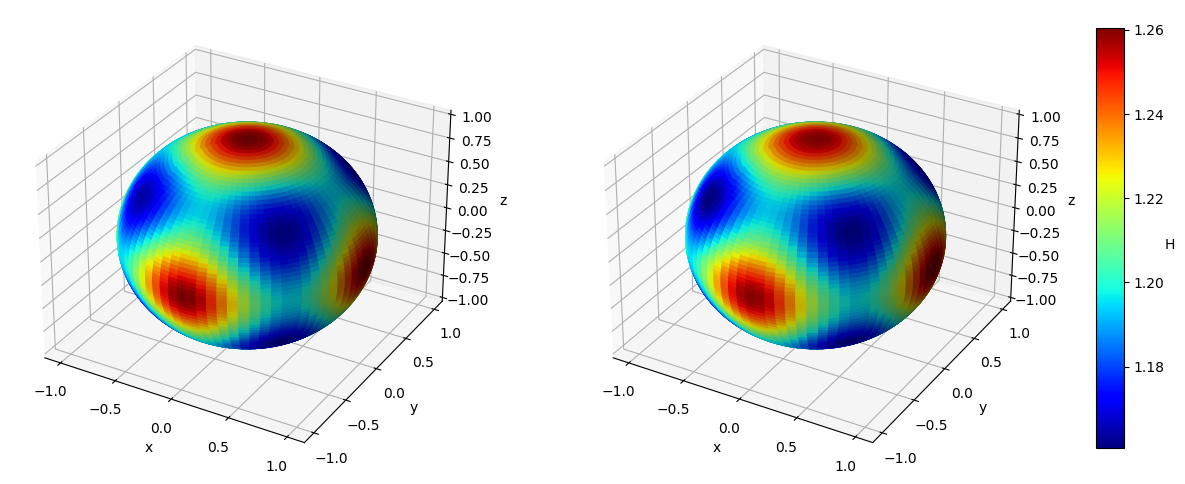}} 
    \caption{Global heatmap of the spherical 3-norm on $S^2\times S^2$ with $N=1200$: (a) left: ground-truth Hamiltonian on the first sphere; right: learned Hamiltonian on the first sphere adjusted by constant (b) left: ground-truth Hamiltonian on the second sphere; right: learned Hamiltonian on the second sphere adjusted by constant.} 
    \label{sphere_norm_heatmap}
\end{figure}

\begin{figure}[htp]
    \centering
    \includegraphics[width=0.4\textwidth]{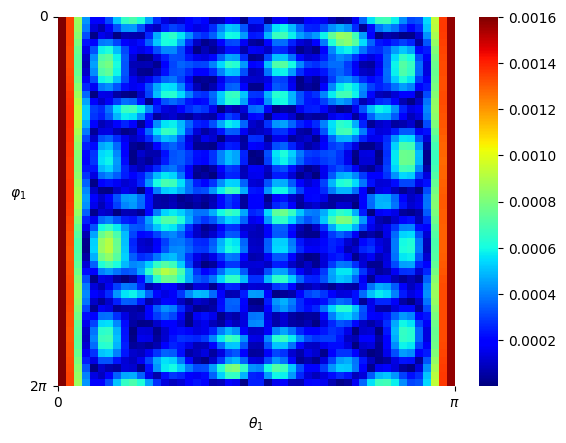}
    \caption{Spherical 3-norm on $S^2$ with $N=2000$: Absolute Error of the predicted Hamiltonian function adjusted by constant.} 
    \label{single_sphere_norm}
\end{figure}

% \noindent{\bf Analysis} \quad In this setting, $H$ does not have singularities, and no universality is needed in the kernel because $H\in \mathcal{H}_K$ by construction. It can be seen that the estimator $\widehat{h}_{\lambda,N}$ recovers the ground-truth $H$ well. However, theoretically, we still cannot claim that such an outcome is guaranteed by Theorem \ref{convergence_theorem} because we have not studied if the source condition \eqref{sou-con} is satisfied. In general, this requires a characterization of the function space $\mathcal{H}_K$ on the manifold. 

\subsubsection{Recovery of two-vortex dynamics}

The Hamiltonian function of the two-vortex dynamics is
\begin{equation}\label{two_vertex Ham}
    H^{\prime}\left(x_1, x_2\right)=-\lambda_1 \lambda_2 \log \left(1-x_1 \cdot x_2\right),
\end{equation}
for $x_1,x_2\in \mathbb{R}^3$, and the Hamiltonian vector field $X_{H^{\prime}}$ is given by
$$
\dot{x}_j=X_{H^{\prime}}(\mathbf{x})_j=\sum_{i \neq j} \lambda_i \frac{x_i \times x_j}{1-x_i \cdot x_j}.
$$

\noindent Moreover, in local coordinate, the Hamiltonian can be written as
\begin{align*}
H(\theta_1,\varphi_1,\theta_{2},\varphi_2)=-\lambda_1\lambda_2\log(1-\sin(\theta_1)\sin(\theta_2)\cos(\varphi_1-\varphi_2)-\cos(\theta_1)\cos(\theta_2)),
\end{align*}
and, as we have seen in \eqref{rep-ham-vector-field}, the Hamiltonian vector field can be locally written as 
\begin{align*}
X_H(\theta_1,\varphi_1,\theta_{2},\varphi_2)=J_{\mathcal{P}}(\theta_1,\varphi_1,\theta_{2},\varphi_2)\nabla H(\theta_1,\varphi_1,\theta_{2},\varphi_2).
\end{align*}

\noindent{\bf Numerical results} \quad For the numerical experiment, since the patch $(V_1 \times V_1,\phi_1\times \phi_1)$ is topologically dense in the space $\mathcal{P}$, we shall assume that all the training data will fall into this patch for computational convenience. We consider the standard Gaussian kernel with parameter $\eta$ on $\mathbb{R}^6$ but restricted to $\mathcal{P}$. Since the Gaussian kernel is positive-definite on the entire $\mathbb{R}^6$, it is also positive-definite when restricted to the manifold $\mathcal{P}$, and hence constitutes a valid choice of a kernel on $\mathcal{P}$.  We set $N=1200$ and sample $N$ states from the uniform distribution on $\mathcal{P}$ as elaborated above, and then obtain the Hamiltonian vector fields at these states. In view of \eqref{dyn-sca}, we set the regularization parameter $\lambda=c\cdot N^{-\alpha}$ with $\alpha=0.4$. We performed hyperparameter tunning and set $\eta=0.7$ and $c=0.01$. Given that $\mathcal{P}$ is $4$-dimensional, we shall respectively fix two of the dimensions, that is, $(\theta_2,\varphi_2)=(\frac{\pi}{2},0)$, $(\theta_1,\varphi_1)=(\frac{\pi}{2},\frac{\pi}{2})$ and $(\theta_1,\theta_2)=(\frac{\pi}{3},\frac{\pi}{3})$, and respectively visualize the mean-square error of the predicted Hamiltonian vector fields in a heatmap, see Figure \ref{typhoon} (a)(b)(c). In the end, we picture the Hamiltonian function globally on the first sphere as a heatmap, while fixing $(\theta_2,\varphi_2)=(\frac{\pi}{2},0)$ on the second sphere, see Figure \ref{typhoon_heatmap}.

\begin{figure}[htp]
    \centering
    % \subfigure[]{\includegraphics[width=0.32\textwidth]{Groundtruth_Typhoon.png}}
    % \subfigure[]{\includegraphics[width=0.32\textwidth]{Groundtruth_Typhoon_2.png}} 
    % \subfigure[]{\includegraphics[width=0.32\textwidth]{Groundtruth_Typhoon_3.png}}
    % \medskip
    % \subfigure[]
    % {\includegraphics[width=0.32\textwidth]{Learned_Typhoon.png}}
    % \subfigure[]{\includegraphics[width=0.32\textwidth]{Learned_Typhoon_2.png}} 
    % \subfigure[]{\includegraphics[width=0.32\textwidth]{Learned_Typhoon_3.png}}
    % \medskip
    \subfigure[]
    {\includegraphics[width=0.32\textwidth]{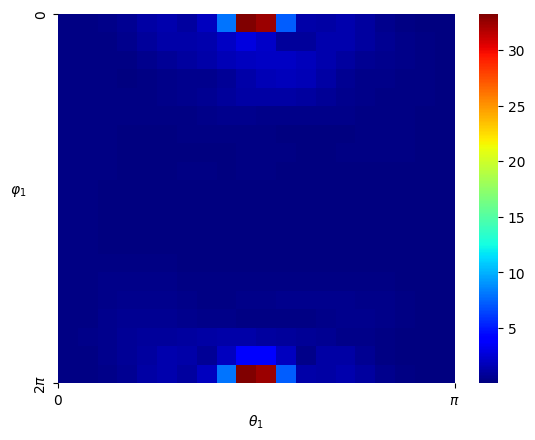}}
    \subfigure[]{\includegraphics[width=0.32\textwidth]{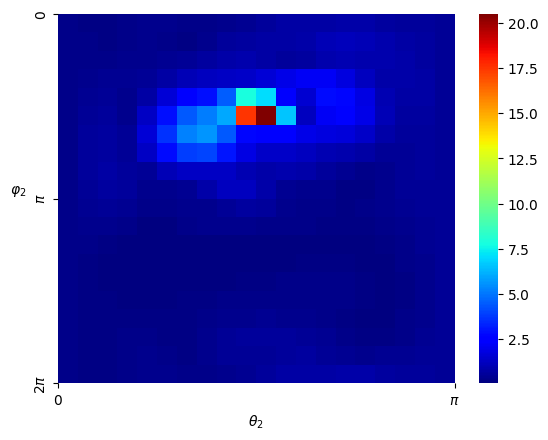}} 
    \subfigure[]{\includegraphics[width=0.32\textwidth]{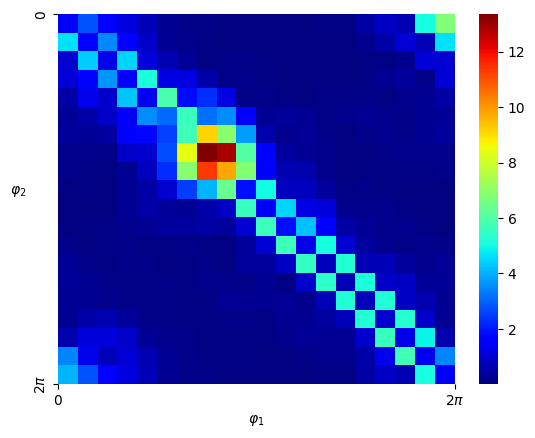}}
    \caption{Two-vortex Hamiltonian dynamics: Squared error of the predicted Hamiltonian vector field. It is observed that the vector field prediction error only becomes significant around the singularities.} 
    \label{typhoon}
\end{figure}

\begin{figure}[htp]
    \centering
    \subfigure[]{\includegraphics[width=0.8\textwidth]{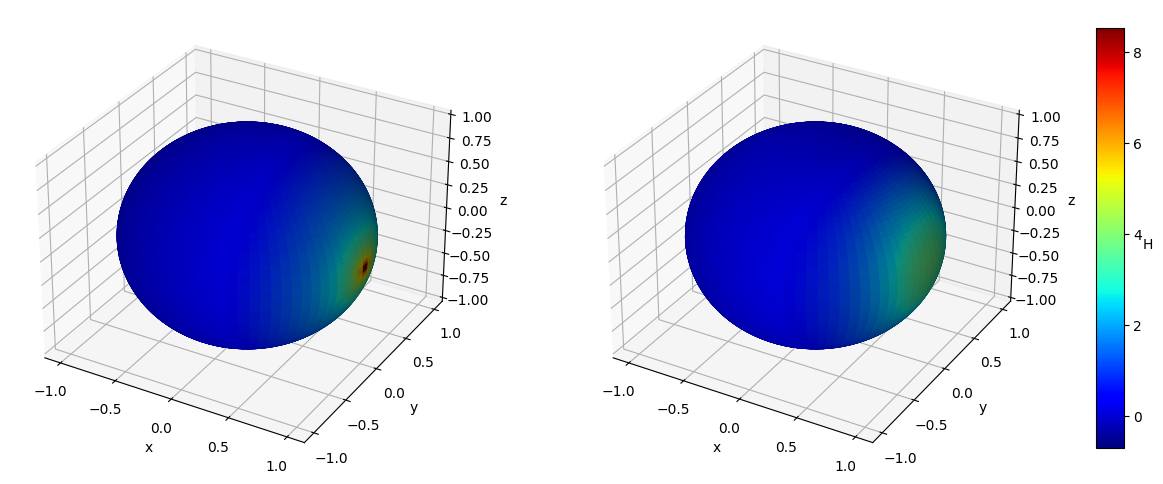}}
    \caption{Global heatmap of the Two-vortex Hamiltonian with N=1200. Left: ground-truth Hamiltonian on the first sphere; right: learned Hamiltonian on the first sphere adjusted by constant.} 
    \label{typhoon_heatmap}
\end{figure}

\noindent{\bf Analysis} \quad It can be seen that the estimator $\widehat{h}_{\lambda,N}$ fully captures the qualitative behavior of the ground-truth Hamiltonian $H$. Since $H$ exhibits singularities, by Theorem \ref{dif-rep} (iii), $H$ cannot be in $\mathcal{H}_K$. In this scenario, even though the restricted Gaussian kernel to the manifold $\mathcal{P}$ is universal, it only guarantees good approximation properties within the category of continuous functions, and not for functions with singularities.

\section{Conclusion}

This paper proposes a novel kernel-based machine learning method that offers a closed-form solution to the inverse problem of recovering a globally defined and potentially high-dimensional and nonlinear Hamiltonian function on Poisson manifolds. The method uses noisy observations of Hamiltonian vector fields on the phase space as data input. The approach is formulated as a kernel ridge regression problem, where the optimal candidate is sought within a reproducing kernel Hilbert space (RKHS) on the manifold, minimizing the discrepancy between the observed and candidate Hamiltonian vector fields (measured with respect to a chosen Riemannian metric) and an RKHS-based regularization term. Despite the complexities associated with optimization on a manifold, the problem remains convex, leading to an explicit solution derived by leveraging a ``differential" version of the reproducing property on the manifold and considering the variations of the loss function.

The kernel method exhibits various key advantages. First, as we demonstrated in Section \ref{RKHS on Riemannian manifolds}, the differential reproducing property allows the kernel ridge regression framework to naturally incorporate the differential equation constraints necessary for structure preservation, even on manifolds. Second, the method provides a closed-form solution, eliminating the computational burdens typically associated with iterative, gradient-based optimization. Additionally, the strict convexity of the Tikhonov-regularized kernel regression ensures a unique solution, addressing the ill-posedness that is inherent in recovering the Hamiltonian function in the presence of Poisson degeneracy due to the potential presence of Casimir functions. Finally, we introduced an operator-theoretic framework and a kernel estimator representation that enables a rigorous analysis of estimation and approximation errors. Additionally, when the target function exhibits a symmetry that is a priori known, as it is common in mechanical systems, an appropriate kernel can be selected to satisfy the symmetry constraints, further enhancing the method's versatility, applicability, and structure-preservation features.

This paper tackles just one aspect of the broader class of structure-preserving inverse problems, focusing on extending prior approaches from Euclidean spaces to non-Euclidean manifolds. A significant challenge still lies in addressing systems defined in infinite-dimensional spaces, such as the space of measures. A notable example is the Schr\"odinger equation, which can be interpreted as a transport equation for measures under the Hamiltonian flow. Interestingly, this transport equation can also be viewed as a Hamiltonian flow within the Wasserstein space, referred to as the Wasserstein Hamiltonian flow. Extending the kernel-based framework to identify the potential function in the Schr\"odinger equation, thereby accommodating infinite-dimensional underlying spaces, is the focus of the authors' research agenda.

The methodology in this paper assumes that Hamiltonian vector field data is directly available, which is a reasonable assumption when the observed data has a measurable physical interpretation, such as acceleration or force. However, in more realistic scenarios, the collected data might consist of discrete-time trajectories rather than continuous vector field observations. In this context, using variational integrators and other manifold-preserving techniques becomes crucial, yet these methods have not been fully explored in the existing literature in connection with this problem. Addressing this gap could open new avenues for applying structure-preserving kernel methods to a broader range of real-world problems.

As a final remark, a great wealth of knowledge has been accumulated on the qualitative behavior of Hamiltonian dynamical systems based on their geometry \cite{Abraham1978}, symmetries \cite{mwr, Ortega2004}, stability properties \cite{Ortega2005re}, or persistence \cite{mre, persistence:periodic, paper1:persitence} and bifurcation phenomena \cite{pascal:hopf, pascal}. All these concepts surely have an interplay in relation to learnability that needs to be explored.

\begin{table}\scriptsize
\caption{Comparison among symplectic vector space, symplectic manifold, and Poisson manifold setups}

\medskip

\medskip

%\centering
%\begin{center}
\!\!\!\!\!\!\!\!\!\!\!\!\!\!\!\!\!\!\!\!\!\!\!\!\!\!\!\!\!\!\!\!\!\!\!\!
\begin{tblr}{
vline{2,3,4} = {-}{},
hline{1,7} = {-}{0.1em},
hline{2,3,4,5,6} = {-}{},
}
& Symplectic vector space  & Symplectic manifold &   Poisson manifold   \\
Space & $(V,\omega_{\mathrm{can}})$ & $(M,\omega)$ & $(P,\{\cdot,\cdot\})$   \\
Hamiltonian vector field& $X_h=J_{\mathrm{can}}\nabla h$ & $i_{X_h}=\mathbf{d}h$ & $X_h=B^{\sharp}\mathbf{d}h$   \\
Compatible structure & $J_{\mathrm{can}}:TV\to TV$ & $J=\omega^\sharp g^\flat:TM\to TM$ & $J=B^\sharp g^\flat:TP\to TP$   \\
Differential Gram matrix& $G_N=\mathbb{J}_{\mathrm{can}}\nabla_{1,2}K(\mathbb{Z}_N,\mathbb{Z}_N)\mathbb{J}_{\mathrm{can}}^{\top}$ & $G_N\mathbf{c}:= X_{g_N(\mathbf{c},X_{K_{\cdot}}(\mathbf{Z}_N))}(\mathbf{Z}_N)$ & $G_N\mathbf{c}:= X_{g_N(\mathbf{c},X_{K_{\cdot}}(\mathbf{Z}_N))}(\mathbf{Z}_N)$ \\
Estimator& $\mathbf{X}_{\sigma^2,N}^\top \left(G_N+\lambda NI\right)^{-1}\mathbb{J}_{\mathrm{can}}\nabla_1K(\mathbf{Z}_N,\cdot)$& $ g_N\left( (G_N+\lambda NI)^{-1}\mathbf{X}_{\sigma^2,N},X_{K_{\cdot}}(\mathbf{ Z}_N)\right)$ & $g_N\left( (G_N+\lambda NI)^{-1}\mathbf{X}_{\sigma^2,N},X_{K_{\cdot}}(\mathbf{ Z}_N)\right)$ 
\end{tblr}
%\end{center}
\end{table}

%\begin{table}\footnotesize
%\caption{Comparison among symplectic vector space, symplectic manifold, and Poisson manifold }
%\centering
%\begin{center}
%\begin{tblr}{
%vline{2,3,4} = {-}{},
%hline{1,6} = {-}{0.1em},
%hline{2,3,4,5} = {-}{},
%}
%& Symplectic vector space  & Symplectic manifold &   Poisson manifold   \\
%Space & $(V,\omega_{\mathrm{can}})$ & $(M,\omega)$ & $(P,\{\cdot,\cdot\})$   \\
%Hamiltonian vector field& $X_h=J_{\mathrm{can}}\nabla h$ & $i_{X_h}=\mathbf{d}h$ & $X_h=B^{\sharp}\mathbf{d}h$   \\
%Compatible structure & $J_{\mathrm{can}}:TV\to TV$ & $J=\omega^\sharp g^\flat:TM\to TM$ & $J=B^\sharp g^\flat:TP\to TP$   \\
%Differential Gram matrix& $\mathbb{J}_{\mathrm{can}}\nabla_{1,2}K(\mathbb{Z}_N,\mathbb{Z}_N)\mathbb{J}_{\mathrm{can}}^{\top}$ & $G_N\mathbf{c}:= X_{g_N(\mathbf{c},X_{K_{\cdot}}(\mathbf{Z}_N))}(\mathbf{Z}_N)$ & $G_N\mathbf{c}:= X_{g_N(\mathbf{c},X_{K_{\cdot}}(\mathbf{Z}_N))}(\mathbf{Z}_N)$  
%\end{tblr}
%\end{center}
%\end{table}

\appendix

\section{Appendices}

\subsection{The proof of Theorem \ref{Par-Rep}}\label{proof-differential-reproducing-properties}

We prove {\bf (i)} and {\bf (ii)} simultaneously by induction on $s$.
The case $s=0$ is trivial since that means for any $x\in M$, $D^{(0,0)}K(x,\cdot)=K_{x}$ satisfies the standard reproducing property in $\mathcal{H}_K$. Let now $0\leq l\leq s-1$. Suppose that $D^{(l,0)}K(y,\cdot)\cdot v\in \mathcal{H}_K$ and \eqref{dif-rep} holds for any $y\in T^{l-1}M$ and $v\in T_yT^{l-1}M$. 
Then \eqref{dif-rep} implies that for any $1\leq k,k'\leq l-1$, $x\in T^kM$, $u\in T_xT^kM$, $x'\in T^{k'}M$, and $u'\in T_{x'}T^{k'}M$,
\begin{equation}\label{def-rep}
\langle D^{(k,0)}K(x,\cdot)\cdot u, D^{(k',0)}K(x',\cdot)\cdot u'\rangle_{\mathcal{H}_K} = D^{(k,k')} K(x,x')\cdot(u,u').
\end{equation}
Denote $K^k(y,\cdot)=D^{(k,0)}K(x,\cdot)\cdot u$ and $K^{(k,k')}(y,y'):=D^{(k,k')} K(x,x')\cdot(u,u')$ with with $y=(x,u)$ and $y'=(x',u')$.

Now, we turn to the case $l+1\leq s$. In the following proof, let $\phi:[-1,1]\to T^lM$ be a differentiable curve with $\phi(0)=y\in T^lM$ and $\phi'(0)=v\in T_yT^lM$. 
Since $T^lM$ with Riemannian metric $g_l$ is complete (see Remark \ref{Functions with bounded differentials}), by Hopf-Rinow theorem, we can always choose $\phi$ to be geodesics with constant speed 1.
We prove that {\bf (i)} and {\bf (ii)} hold for $l+1$ in three steps.

\medskip

\noindent {\bf  Step 1}: Proving $D^{(l+1,0)}K(y,\cdot)\cdot v\in \mathcal{H}_K$ for all $y\in T^{l}M$ and $v\in T_yT^lM$. 
Equations \eqref{fun-the} and \eqref{def-rep} imply that the set $\{\frac{1}{t}(K^l(\phi(t))-K^l(y)):|t|\leq 1\}$ of functions in $\mathcal{H}_K$ satisfies 
\begin{align*}
\Big\|\frac{1}{t}(&K^l(\phi(t),\cdot)-K^l(y,\cdot))\Big\|^2_{\mathcal{H}_K}=\frac{1}{t^2}\Big(K^{(l,l)}(\phi(t),\phi(t))
-K^{(l,l)}(y,\phi(t))-K^{(l,l)}(\phi(t),y)+K^{(l,l)}(y,y)\Big)\\
&= \frac{1}{t^2}\int_0^t\Big(D^{(1,0)}K^{(l,l)}(\phi(s),\phi(t))\cdot\phi'(s) -D^{(1,0)}K^{(l,l)}(\phi(s),y)\cdot \phi'(s)\Big) \mathrm{d}s\\
&= \frac{1}{t^2}\int_0^t\int_0^tD^{(1,1)}K^{(l,l)}(\phi(s),\phi(\tau))\cdot(\phi'(s),\phi'(\tau)) \mathrm{d}\tau \mathrm{d}s\\
&\leq \frac{1}{t^2}\|K\|_{C_b^{2(l+1)}}\int_0^t\int_0^t\|\phi'(\tau)\|_l\|\phi'(s)\|_l \mathrm{d}\tau \mathrm{d}s=\frac{1}{2}\|K\|_{C_b^{2(l+1)}}, \quad \forall ~|t|\leq 1,
\end{align*}
where $\|\cdot\|_l$ is induced by the Riemannian metric on $T^lM$ and we have used the assumption that $K\in C_b^{2s+1}(M\times M)$. This means that $\{\frac{1}{t}(K^l(\phi(t))-K^l(y)):|t|\leq 1\}$ lies in a closed ball of the Hilbert space $\mathcal{H}_K$ with a finite radius. Since this ball is sequentially weakly compact (see \cite[Theorem 4.2]{conway:book}), 
there is a sequence $\{t_i\}_{i=1}^{\infty}$ with $|t_i|\leq 1$ and $\lim_{i\rightarrow \infty}t_i=0$ such that $\{\frac{1}{t_i}(K^{l}(\phi(t_i),\cdot)-K^l(y,\cdot)):|t_i|\leq 1\}$ converges weakly to an element $h_{y}$ of $\mathcal{H}_K$ as $i\rightarrow\infty$. The weak convergence tells us that
\begin{equation}\label{wea-con}
\lim\limits_{i\rightarrow \infty}\Big \langle \frac{1}{t_i}(K^l(\phi(t_i),\cdot)-K^l(y,\cdot)), f \Big \rangle_{\mathcal{H}_K} = \langle h_{y},f \rangle_{\mathcal{H}_K}, \quad \forall f\in \mathcal{H}_K.
\end{equation}
In particular, by taking $f=K(x,\cdot)$ with $x\in M$, it holds that
\begin{align*}
h_{y}(x)&=\lim\limits_{i\rightarrow\infty}\Big \langle \frac{1}{t_i}(K^l(\phi(t_i),\cdot)-K^l(y,\cdot)), K(x,\cdot) \Big \rangle_{\mathcal{H}_K}\\
&=\lim\limits_{i\rightarrow\infty} \frac{1}{t_i}(K^{(l,0)}(\phi(t_i),x)-K^{(l,0)}K(y,x))=D^{(1,0)}K^{(l,0)}(y,x)\cdot v=D^{(l+1,0)}K(y,x)\cdot v.
\end{align*}
This is true for an arbitrary point $x\in M$. Hence $D^{(l+1,0)}K(y,\cdot)\cdot v=h_{y}$ as functions on $M$. Since $h_{y}\in \mathcal{H}_K$, we know $D^{(l+1,0)}K(y,\cdot)\cdot v\in \mathcal{H}_K$.

\medskip

\noindent {\bf  Step 2:} Proving the convergence 
\begin{align}\label{con-in-hk}
\frac{1}{t}\Big(K^l(\phi(t),\cdot)-K^l(y,\cdot)\Big) \longrightarrow D^{(l+1,0)}K(y,\cdot)\cdot v, \quad \text{in} \quad \mathcal{H}_K.
\end{align}\\
Using the notation introduced above, $K^{l+1}((y,v),\cdot)=D^{(l+1,0)}K(y,\cdot)\cdot v\in \mathcal{H}_K$. Then equation \eqref{wea-con} yields
\begin{align*}
\left\langle K^{l+1}((y,v),\cdot), K^{l+1}((y,v),\cdot)\right\rangle_{\mathcal{H}_K}&
=\lim\limits_{i\rightarrow\infty}\frac{1}{t_i} \left\langle K^{l+1}((y,v),\cdot), K^l(\phi(t_i),\cdot)-K^l(y,\cdot)\right\rangle_{\mathcal{H}_K}\\
=\ & \lim\limits_{i\rightarrow\infty}\frac{1}{t_i} \left( D^{(1,0)}K^{(l,l)}(y,\phi(t_i))\cdot v- D^{(1,0)}K^{(l,l)}(y,y)\cdot v\right)\\
=\ & D^{(1,1)}K^{(l,l)}(y,y)\cdot(v,v)=D^{(l+1,l+1)}K(y,y)\cdot (v,v).
\end{align*}\\
By equation \eqref{fun-the}, we obtain that
\begin{align*}
&\Big\|\frac{1}{t}\Big(K^l(\phi(t),\cdot)-K^l(y,\cdot)\Big)-K^{l+1}((y,v),\cdot)\Big\|^2_{\mathcal{H}_K}\\
=~&\frac{1}{t^2}\Big(K^{(l,l)}(\phi(t),\phi(t))-2K^{(l,l)}(\phi(t),y)+K^{(l,l)}(y,y)\Big)+K^{(l+1,l+1)}((y,v),(y,v))\\
~& -\frac{2}{t}\Big(K^{(l+1,l)}((y,v),\phi(t))-K^{(l+1,l)}((y,v),y)\Big)\\
=~&\frac{1}{t^2}\int_{0}^t\int_{0}^tK^{(l+1,l+1)}((\phi(s),\phi'(s)),(\phi(\tau),\phi'(\tau)))\mathrm{d}\tau\mathrm{d}s+K^{(l+1,l+1)}((y,v),(y,v))\\
~& -\frac{2}{t}\int_0^tK^{(l+1,l+1)}((\phi(\tau),\phi'(\tau)),(u,v))~\mathrm{d}\tau\\
=~&\frac{1}{t^2}\int_{0}^t\int_{0}^t\left[K^{(l+1,l+1)}((\phi(s),\phi'(s)),(\phi(\tau),\phi'(\tau))-K^{(l+1,l+1)}((\phi(s),\phi'(s)),(y,v))\right]\mathrm{d}\tau\mathrm{d}s\\
&-\frac{1}{t}\int_{0}^t\left[K^{(l+1,l+1)}((\phi(\tau),\phi'(\tau)),(y,v))-K^{(l+1,l+1)}((y,v),(y,v))\right]\mathrm{d}\tau\\
=~&\frac{1}{t^2}\int_{0}^t\int_{0}^t\int_0^\tau D^{(0,1)}K^{(l+1,l+1)}((\phi(s),\phi'(s)),(\phi(z),\phi'(z)))\cdot \psi'(z)\mathrm{d}z\mathrm{d}\tau\mathrm{d}s\\
~&-\frac{1}{t}\int_{0}^t\int_0^{\tau}D^{(1,0)}K^{(l+1,l+1)}((\phi(z),\phi'(z)),(y,v))\cdot\psi'(z)\mathrm{d}z\mathrm{d}\tau\\
=~&\frac{1}{t^2}\int_{0}^t\int_{0}^t\int_0^\tau D^{(l+1,l+2)}K(x,x)\cdot ((\phi(s),\phi'(s)),(\phi(z),\phi'(z)))\cdot \psi'(z)\mathrm{d}z\mathrm{d}\tau\mathrm{d}s\\
~&-\frac{1}{t}\int_{0}^t\int_0^{\tau}D^{(l+2,l+1)}K(x,x)\cdot((\phi(z),\phi'(z)),(y,v))\cdot\psi'(z)\mathrm{d}z\mathrm{d}\tau\\
\leq~&\frac{1}{t^2}\int_{0}^t\int_{0}^t\int_0^\tau\|D^{(l+1,l+2)}K\|_{\infty}\|\psi'(z)\|_{l+1}\mathrm{d}z\mathrm{d}\tau\mathrm{d}s+\frac{1}{t}\int_{0}^t\int_0^s\|D^{(l+2,l+1)}K\|_{\infty}\|\psi'(z)\|_{l+1}\mathrm{d}z\mathrm{d}s\\
=~& \frac{1}{6}\|D^{(l+1,l+2)}K\|_{\infty}~ t+\frac{1}{2}\|D^{(l+2,l+1)}K\|_{\infty}~ t
\leq \frac{1}{2} \|K\|_{C_b^{2l+3}} ~ t\leq \frac{1}{2}\|K\|_{C_b^{2s+1}} ~ t,
\end{align*}
where $\psi:[0,t]\to T^{l+1}M$ is chosen as a geodesic with constant speed 1, together with $\psi(0)=(y,v)$ and $\psi(z)=(\phi(z),\phi'(z))$, the second and fourth qualities are due to the formula \eqref{fun-the}, and the last two inequalities thank to $K\in C_b^{2s+1}(M\times M)$ and $2l+3\leq 2s+1$.
As $t\rightarrow 0$, \eqref{con-in-hk} follows.

\medskip

\noindent {\bf  Step 3:} Proving \eqref{dif-rep}. Let $f\in \mathcal{H}_K$. By (\ref{con-in-hk}) we have
\begin{equation*}
\begin{aligned}
D^{l+1}f(y)\cdot v&=\lim\limits_{t\rightarrow0}\frac{1}{t}\left ( D^lf(\phi(t))-D^lf(y)\right )=\lim\limits_{t\rightarrow0}\Big\langle \frac{1}{t}\left(K^l(\phi(t),\cdot)-K^l(y,\cdot)\right),f \Big \rangle_{\mathcal{H}_K}\\
&=\langle K^{(l+1,0)}((y,v),\cdot),f \rangle_{\mathcal{H}_K}=\langle D^{(l+1,0)}K(y,\cdot)\cdot v,f \rangle_{\mathcal{H}_K}.
\end{aligned}
\end{equation*}
That is, $D^{l+1}f(y)\cdot v$ exists and equals $\langle (D^{(l+1,0)}K(y,\cdot)\cdot v,f \rangle_{\mathcal{H}_K}$. This verifies (\ref{dif-rep}) for $l+1$. Hence, the parts {\bf (i)} and {\bf (ii)} follow.\\
We conclude by proving part {\bf (iii)} using \eqref{dif-rep} and \eqref{def-rep}. For $f\in\mathcal{H}_K$, and $x\in M$, the Cauchy-Schwarz inequality implies that, 
\begin{equation*}
\begin{aligned}
|f(x)|=|\langle K_x,f \rangle_{\mathcal{H}_K}|\leq \sqrt{K(x,x)}~\|f\|_{\mathcal{H}_K}\leq \sqrt{\|K\|_{\infty}}~\|f\|_{\mathcal{H}_K},
\end{aligned}
\end{equation*}
and for all $1\leq k\leq s$, $y\in T^{k-1}M$ and $v\in T_yT^{k-1}M$,
\begin{equation*}
\begin{aligned}
|D^kf(y)\cdot v|&=|\langle (D^{(k,0)}K(y,\cdot)\cdot v,f \rangle_{\mathcal{H}_K}|\leq \sqrt{D^{(k,k)}K(y,y)\cdot (v,v)}~\|f\|_{\mathcal{H}_K}\leq \sqrt{\|D^{(k,k)}K\|_{\infty}}~\|v\|_{k-1}\|f\|_{\mathcal{H}_K}.
\end{aligned}
\end{equation*}
Therefore, $f\in C^{k}_b(M)$ and $\|D^kf\|_{\infty}\leq \sqrt{\|D^{(k,k)}K\|_{\infty}}~\|f\|_{\mathcal{H}_K}$. Denote $\kappa=\sqrt{(s+1)\|K\|_{C_b^2}}$.
It follows that
\begin{align*}
\|f\|_{C^{s}_b}&=\|f\|_{\infty}+\sum_{k=1}^s\|D^kf\|_{\infty}\leq 
\left(\sqrt{\|K\|_{\infty}}+\sum_{k=1}^s\sqrt{\|D^{(k,k)}K\|_{\infty}}\right)~\|f\|_{\mathcal{H}_K}\\
&\leq \sqrt{(s+1)\left(\|K\|_{\infty}+\sum_{k=1}^s\|D^{(k,k)}K\|_{\infty}\right)}~\|f\|_{\mathcal{H}_K}\leq\sqrt{(s+1)\|K\|_{C_b^2}}~\|f\|_{\mathcal{H}_K}=\kappa\|f\|_{\mathcal{H}_K}.
\end{align*}

\subsection{The proof of Proposition \ref{BoundA}}\label{proof of boundA}

Since $K\in C_b^3(P\times P)$, Theorem \ref{Par-Rep} yields that for any $h\in \mathcal{H}_K$ and $z\in P$, 
\begin{align*}
g(z)(\nabla h(z),\nabla h(z))\leq \|h\|_{C_b^1} \left(g(z)(\nabla h(z),\nabla h(z))\right)^{\frac{1}{2}} \leq \kappa  \|h\|_{\mathcal{H}_K} \left(g(z)(\nabla h(z),\nabla h(z))\right)^{\frac{1}{2}},
\end{align*}
which implies that $g(z)(\nabla h(z),\nabla h(z))\leq \kappa^2 \|h\|_{\mathcal{H}_K}^2$.
Then by Assumption \ref{Con-com} and equation \eqref{rep-ham-vector-field}, we have that
\begin{align}\label{norm-control}
g(z)(X_h(z),X_h(z))=g(z)(J(z)\nabla h(z),J(z)\nabla h(z))\leq \gamma(z) g(z)(\nabla h(z),\nabla h(z))\leq
\kappa^2\gamma(z) \|h\|_{\mathcal{H}_K}^2.
\end{align}
Thus, the operator $A$ is bounded linear from $\mathcal{H}_K$ to $L^2(P,\mu_{\mathbf{Z}})$ since
\begin{align*}
\|Ah\|^2_{L^2(\mu_{\mathbf{Z}})}= \int_ {P} g(z)( X_h(z), X_h(z))\mathrm{d}\mu_{\mathbf{Z}}(z)\leq\int_ {P}\kappa^2\gamma(z) \|h\|_{\mathcal{H}_K}^2\mu_{\mathbf{Z}}(z)\leq  \kappa^2C^2\|h\|_{\mathcal{H}_K}^2,
\end{align*}
Furthermore, for any vector field $Y\in  L^2(P,\mu_{\mathbf{Z}})$, Corollary \ref{Wel-Ope} and Proposition \ref{Casimir property} imply that
\begin{align*}
\langle Ah, Y\rangle_{L^2(\mu_{\mathbf{Z}})} &=\int_ {P} g(z)(X_h(z),Y(z))\mathrm{d}\mu_{\mathbf{Z}}(z) =\int_ {P} g(z)\left( \nabla h(z), -J(z)Y(z)\right)\mathrm{d}\mu_{\mathbf{Z}}(z)\\
&=\int_ {P} (-JY)[h](z)\mathrm{d}\mu_{\mathbf{Z}}(z)=\int_ {P} (-JY)[\langle h,K_{\cdot}\rangle_{\mathcal{H}_K}](z)\mathrm{d}\mu_{\mathbf{Z}}(z)\\
&=\int_ {P} \frac{d}{dt}\mid_{t=0}[\langle h,K_{F_t(z)}\rangle_{\mathcal{H}_K}]\mathrm{d}\mu_{\mathbf{Z}}(z)=\left\langle h,\int_ {P} (-JY)[K_{\cdot}](z)\mathrm{d}\mu_{\mathbf{Z}}(z)\right\rangle_{\mathcal{H}_K}.
\end{align*}
where $F_t$ is the flow of the vector field $-JY$ and Remark \ref{Par-Rep1}{\bf (i)} shows that $(-JY)[K_{\cdot}](z)\in \mathcal{H}_K$. This leads to 
\begin{align*}
(A^*Y)(y) =\int_ {P} (-JY)[K_{y}](z)\mathrm{d}\mu_{\mathbf{Z}}(z)=\int_ {P} g(z)( X_{K_{y}}(z), Y(z))\mathrm{d}\mu_{\mathbf{Z}}(z).
\end{align*}
Note that hence, for each $h\in\mathcal{H}_K$,
\begin{align*}
(Qh)(y):&=(A^*Ah)(y) = (A^* X_h)(y)=\int_ {P} g(z)( X_{K_y}(z), X_h(z))\mathrm{d}\mu_{\mathbf{Z}}(z).
\end{align*}  

We now show that $Q$ is a trace class operator; that is, we show that $\operatorname{Tr}(|Q|)<\infty$, where $|Q|=\sqrt{Q^* Q}$. The form of the operator $Q=A^{*}A$ automatically guarantees that it is positive semidefinite, and hence we have that $|Q|=Q$. Therefore, it is equivalent to show that $\operatorname{Tr}(Q)<\infty$. To do that, we choose a countable spanning orthonormal set $\left\{e _n\right\} _{n \in \mathbb{N}}$ for ${\mathcal H} _K$ whose existence is guaranteed by the continuity of the canonical feature map associated to $K$ that we established in \cite[Lemma A.3]{hu2024structure} and \cite[Theorem 2.4]{owhadi2017separability}. Then,
$$
\begin{aligned}
\operatorname{Tr}(Q)&=\operatorname{Tr}\left(A^* A\right) =\sum_n\left\langle A^* A e_n, e_n\right\rangle_{\mathcal{H}_K}=\sum_n\left\langle A e_n, A e_n\right\rangle_{L^2\left(\mu_{\mathbf{Z}}\right)} \\
& =\sum_n\int_{P}g(z)( X_{e_n}(z), X_{e_n}(z))\mathrm{d}\mu_{\mathbf{Z}}(\mathbf{x})  \\
& \leq C^2\sum_n\int_{P}g(z)( \nabla{e_n}(z), \nabla{e_n}(z))\mathrm{d}\mu_{\mathbf{Z}}(\mathbf{x})\\
&\leq dC^2\|K\|_{C_b^2}= \frac{1}{2}dC^2\kappa^2<\infty,
\end{aligned}
$$
where the sixth inequality is derived by choosing an arbitrary orthonormal basis $\{v_i\}_{i=1}^d$ of the vector space $T_{z}P$ as follows.
\begin{align*}
\sum_ng(z)( \nabla{e_n}(z), \nabla{e_n}(z))&= \sum_nDe_n(z)\cdot \nabla{e_n}(z) = \sum_n\langle D^{(1,0)}K(z,\cdot)\cdot \nabla{e_n}(z),e_n\rangle_{\mathcal{H}_K}\\
&=\sum_n\left\langle D^{(1,0)}K(z,\cdot)\cdot \left(\sum_{i=1}^dg(z)(\nabla{e_n}(z),v_i)v_i\right),e_n\right\rangle_{\mathcal{H}_K}\\
&=\sum_n\sum_{i=1}^d\left\langle D^{(1,0)}K(z,\cdot)\cdot v_i,g(z)(\nabla{e_n}(z),v_i)e_n\right\rangle_{\mathcal{H}_K}\\
&=\sum_{i=1}^d\left\langle D^{(1,0)}K(z,\cdot)\cdot v_i,\sum_n(D{e_n}(z)\cdot v_i)e_n\right\rangle_{\mathcal{H}_K}\\
&=\sum_{i=1}^d\left\langle D^{(1,0)}K(z,\cdot)\cdot v_i,\sum_n\langle D^{(1,0)}K(z,\cdot)\cdot v_i,e_n\rangle_{\mathcal{H}_K}e_n\right\rangle_{\mathcal{H}_K}\\
&=\sum_{i=1}^d\left\langle D^{(1,0)}K(z,\cdot)\cdot v_i,D^{(1,0)}K(z,\cdot)\cdot v_i\right\rangle_{\mathcal{H}_K}\\
&=\sum_{i=1}^dD^{(1,1)}K(z,z)\cdot (v_i,v_i)\leq d\|K\|_{C_b^2}.
\end{align*}

\subsection{The proof of Proposition \ref{Pos-Gra}}

Firstly, we prove the symmetry of $G_N$. It is sufficient to show that $g_N(\mathbf{c},G_N\mathbf{c}')=g_N(\mathbf{c}',G_N\mathbf{c})$ for $\mathbf{c},\mathbf{c}'\in T_{\mathbf{Z}_N}P$ arbitrary. Denote $\mathbb{J}_N=\operatorname{diag}\{J(\mathbf{Z}^{(1)}),\ldots,J(\mathbf{Z}^{(N)})\}$. Then by Remark \ref{Par-Rep1}{\bf (i)}, we obtain that for any $\mathbf c\in T_{\mathbf{Z}_N}P$,
\begin{align*}
g_N(\mathbf{c},X_{K_{\cdot}}(\mathbf{Z}_N))= D^{(1,0)}K(\mathbf{Z}_N,\cdot)\cdot \left(-\mathbb{J}_N\mathbf{c}\right).  
\end{align*}
Therefore, again by Remark \ref{Par-Rep1}{\bf (i)}, we have that
\begin{align*}
g_N(\mathbf{c},G_N\mathbf{c}')=g_N(\mathbf{c},X_{D^{(1,0)}K(\mathbf{Z}_N,\cdot)\cdot \left(-\mathbb{J}_N\mathbf{c}'\right)}(\mathbf{Z}_N)) = D^{(1,1)}K(\mathbf{Z}_N,\mathbf{Z}_N)(-\mathbb{J}_N\mathbf{c}',-\mathbb{J}_N\mathbf{c})
\end{align*}
Since $K$ is a symmetric function, $D^{(1,1)}K(\mathbf{Z}_N,\mathbf{Z}_N)$ is also symmetric by definition. This proves that $G_N$ is symmetric. 

We now show the positive semi-definiteness of $G_N$. Let $L=dN$. Since $G_N$ is real symmetric, there exists an orthonormal matrix $O\in T_{\mathbf{Z}_N}P \times T_{\mathbf{Z}_N}P$ that diagonalizes $G_N$, that is
\begin{align*}
G_N=ODO^{\top}
=\begin{bmatrix}
|&|&\dots&|\\
f_1&f_2&\dots&f_L\\
|&|&\dots&|
\end{bmatrix}
\begin{bmatrix}
d_1&0&\dots&0 \\
0&d_2&\dots&0 \\
\vdots&\vdots&\ddots&\vdots\\
0&0&\dots&d_L
\end{bmatrix}\begin{bmatrix}
|&|&\dots&|\\
f_1&f_2&\dots&f_L\\
|&|&\dots&|
\end{bmatrix}^{\top},
\end{align*} 
where  $\left \{d_i \right\}_{i=1}^{L}$ and $\left \{f_i \right\}_{i=1}^{L}$ are the real eigenvalues and the corresponding eigenvectors of $G_N$, respectively.
We now define $\widetilde{e}_i=g_N\left( f_i, X_{K_{\cdot}}(\mathbf{Z}_N) \right)$. By Remark \ref{Par-Rep1}{\bf (i)}, we obtain that $\widetilde{e}_i\in\mathcal{H}_K$ and that
\begin{equation*}
\begin{aligned}
\|\widetilde{e}_i\|_{\mathcal{H}_K}^2 &=\left\langle g_N\left( f_i, X_{K_{\cdot}}(\mathbf{Z}_N) \right),g_N\left( f_i, X_{K_{\cdot}}(\mathbf{Z}_N) \right)\right\rangle_{\mathcal{H}_K}\\
&= \left\langle D^{(1,0)}K(\mathbf{Z}_N,\cdot )\cdot(-\mathbb{J}_Nf_i), D^{(1,0)}K(\mathbf{Z}_N,\cdot )\cdot(-\mathbb{J}_Nf_i)\right\rangle_{\mathcal{H}_K} \\
&= D^{(1,1)}K(\mathbf{Z}_N,\mathbf{Z}_N)(-\mathbb{J}_Nf_i,-\mathbb{J}_Nf_i)=g_N(f_i,G_Nf_i)
=d_i.
\end{aligned}
\end{equation*}
Hence, $d_i\geq0$ for all $i=1,\ldots,L$ and we can conclude that $G_N$ is positive semi-definite.

\subsection{The proof of Lemma \ref{operator error}}\label{proof of operator error}
\begin{proof}
Since $K\in C_b^3(P\times P)$, then for any  $h\in\mathcal{H}_K$, we have 
\begin{align*}
\|Q_Nh\|_{\mathcal{H}_K}^2&=\frac{1}{N}\|g_N(X_h(\mathbf{Z}_N),X_{K_{\cdot}}(\mathbf{Z}_N))\|^2_{\mathcal{H}_K}=\frac{1}{N}\|D^{(1,0)}K(\mathbf{Z}_N,\cdot)\cdot (-JX_h)(\mathbf{Z}_N)\|^2_{\mathcal{H}_K}\\
&=\frac{1}{N}D^{(1,1)}K(\mathbf{Z}_N,\mathbf{Z}_N)\cdot ((-JX_h)(\mathbf{Z}_N),(-JX_h)(\mathbf{Z}_N))\\
&\leq \frac{1}{N}\|D^{(1,1)}K\|_{\infty} g_N(-JX_h(\mathbf{Z}_N),-JX_h(\mathbf{Z}_N))\leq \frac{C^2}{2N}\kappa^2g_N(X_h(\mathbf{Z}_N),X_h(\mathbf{Z}_N))\\
&\leq \frac{1}{2}C^4\kappa^4\|h\|^2_{\mathcal{H}_K}<\infty,
\end{align*}
which shows that $Q_Nh$ are bounded random variables in $\mathcal{H}_K$ for all $N\geq1$. Moreover, we have $\mathbb{E}\left[\|Q_Nh\|^2_{\mathcal{H}_K}\right]\leq \frac{1}{2}C^4\kappa^4\|h\|^2_{\mathcal{H}_K}$.  Define now the $\mathcal{H}_K$-valued random variables
$$
\xi^{(n)}=g(\mathbf{Z}^{(n)})(X_h(\mathbf{Z}^{(n)}),X_{K_\cdot}(\mathbf{Z}^{(n)})), \quad \mbox{$n=1, \ldots, N$.} 
$$ 
Note that the random variables $\{\xi^{(n)}\}_{n=1}^{N}$ are IID and that 
$Q_Nh-Qh=\frac{1}{N}\sum_{n=1}^N (\xi^{(n)}-\mathbb{E}(\xi^{(n)})).$ The result follows by applying \cite[Lemma 8]{de2005learning} or \cite[Lemma A.2]{yurinsky1995sums} to $\{\xi^{(n)}\}_{n=1}^{N}$. 
\end{proof}

\subsection{The proof of Lemma \ref{noisy sampling error}}
\label{proof of noisy sampling error}
\begin{proof}
We first notice that, by the symmetry of $G_N$, for any $\mathbf{c},\mathbf{c}^{\prime}\in T_{\mathbf{Z}_N}P$, it holds that
\begin{equation} \label{symmetry_of_G_N}
    g_N((G_N+\lambda NI)^{-1}\mathbf{c},\mathbf{c}^{\prime})=g_N(\mathbf{c},(G_N+\lambda NI)^{-1}\mathbf{c}^{\prime}).
\end{equation}

Following an approach similar to the one in Theorem \ref{Rep-Ker}, we obtain that
\begin{align*}
\big\|\widehat{h}_{\lambda,N}-\widetilde h_{\lambda,N}\big\|_{\mathcal{H}_K}^2 &=\|g_N(X_{K_{\cdot}}(\mathbf{Z}_N),(G_N+\lambda NI)^{-1}\mathbf{E}_N)\|_{\mathcal{H}_K}^2\\
&= D^{(1,1)}K(\mathbf{Z}_N,\mathbf{Z}_N)(-\mathbb{J}_N(G_N+\lambda NI)^{-1}\mathbf{E}_N,-\mathbb{J}_N(G_N+\lambda NI)^{-1}\mathbf{E}_N)\\
&:=g_N(\mathbf{E}_N,\Sigma_N\mathbf{E}_N),
% &\leq \frac{1}{(\lambda N)^2}\left |D^{(1,1)}K(\mathbf{Z}_N,\mathbf{Z}_N)(-\mathbb{J}_N\mathbf{E}_N,-\mathbb{J}_N\mathbf{E}_N)\right|
\end{align*} 
where $G_N$ is defined in \eqref{general-Gram} and $\Sigma_N$ is a matrix in $T_{\mathbf{Z}_N}P\times T_{\mathbf{Z}_N}P$ defined by
\begin{align*}
\Sigma_N \mathbf{c}:= (G_N+\lambda NI)^{-1}X_{g_N((G_N+\lambda NI)^{-1}\mathbf{c},X_{K_{\cdot}}(\mathbf{Z}_N))}(\mathbf{Z}_N),\quad \mathbf{c}\in T_{\mathbf{Z}_N}P.  
\end{align*}
Indeed, in light of \eqref{symmetry_of_G_N}, the above definition can be verified since
\begin{align*}
g_N(\mathbf{E}_N,\Sigma_N\mathbf{E}_N)&=g_{N}(\mathbf{E}_N,(G_N+\lambda NI)^{-1}X_{g_N((G_N+\lambda NI)^{-1}\mathbf{E}_N,X_{K_{\cdot}}(\mathbf{Z}_N))}(\mathbf{Z}_N))\\
&=g_{N}((G_N+\lambda NI)^{-1}\mathbf{E}_N,X_{g_N((G_N+\lambda NI)^{-1}\mathbf{E}_N,X_{K_{\cdot}}(\mathbf{Z}_N))}(\mathbf{Z}_N))\\
&=D^{(1,1)}K(\mathbf{Z}_N,\mathbf{Z}_N)(-\mathbb{J}_N(G_N+\lambda NI)^{-1}\mathbf{E}_N,-\mathbb{J}_N(G_N+\lambda NI)^{-1}\mathbf{E}_N).    
\end{align*}
Choose a basis of the vector space $T_{\mathbf{Z}_N}P$ consisting of orthonormal eigenvectors $\{f_i\}_{i=1}^{dN}$ of $G_N$. Note that $\{f_i\}_{i=1}^{dN}$ are also eigenvectors of $\Sigma_N$. Then,
\begin{align*}
\mathrm{Tr}(\Sigma_N)&=\sum_{i=1}^{dN}g_N(f_i,\Sigma_Nf_i)\leq \frac{1}{\lambda^2N^2}\sum_{i=1}^{dN} D^{(1,1)}K(\mathbf{Z}_N,\mathbf{Z}_N)(-\mathbb{J}_Nf_i,-\mathbb{J}_Nf_i)\\
&\leq \frac{1}{\lambda^2N^2}\sum_{i=1}^{dN}\|D^{(1,1)}K\|_{\infty}\cdot g_N(-\mathbb{J}_Nf_i,-\mathbb{J}_Nf_i)\\
&\leq \frac{1}{\lambda^2N^2}\|K\|_{C_b^2}\sum_{i=1}^{dN} g_N(-\mathbb{J}_Nf_i,-\mathbb{J}_Nf_i)\\
&\leq \frac{1}{\lambda^2N^2}\|K\|_{C_b^2}\sum_{i=1}^{dN}\left( \sum_{j=1}^N g(-Jf^{(j)}_i,-Jf^{(j)}_i)\right )\\
&\leq \frac{1}{\lambda^2N^2}\|K\|_{C_b^2}\sum_{i=1}^{dN}\left( \sum_{j=1}^NC^2\cdot g(f^{(j)}_i,f^{(j)}_i)\right )\\
&= \frac{C^2}{\lambda^2N^2} \|K\|_{C_b^2}\sum_{i=1}^{dN}g_N(f_i,f_i)= \frac{dC^2}{\lambda^2N} \|K\|_{C_b^2}=\frac{\kappa^2C^2d}{2\lambda^2N},
\end{align*}
with $\kappa^2=2\|K\|_{C_b^{2}}$ as introduced in the statement of Proposition \ref{Par-Rep}. The rest of the proof follows from the same argument as in \cite[Appendix C]{hu2024structure}.
\end{proof}

\section*{Acknowledgments}
The authors thank Lyudmila Grigoryeva for helpful discussions and remarks and acknowledge partial financial support from the School of Physical and Mathematical Sciences of the Nanyang Technological University. DY is funded by a Nanyang President's Graduate Scholarship of Nanyang Technological University.

\footnotesize
\addcontentsline{toc}{section}{References}
\bibliographystyle{wmaainf}
\bibliography{Refs}
\end{document}